\documentclass[11pt,reqno]{amsart}

\usepackage{fullpage}

\usepackage[margin=1in]{geometry}
\usepackage{amsmath}
\usepackage{amsfonts}
\usepackage{amssymb}
\usepackage{amsthm,mathtools,mathabx,accents,}
\usepackage{newlfont}
\usepackage{graphicx}
\usepackage{xcolor,enumerate,mathrsfs, scalerel}
\newtheorem{thm}{Theorem}[section]
\newtheorem{cor}[thm]{Corollary}
\newtheorem{lem}[thm]{Lemma}
\newtheorem{prop}[thm]{Proposition}

\theoremstyle{definition}
\newtheorem{defn}[thm]{Definition}
\newtheorem{rem}[thm]{Remark}
\numberwithin{equation}{section}

\newcommand{\norm}[1]{\Vert#1\Vert}
\newcommand{\Norm}[1]{\left\Vert#1\right\Vert}
\newcommand{\na}{\nabla}
\newcommand{\pa}{\partial}
\newcommand{\lec}{\lesssim}
\newcommand{\td}{\tilde}

\renewcommand{\div}{\operatorname{div}}

\newcommand\al{\alpha}
\newcommand\be{\beta}
\newcommand\de{\delta}
\newcommand\De{\Delta}
\newcommand{\ga}{\gamma}
\newcommand\Ga{\Gamma}

\newcommand\e {\varepsilon}
\newcommand\ka{\kappa}
\newcommand{\la}{\lambda}
\newcommand{\La}{\Lambda}
\newcommand\ph{\varphi}
\newcommand{\ta}{\tau}
\renewcommand{\th}{\theta}

\newcommand{\T}{\mathbb{T}}
\newcommand{\R}{\mathbb{R}}
\newcommand{\Z}{\mathbb{Z}}
\newcommand{\N}{\mathbb{N}}

\newcommand{\cF}{\mathcal{F}}

\newcommand{\cT}{\mathcal{T}}

\newcommand{\cP}{\mathcal{P}}
\newcommand{\cR}{\mathcal{R}}

\newcommand{\crF}{\mathscr{F}}

\newcommand{\dis}{\operatorname{dist}}

\newcommand{\supp}{\operatorname{supp}}
\newcommand{\tr} {\mathop{\mathrm{tr}}}
\newcommand{\I}{\textrm{Id}}
\newcommand{\tri}{\triangle}

\newcommand{\idv}[1]{\mathcal{R}\left( #1\right)}

\newcommand{\as}[1]{\accentset{*}#1}
\renewcommand{\dot}[1]{\accentset{\circ}#1}
\newcommand{\asR}{\accentset{*}{R}}

\renewcommand{\b}[1]{\underline{#1}}

\DeclarePairedDelimiter{\ceil}{\lceil}{\rceil}

\usepackage[usestackEOL]{stackengine}
\def\dint{\,\ThisStyle{\ensurestackMath{%
  \stackinset{c}{.2\LMpt}{c}{.5\LMpt}{\SavedStyle-}{\SavedStyle\phantom{\int}}}%
  \setbox0=\hbox{$\SavedStyle\int\,$}\kern-\wd0}\int}

\begin{document}

\title{On Non-uniqueness of H\"{o}lder continuous globally dissipative Euler flows}

\author{Camillo De Lellis, Hyunju Kwon}
%\date{\today}
\address{\parbox{\linewidth}{
Camillo De Lellis \\
School of Mathematics, Institute for Advanced Study and Universit\"{a}t Z\"{u}rich\\
1 Einstein Dr., Princeton NJ 08540, USA\\
E-mail address: camillo.delellis@math.ias.edu \\
\\
Hyunju Kwon \\
School of Mathematics, Institute for Advanced Study\\
1 Einstein Dr., Princeton NJ 08540, USA\\
E-mail address: hkwon@math.ias.edu
 }
} 

%\subjclass[2010]{Primary: . Secondary: }
% \keywords{} 
\begin{abstract} 
We show that for any $\al<\frac 17$ there exist $\al$-H\"older continuous weak solutions of the three-dimensional incompressible Euler equation, which satisfy the local energy inequality and strictly dissipate the total kinetic energy. The proof relies on the convex integration scheme and the main building blocks of the solution are various Mikado flows with disjoint supports in space and time.
\end{abstract}

\maketitle

\section{Introduction} 
In this work, we consider the Cauchy problem for the incompressible Euler equations on the spatially periodic domain $[0,T] \times \T^3$,
\begin{equation}\label{eqn.E}
\begin{cases}
\pa_t v + \na \cdot (v\otimes v) + \na p = 0\\
\na \cdot v =0.
\end{cases}
\end{equation}
where $\T^3=[-\pi,\pi]^3$ and $0<T<\infty$. The Euler equations describe the motion of an ideal volume-preserving fluid: $v:[0,T] \times \T^3\to \R^3$ represents the velocity of the fluid and $p:[0,T] \times \T^3\to \R$ the pressure.

A distributional solution of \eqref{eqn.E} is a solenoidal vector field $v\in L^2 ([0, T]\times \T^3; \R^3)$ for which the first equation (the momentum equation) holds distributionally
(i.e. the distributional curl of $\pa_t v + \na \cdot (v\otimes v)$ vanishes). The pressure is the unique (up to a time-dependent constant) solution of $\Delta p = \sum_{ij} \partial^2_{ij} (v_iv_j)$.
A {\it globally dissipative} Euler flow is a distributional solution which belongs $L^3([0,T]\times \T^3)$ and satisfies additionally the {\it local energy inequality} 
\begin{equation}\label{eqn.EI}
\pa_t\left( \frac {|v|^2}2\right) + \na \cdot \left(\left( \frac {|v|^2}2 + p\right) v \right) \le 0
\end{equation}
in the sense of distributions (note that, by the classical Calderon-Zygmund inequality, $p\in L^{3/2} ([0,T]\times \T^3)$ and hence $p|v|^2$ is well defined). Integrating the latter inequality in space, we derive that the global kinetic energy of the solution is nondecreasing:
\begin{equation}\label{eqn.EI2}
\frac{d}{dt} \int \frac{|v|^2}{2} (t,x)\, dx \leq 0\, .
\end{equation}
In order to motivate the energy inequality, we recall the well-known fact that smooth solutions satisfy the {\it local energy equality}, namely
\begin{align}\label{LEE}
\pa_t\left( \frac {|v|^2}2\right) + \na \cdot \left(\left( \frac {|v|^2}2 + p\right) v \right) = 0\, ,
\end{align}
which can be derived by scalar multiplying the momentum equation and making some standard calculus manipulations. Correspondingly, smooth solutions preserve the total kinetic energy, namely
\begin{equation}\label{eqn.EE2}
\int \frac{|v|^2}{2} (x,t)\,dx \equiv \mathrm{const}\, .
\end{equation}
Next, consider ``suitable'' weak solutions of the Navier-Stokes equation
\begin{equation}\label{eqn.NS}
\begin{cases}
\pa_t v + \na \cdot (v\otimes v) + \na p = \varepsilon \Delta v\\
\na \cdot v =0\, ,
\end{cases}
\end{equation}
as defined in the celebrated work of Caffarelli, Kohn and Nirenberg \cite{CKN}.
The latter are distributional solutions in $L^\infty ([0, T]; L^2 (\T^3)) \cap L^2 ([0, T]; W^{1,2} (\T^3))$ which satisfy a corresponding local energy inequality
\[
\pa_t\left( \frac {|v|^2}2\right) + \na \cdot \left(\left( \frac {|v|^2}2 + p\right) v \right) \le \varepsilon \left(\Delta \frac{|v|^2}{2} - |Dv|^2\right)\, 
\] 
and hence a corresponding integrated form
\begin{equation}\label{eqn.LH}
\frac{d}{dt} \int \frac{|v|^2}{2}  \leq - \varepsilon \int |Dv|^2\, .  
\end{equation}
Such suitable weak solutions can be proven to exist for any given $L^2$ (divergence-free) initial and, thanks to the theory developed in \cite{CKN}, are regular outside of a compact set with zero Hausdorff one-dimensional measure. If, as $\varepsilon \downarrow 0$, they were to converge strongly in $L^3$ to a solution of the Euler equation, the latter would be globally dissipative. 

After a series of developments in the field, see \cite{DLSz2013,DlSzJEMS,Is2013, Bu2014,BuDLeSz2013,Bu2015,BuDLeIsSz2015,BuDLeSz2016,DaSz2016,IsOh2016}, Isett solved in \cite{Is2016} a famous Conjecture of Lars Onsager in the theory of fully developed turbulence (cf. \cite{On1949}) showing that for every $\alpha < \frac{1}{3}$ there are $\alpha$-H\"older solutions of \eqref{eqn.E} for which \eqref{eqn.EE2} fails. Slightly after, Isett's result was improved in \cite{BDLSV2020} by Buckmaster, the first author, Sz\'ekelyhidi and Vicol, who showed the existence of $\alpha$-\"Holder solutions in the Onsager range for which \eqref{eqn.EI2} holds with a strict inequality. As a rigorous mathematical validation of the classical Kolmogorov's theory of turbulence, it would be rather interesting if one could show that some of these dissipative solutions can be recovered as strong limits of suitable weak solutions $v_k$ of the Navier-Stokes equations with vanishing $\varepsilon_k$, since such a sequence would display anomalous dissipation, i.e.
\[
\liminf_{k\uparrow \infty} \varepsilon_k \int_0^T \int |Dv_k|^2 (t,x)\, dx\, dt > 0
\]
for some finite time $T>0$. However, as observed by Isett in \cite{Is17}, a strong limit of suitable weak solutions of \eqref{eqn.NS} would necessarily satisfy the local energy inequality (at least if the convergence were to be in the $L^3$ topology). This naturally motivates a stronger version of the Onsager conjecture, namely the existence of $\alpha$-H\"older globally dissipative solutions of the Euler equations (i.e. satisfying \eqref{eqn.E} and \eqref{eqn.EI}) for which
\begin{equation}\label{eqn.Onsager}
\int \frac{|v|^2}{2} (T,x)\, dx < \int \frac{|v|^2}{2} (0,x)\, dx \, .
\end{equation}
The first author and L\'aszl\'o Sz\'ekelyhidi produced the first bounded examples in \cite{DLeSz2010}, while Isett in \cite{Is17} has recently provided the first H\"older examples. In this paper we improve upon the regularity obtained by Isett, even though we are still relatively far from the conjectural threshold $\frac{1}{3}$:  

\begin{thm}\label{thm.onsager}
For any $0\leq \beta < \frac17$ there are globally dissipative weak solutions $v$ to the Euler equation \eqref{eqn.E} in $C^\be([0, T]\times \T^3)$ for which \eqref{eqn.Onsager} holds.
\end{thm}

As it is the case of \cite{Is17} (and in fact of any ``convex integration'' arguments starting from \cite{DLSz2012,DLSz2013}) a byproduct of the construction is that neither \eqref{eqn.EI} nor \eqref{LEE} are enough to restore uniqueness of weak solutions. 

\begin{thm}\label{thm}
For any $0\le \be< \frac 17$, we can find infinitely many time-global weak solutions $v$ to the Euler equation \eqref{eqn.E} in $C^\be([0, T]\times \T^3)$ which have zero mean, satisfy the local energy equality \eqref{LEE} and share the same initial data.  
\end{thm}

While the statement above imposes \eqref{LEE}, which is clearly stronger than \eqref{eqn.EI}, it is not difficult to modify our arguments to produce an analogous example of infinitely many distinct globally dissipative weak solutions with the same initial data and such that \eqref{eqn.Onsager} holds. 

\section{Outline of the proof}

We construct globally dissipative Euler flows approximated by sequences of dissipative Euler-Reynolds flows, introduced in \cite{Is17}. 

\begin{defn}[Dissipative Euler-Reynolds flows] A tuple of smooth tensors $(v,p,R,\ka,\ph)$ is a {\it dissipative Euler-Reynolds flow} {with {\em global energy loss} $E (t)$} if 
$\ka = \frac{1}{2} \tr R$ and the tuple solves the Euler-Reynolds system with the Reynold-stress $R + \frac{E}{3}\, \I$ and the relaxed local energy equality:
\begin{equation}\label{app.eq}
\begin{cases}
\pa_t v + \na \cdot (v\otimes v) + \na p = \na \cdot \left(R + \frac{E}{3} \I\right) = \na \cdot R\\
\na \cdot v =0\\
\pa_t \left(\frac 12 |v|^2 \right) + \na \cdot \left(\left(\frac 12|v|^2 + p\right) v\right) =  D_t \left(\ka + \frac{E}{2}\right) + \na \cdot (Rv) + \na \cdot \ph,
\end{cases}
\end{equation}
where the advective derivative $D_t$ is $D_t = \pa_t+ (v\cdot\na)$ {and $(\na \cdot R)_j = \pa_i{R_{ij}}$. }Even though it is not really essential for our arguments, to be consistent with the term dissipative, we assume that $E'\leq 0$. Note moreover that the addition of a constant to $E$. We will therefore impose
\begin{equation}\label{e:E(0)=0}
E (0) =0\, 
\end{equation}
and observe that, as a consequence,
\begin{equation}\label{e:deficit}
0\geq E (t) = \frac{1}{2} \int |v|^2 (x,t)\, dx - \frac{1}{2} \int |v|^2 (x,0)\, dx\, 
\end{equation}
(which thus justifies the term {\em energy loss}).
\end{defn}
It is obvious that when $R$, $\ka$, $\ph$ are all zero, $(v,p)$ becomes a globally dissipative Euler flow (and the requirement $E'\leq 0$ is used only in this simple conclusion). Denoting an average in space of a quantity $w$ by $\overline{w}$, for a given globally dissipative Euler flow $(v,p)$, the averaged pair $(\bar v, \bar p)$ becomes a dissipative Euler-Reynolds flow with Reynolds stress ${R + \frac E3 \I}= \bar v \otimes \bar v- \overline{v\otimes v}$, unsolved flux density ${\ka+\frac E2 = \frac12 |\bar v|^2- \frac12\overline{|v|^2}= \frac 12 \tr(R+\frac E3\I)}$, unsolved flux current $\ph = \left(\frac 12 |\bar v|^2 +\bar p\right) \bar v - \overline{\left(\frac12 |v|^2 + p\right) v} -\left( \frac 12|\bar v|^2 - \frac12 \overline{|v|^2}\right)\bar v -R\bar v$. This motivates the relation $\ka = \frac 12 \tr(R)$ on the approximate solutions.

\subsection{Induction scheme}
The proof of the main theorem is based on an iterative procedure, similarly to all the literature which came out after the works \cite{DLSz2013,DlSzJEMS}, which introduced the ``Euler-Reynolds flows'', namely the system of PDEs consisting of the first two equations in \eqref{app.eq}. The idea to handle the local energy inequality by adding the third equation and the unknown $\ph$ is instead a notable contribution of \cite{Is17}. In the inductive procedure we assume to have a tuple $(v_q, p_q, R_q, \ka_q, \ph_q)$ solving \eqref{app.eq} for which the ``error $(R_q, \ka_q,\ph_q)$'' is suitably small. At the step $q+1$ we aim at finding a new dissipative Euler-Reynolds flow with reduced error is substantially reduced compared to that of step $q$. This is accomplished by adding a suitable correction $(w_{q+1},q_{q+1})$ to the velocity and pressure $(v_q, p_q)$, namely defining $ v_{q+1} = v_q+w_{q+1}$ and $p_{q+1} = p_q +q_{q+1}$ so that the new error $(R_{q+1},\ka_{q+1},\ph_{q+1})$ (which roughly speaking is determined by the equations) is sensibly smaller than $(R_q,\ka_q,\ph_q)$. The precise statement is given in Proposition \ref{ind.hyp}.

First of all for $q\in \N$ (where we use the convention that $0\in \N$) we introduce the frequency $\la_q$ and the amplitude $\de_q^\frac 12$ of the velocity $v_q$, which have the form
\begin{align*}
\la_q = \ceil{\la_0 ^{(b^q)}}, \quad \de_q = \la_q^{-2\al}\, ,
\end{align*}
where $\alpha$ is a positive parameter smaller than $1$ and $b$ and $\lambda_0$ are real parameters larger than $1$ (however, while $b$ will be typically chosen close to $1$, $\lambda_0$ will be typically chosen very large). In particular $\delta_q^{\frac{1}{2}} \lambda_q$ is a monotone increasing sequence.
 
 In the induction hypothesis, we will assume several estimates on the tuple $(v_q, p_q, R_q, \ka_q, \ph_q)$. For technical reasons, the domains of definition of the tuples is changing at each step and it is given by $[-\tau_{q-1}, T+\tau_{q-1}]\times \mathbb T^3$, where $\tau_{-1} = \infty$ and for $q\geq 0$ the parameter $\tau_q$ is defined by 
\[
\tau_q = \left(C_0 M \la_q^\frac12\la_{q+1}^\frac12 \de_q^\frac14\de_{q+1}^\frac14 \right)^{-1}
\] 
for some geometric constant $C_0$ and $M$ (which will specified later in \eqref{mu.tau} and Proposition \ref{ind.hyp}, respectively). Note the important fact that $\tau_q$ is decreasing in $q$. In order to shorten our formulas, it is convenient to introduce the following notation:
\begin{itemize}
\item $\cal{I} + \sigma$ is the concentric enlarged interval $(a-\sigma, b+\sigma)$ when $\cal{I} = [a,b]$;
\item  $\norm{F_q}_N$ is the $C^0_t C^N_x$ norm of $F_q$ on its domain of definition, namely 
\[
\norm{F_q}_N := \norm{F_q}_{C^0([0, T]+ \tau_{q-1};C^N(\T^3))}\, .
\]
\end{itemize}
We are now ready to detail the inductive estimates:
\begin{align}
\norm{v_q}_0 \leq 1-\de_q^\frac12, \quad \norm{v_q}_N &\leq M\la_q^N \de_q^\frac 12, \quad \norm{p_q}_N \leq \la_q^N \de_q, \quad N=1,2, \label{est.vp}
\end{align}
and
\begin{align}
&\qquad\qquad\ka_q = \frac 12 \tr(R_q),\label{ka} \\
&\qquad\qquad\qquad\quad\;\norm{D_t p_q}_{N-1}\leq \delta_q^{\frac{3}{2}} \lambda_q^N,\label{e:ad-pressure}\\
\norm{R_q}_N &\leq \la_q^{N-3\ga}\de_{q+1}, \quad \norm{D_t R_q}_{N-1} \leq \la_q^{N-3\ga} \de_q^\frac 12 \de_{q+1}, \qquad N = 0,1,2\label{est.R}\\
\norm{\ph_q}_N &\leq \la_q^{N-3\ga} \de_{q+1}^\frac 32, \quad
\norm{D_t \ph_q}_{N-1} \leq \la_q^{N-3\ga}\de_q^\frac12 \de_{q+1}^\frac 32, \qquad N = 0,1,2
 \label{est.ph}
\end{align}
where $D_t = \pa_t + v_q\cdot \na $ and $\ga = (b-1)^2$.

\begin{rem}\label{r:estimates_material_derivative}
Throughout the rest of the paper we will typically estimate $\|F \|_N$ and $\|D_t F\|_{N-1}$ for several functions, vectors and tensors $F$.  However, in a writing like \eqref{est.R} and \eqref{est.ph}, for $N=0$ we are {\em not claiming any negative Sobolev estimate} on $D_t F$: the reader should just consider the advective derivative estimate to be an empty statement when $N=0$. The reason for this convention is just to make the notation easier, as we do not have to state in a separate line the estimate for $\|F\|_0$ in many future statements. 
\end{rem} 

It seems natural to impose that $\frac{E}{3} \I$ has a size comparable to $R_0$ and require
\begin{align*}
\|E\|_0 &\leq \delta_1\\
\|E'\|_0 &\leq { \delta_0^{\frac12}\delta_1}\, ,
\end{align*}
{where we abuse notation and write $\norm{F}_0 =\norm{F}_{C^0(\R)}$ for a time function $F$. As we already discussed the dissipative Euler-Reynolds system is invariant under addition to $E$ of a constant and we adopt the normalization condition $E(0)=0$. We will therefore assume
\begin{equation}\label{e:assumption_on_E}
E(0) =0\, , \quad E' \leq 0 \quad \mbox{and} \quad {\|E^{(n)}\|_0 \leq \de_0^\frac n2\delta_1\, \quad \forall n=0,1.}
\end{equation}
Under this setting, the core inductive proposition is given as follows. 
\begin{prop}[Inductive proposition]\label{ind.hyp} There exists a geometric constant $M{>1}$ and functions $\bar{b} (\alpha) >1$ and $\Lambda_0 (\alpha, b,{M}) >0$ such that the following property holds. Let $\alpha \in (0, \frac{1}{7})$, $b \in (1, \bar b (\alpha))$ and $\la_0\geq \La_0 (\alpha , b,{M})$ and assume that a tuple of tensors $(v_q,p_q,R_q,\ka_q,\ph_q)$ is a dissipative Euler-Reynolds flow defined on the time interval $[0,T]+ \tau_{q-1}$ satisfying \eqref{est.vp}-\eqref{est.ph} for an energy loss $E(t)$ satisfying \eqref{e:assumption_on_E}. Then, we can find a corrected dissipative Euler-Reynolds flow $({v}_{q+1},  p_{q+1},  R_{q+1}, \ka_{q+1}, \ph_{q+1})$ on the time interval $[0,T]+\tau_q$ for the same energy loss $E(t)$ which satisfies \eqref{est.vp}-\eqref{est.ph} for $q+1$ and 
\begin{align}\label{cauchy}
\norm{v_{q+1}-v_q}_0 + \frac 1{\la_{q+1}}\norm{v_{q+1}-v_q}_1 \leq  M \de_{q+1}^\frac 12.  
\end{align}
\end{prop}

While the latter proposition would be enough to prove Theorem \ref{thm.onsager}, we will indeed need a technical refinement in order to show Theorem \ref{thm}. We have decided to state such refinement separately in order to improve the readability of our paper. In its statement we will need the following convention:
\begin{itemize}
\item given a function $f$ on $[0, T]\times \T^3$, $\supp_t (f)$ will denote its temporal support, namely
\[
\supp_t (f) := \{t: \exists \ x \;\mbox{with}\; f (x) \neq 0\}\, .
\] 
\item given an open interval $\cal{I} = [a,b]$, $|\cal{I}|$ will denote its length $(b-a)$ and $\cal{I} + \sigma$ will denote the concentric enlarged interval $(a-\sigma, b+\sigma)$. 
\end{itemize}

\begin{prop}[Bifurcating inductive proposition]\label{p:ind_technical}
Let the geometric constant $M {>1}$, the functions  $\bar b$, $\Lambda_0$, the parameters $\alpha$, $b$, $\lambda_0$ and the tuple $(v_q,p_q,R_q,\ka_q,\ph_q)$ be as in the statement of Proposition \ref{ind.hyp}. For any time interval $\cal{I}\subset (0, T)$ with $|\cal{I}|\geq 3\tau_q$
%\[
%|\cal{I}|\geq 3\tau_q = 3\left(C_0 M \la_q^\frac12\la_{q+1}^\frac12 \de_q^\frac14\de_{q+1}^\frac14 \right)^{-1}
%\] 
%(where $C_0$ is a geometric constant) 
we can produce a first tuple $({v}_{q+1},  p_{q+1},  R_{q+1}, \ka_{q+1}, \ph_{q+1})$ and a second one $(\td v_{q+1},  \td p_{q+1}, \td R_{q+1}, \td\ka_{q+1}, \td\ph_{q+1})$ which share the same initial data, satisfy the same conclusions of Proposition \ref{ind.hyp} and additionally
\begin{align}\label{e:distance_and_support}
\norm{v_{q+1}-\td v_{q+1}}_{C^0([0, T];L^2(\T^3)) }\geq \de_{q+1}^\frac 12, \quad
\supp_t(v_{q+1}-\td v_{q+1}) \subset \cal{I}. 
\end{align}
Furthermore, if we are given two tuples $(v_q,p_q,R_q,\ka_q,\ph_q)$ and $(\td v_q, \td p_q, \td R_q, \td\ka_q, \td\ph_q)$ satisfying \eqref{est.vp}-\eqref{est.ph} and
\[
\supp_t(v_q-\td v_q, p_q-\td p_q, R_q-\td R_q, \ka_q-\td \ka_q, \ph_q-\td \ph_q)\subset \cal{J} 
\]
for some interval $\cal{J}\subset (0, T)$, we can exhibit corrected counterparts $({v}_{q+1},  p_{q+1},  R_{q+1}, \ka_{q+1}, \ph_{q+1})$ and $(\td v_{q+1},  \td p_{q+1}, \td R_{q+1}, \td\ka_{q+1}, \td\ph_{q+1})$ again satisfying the same conclusions of Proposition \ref{ind.hyp} together with the following control on the support of their difference:
\begin{equation}\label{e:second_support}
\supp_t(v_{q+1}-\td v_{q+1}, p_{q+1}-\td p_{q+1}, R_{q+1}-\td R_{q+1}, \ka_{q+1}-\td \ka_{q+1}, \ph_{q+1}-\td \ph_{q+1})\subset \cal{J} + (\la_q\de_q^\frac12)^{-1}.
\end{equation}
\end{prop}

\subsection{ Proof of Theorem \ref{thm}} In this argument we assume $T\geq 20$.
Fix $\beta< \frac 17$ and $\al\in (\be, \frac  17)$. {Then, choose $b$ and $\la_0$ in the range suggested in Proposition \ref{ind.hyp}.} We take $E\equiv 0$ and we start with 
\[
(v_0,p_0,R_0,\ka_0, \ph_0)= (0,0,0,0,0)\, .
\] 
It is easy to see that it solves \eqref{app.eq} and satisfies \eqref{est.vp}-\eqref{est.ph}. Now, the conclusion of Proposition \ref{ind.hyp} would be trivial in this case, since we could simply define $(v_1, p_1, R_1, \ka_1, \varphi_1)$ to be identically $0$. Nonetheless, consider any sequence of solutions $(v_q, p_q, R_q,\ka_q, \ph_q)$ to \eqref{app.eq} and satisfying \eqref{est.vp}-\eqref{est.ph} {and \eqref{cauchy}}. Since the sequence $\{v_q\}$ also satisfies \eqref{cauchy}, it is Cauchy in $C^0([0, T];C^{\al}(\T^3))$. Indeed, for any $q<q'$, we have
the following estimates\footnote{Here and in the rest of the note, given two quantities $A_q$ and $B_q$ depending on the induction parameter $q$ we will use the notation $A\lec B$ meaning that $A\leq CB$ for some constant $C$ which is independent of $q$. In some situations we will need to be more specific and then we will explicitely the depndence of $C$ on the various parameters involved in our arguments.}
\begin{align*}
\norm{v_{q'} - v_q}_{C^0([0, T];C^\al(\T^3))}
&\leq \sum_{l=1}^{q'-q}\norm{v_{q+l} - v_{q+l-1}}_{C^0([0, T];C^\be(\T^3))}\\
&\lec \sum_{l=1}^{q'-q}
\norm{v_{q+l} - v_{q+l-1}}_{0}^{1-\be}
\norm{v_{q+l} - v_{q+l-1}}_{1}^{\be}\\
&\lec \sum_{l=1}^{q'-q} 
\la_{q+1}^{\be} \de_{q+1}^{\frac12} = \sum_{l=1}^{q'-q} 
\la_{q+1}^{\be-\al} \to 0, 
\end{align*}
as $q$ goes to infinity because of $\be-\al<0$. Therefore, we obtain its limit $v$ in $ C^0([0, T];C^\al( \T^3))$. Also, we note that $p_q$ solves $\De p_q = \div\div (R_q - v_q\otimes v_q)$. By using the convergence of $v_q$ in $C^0([0, T];C^\alpha ( \T^3))$ to $v$ (and that of $R_q$ to $0$ in $C^0([0, T]; C^\alpha ( \T^3))$: note that an analogous interpolation argument can be used for $R_q$ as well), we can obtain a mean-zero pressure $p$ as the limit of $p_q$ in $C^0([0, T]; C^\alpha ( \T^3))$, by Schauder estimates. Since $(R_q, \ka_q, \ph_q)$ converges to $0$ in $C^0([0, T]\times \T^3)$, the limit $(v,p)$ solves the Euler equation \eqref{eqn.E} and satisfies the local energy equality \eqref{LEE} in the distributional distribution. Estimating 
\[
\|\partial_t v_q\|_0 \leq \|v_q\|_0 \|D v_q\|_0 + \|Dp_q\|_0 + \|D R_q\|_0\, ,
\] 
the time regularity of $v$ can be concluded with an analogous interpolation argument. Alternatively it follows from the general results of \cite{Isett-time} (see also \cite{CoDR} for a different argument). Hence $v\in C^\alpha ([0, T]\times \T^3)$. We mention in passing that the pressure can be shown to belong to $C^{2\alpha}$ (in time and space): this can be concluded again by interpolation as done above or it can be inferred from the general results of \cite{Isett-time,CoDR}. 

On the other hand, fix $\bar q\in \N$ satisfying $b^{\bar  q}\geq \bar q$. At the $\bar q$th step using Proposition \ref{p:ind_technical} we can produce two distinct tuples, one which we keep denoting as above and the other which we denote by $(\td v_q, \td p_q, \td R_q, \td\ka_q, \td \ph_q)$ and satisfies \eqref{e:distance_and_support}, namely
\begin{align*}
\norm{\td v_{\bar q} - v_{\bar q}}_{C^0([0, T];L^2(\T^3))} \geq  \de_q^\frac 12, \quad \supp_t (v_q-\td v_q) \subset \cal I\, ,  
\end{align*}
with $\cal I  = (10, 10+3\tau_{\bar q-1})$.
Applying now the Proposition \ref{ind.hyp} iteratively, we can build a new sequence $(\td v_q, \td p_q, \td R_q, \td\ka_q, \td \ph_q)$ of approximate solutions which satisfy \eqref{est.vp}-\eqref{cauchy} and \eqref{e:second_support}, inductively. Arguing as above, this second sequence converges to a solution $(\td v, \td p)$ to the Euler equation \eqref{eqn.E} satisfying \eqref{LEE}. Indeed, $\td v\in C^\beta ([0, T]\times \T^3)$. We remark that for any $q\geq \bar q$,
\[
\supp_t(v_{q} -\td v_{q}) \subset \cal I + \sum_{q=\bar q}^\infty (\la_{q}\de_{q}^\frac12)^{-1} \subset[9,T],
\]
(by adjusting $\la_0$ to be even larger than chosen above, if necessary), and hence $\td v_{ {q}}$ shares initial data with $v_{{q}}$ { for all $q$. As a result, two solutions $\td v_q$ and $v_q$ have the same initial data.} However, the new solution $\td v$ differs from $v$ because
\begin{align*}
\norm{v - \td v}_{C^0([0, T];L^2(\T^3))}
&\geq \norm{v_{\bar q} -\td v_{\bar q}}_{C^0([0,T];L^2(\T^3))}
-\sum_{q=\bar q}^\infty \norm{v_{q+1}-v_q - (\td v_{q+1} - \td v_q)}_{C^0([0,T];L^2(\T^3))} \\
&\geq \norm{v_{\bar q} - \td v_{\bar q}}_{C^0([0,T];L^2(\T^3))}
-(2\pi)^\frac32 \sum_{q=\bar q+1}^\infty (\norm{v_{q+1}-v_q}_0 + \norm{\td v_{q+1}-\td v_q}_0)\\
&\geq   \de_{\bar q}^\frac 12  - 2(2\pi)^\frac32M \sum_{q=\bar q}^\infty\de_{q+1}^\frac12>0.
\end{align*} 
The last inequality follows from adjusting $\la_0$ to a larger one if necessary. By changing the choice of time interval $\cal I$ and the choice of $\bar{q}$, we can easily generate infinitely many solutions. 

\subsection{Proof of Theorem \ref{thm.onsager}} Fix $\beta < \frac{1}{7}$ and as above choose $\alpha \in (\beta, \frac{1}{7})$. In order to prove Theorem \ref{thm.onsager} it suffices to produce a nonzero $E$ which satisfies \eqref{e:assumption_on_E} and a starting tuple $(v_0, p_0, R_0, \kappa_0, \varphi_0)$ which satisfies \eqref{app.eq}-\eqref{est.ph}.
%\begin{itemize}
%\item $\kappa_0 = \frac{1}{2} \tr R_0$ and \eqref{app.eq};
%\item \eqref{est.vp}, \eqref{est.R} and \eqref{est.ph}.
%\end{itemize} 
In fact, arguing as in the previous section we can use inductively Proposition \ref{ind.hyp} to produce a sequence $(v_q, p_q)$ which is converging uniformly to a $C^\alpha$ solution $(v,p)$ of the Euler equations on [0,T] for which, additionally, the identity
\[
\partial_t \frac{|v|^2}{2} + \nabla \cdot \left(\left(\frac{|v|^2}{2}+p\right) v\right) = E'\, 
\]
holds.
Since any time-dependent function $\chi \in C^\infty_c ((0,T))$ would be an admissible test function, we conclude
\[
- \int_0^\infty \partial_t \chi \int_{\mathbb T^3} \frac{|v|^2}{2} (t,x)\, dx\, dt = \int_0^\infty E'(t) \chi (t)\, dt\, . 
\] 
Given that $E$ is $C^1$, the latter implies that the total kinetic energy is a $C^1$ function and it in facts coincides with $E$ up to a constant, namely
\[
\int_{\T^3} \frac{|v|^2}{2} (T,x)\, dx - \int_{\T^3} \frac{|v|^2}{2} (0,x)\, dx = E(T) - E(0). 
\] 
On the other hand $E(0)=0$, $E'\leq 0$ and $E$ is not identically $0$. In particular ${E(T) - E(0)} <0$. 

We are thus left with the task of finding a suitable $E$ and a starting tuple. $E$ will be assumed to be smooth. We fix an integer parameter $\bar\lambda>0$ (which will be chosen appropriately later) and define 
\begin{align*}
p_0 &=\kappa_0 = 0\, ,\\
v_0 &= (1-2 \delta_0^\frac{1}{2} + E(t))^{\frac{1}{2}} (\cos \bar \lambda x_3, \sin \bar \lambda x_3, 0)\, ,\\
\ph_0&=0
\end{align*}
and 
\[
R_0 := \bar \lambda^{-1} \frac{d}{dt} (1-2 \delta_0^{\frac{1}{2}} + E(t))^{\frac{1}{2}}
\left(\begin{array}{lll}
0 & 0 & \sin \bar\lambda x_3\\
0 & 0 & - \cos \bar\lambda x_3\\
\sin \bar\lambda x_3 & - \cos \bar\lambda x_3 & 0
\end{array}\right)\, .
\]
Note that, by assuming $\lambda_0$ large enough, $0 \leq \delta_0 - {\de_1} \leq 2\delta_0^\frac{1}{2} - {E(t)} \leq 3\delta_0^\frac{1}{2} \leq \frac{1}{2}$ 
and thus $v_0$ and $R_0$ are well defined and smooth, since
\begin{equation}
1 - 2\delta_0^{\frac{1}{2}} + E(t) \geq \frac{1}{2}\, .
\end{equation}
Moreover, it is easy to check that the tuple satisfies \eqref{app.eq}, \eqref{ka}, and \eqref{est.ph}. Our task is therefore to show that we can choose a nontrivial $E$ and a $\bar\lambda$ so to satisfy all the estimates \eqref{est.vp} and \eqref{est.R}. 

We start recalling that $E\leq 0$ and estimating
\[
\|v_0\|_N \leq (1-2\delta_0^\frac{1}{2})^{\frac{1}{2}} \bar\lambda^N\, ,
\]
which implies that \eqref{est.vp} is satisfied as soon as
\begin{equation}\label{e:barlambda_1}
\bar\lambda \leq  \lambda_0\delta_0^{\frac{1}{2}}\, .
\end{equation}
As for \eqref{est.R} we estimate
\begin{align*}
\|R_0\|_N  &\leq C \bar \lambda^{N-1} \|E'\|_0\\
 \|D_t R_0\|_N &=\|\partial_t R_0\|_N \leq C \bar \lambda^{N-1} (\|E''\|_0 + \|E'\|_0^2)\, ,
\end{align*}
where $C$ is a geometric constant.
Considering that by \eqref{e:barlambda_1} we already have $\bar \lambda \leq \lambda_0$ and that $\|E'\|_0 \leq {\de_0^\frac12}\delta_1$ by assumption (and hence
$\|E'\|_0^2 \leq {\de_0} \delta_1^2 \leq {\delta_0\delta_1}$), it suffices to impose
$\|E''\|_0 \leq {\de_0}\delta_1$ and 
\begin{equation}\label{e:barlambda_2}
\bar \lambda\geq C \lambda_0^{3\gamma} {\de_0^\frac12}\, .
\end{equation} 
Since the existence of a nontrivial smooth function $E$ with $E (0) =0$, $E'\leq 0$, and {$\|E^{(n)}\|_0 \leq\de_0^\frac n2 \delta_1$ for $n=0,1,2$} is obvious (choose for example $E(t) = -\de_1(1-\exp(-\de_0^\frac12 t))$), we just need to check that the requirements \eqref{e:barlambda_1} are mutually compatible for some choice of the parameters satisfying the assumptions of Proposition \ref{ind.hyp}: taking into account that $\bar\lambda$ must be an integer, the requirement amounts to the inequality ${\delta_0^{\frac{1}{2}}( \lambda_0 - C \lambda_0^{3\gamma})} \geq 1$. 
Recall that %$\delta_0^{\frac{1}{2}} \lambda_0 = \lambda_0^{1-\alpha}$ and 
$3\gamma = 3(b-1)^2$ and 
%$\alpha<\frac{1}{7}$ and that 
the restrictions on $b$ given by Proposition \ref{ind.hyp} certainly allows us to choose $b$ so that ${3}(b-1)^2<\frac{1}{2}$. The compatibility is then satisfied if {$\lambda_0 >C \lambda_0^{\frac{1}{2}} +1$.} Since $C$ is a geometric constant, we just need a sufficiently large $\lambda_0$, which is again a restriction compatible with the requirements of Proposition \ref{ind.hyp}.

\section{Construction of the velocity correction}\label{sec:correction.const}

In this section we detail the construction of the correction $w:= v_{q+1}-v_q$. As in the literature  which started from the paper \cite{DLSz2013}, the perturbation $w$ is, in first approximation, obtained from a family of highly oscillatory stationary solution of the incompressible Euler equation, which are modulated by the errors $R_q$ and $\varphi_q$ and transported along the ``course grain flow'' of the vector field $v_q$. There are several choices of stationary solutions that one could use. In  this paper our choice falls on what has proved to be the most efficient ones found so far, called Mikado flows and first introduced in \cite{DaSz2016}. In order to define them, consider a function $\vartheta$ on $\mathbb R^2$ and let $\bar U (x) = e_3 \vartheta (x_1, x_2)$, where $e_3 = (0,0,1)$. Then it can be readily checked that $\bar U$ is a stationary solution. We can now apply a stretching factor $s>0$, a general rotation $O$ and a translation by a vector $p$ to define 
\[
U (x) = s O \bar U (O^{-1} (x-p))\, .
\] 
Observe that the periodization of this function defines a stationary solution on $\mathbb T^3$ when $O e_3$ belongs to $a \mathbb Q^3$ for some $a>0$. From now on, with a slight abuse of our terminology, the word Mikado flow will always refer to such periodization. Moreover the vector $f = s O e_3$ will be, without loss of generality assumed to belong to $\mathbb Z^3$ and will be called the {\em direction} of the Mikado flow, while $p$ will be called  its {\em shift}. 

In this paper $f$ will vary in a finite fixed set of directions $\mathcal{F}$ (which in fact has cardinality ${270}$) and for each $f$ we will specify an appropriate choice of $\vartheta$, which will be smooth and compactly supported in a disk $B (0, \frac{d_0}{4})$. The precise choice will be specified later. $\vartheta$ will not depend on the shift $p$ and we will denote by $U_f$ the corresponding Mikado flows when $p=0$. Observe that $U_f = f \psi_f$ for some smooth $\psi_f\in C^\infty_c (\mathbb R^3)$ with $f\cdot \nabla \psi_f =0$. One key point, which is used since the pioneering work \cite{DaSz2016} is the following elementary lemma:

\begin{lem}\label{l:Mikado}
For each $f\in \mathcal{F}$, let $p(f) \in \mathbb R^3$, $\gamma_f \in \mathbb R$ and $\lambda \in \mathbb N\setminus \{0\}$. Assume that the supports of the maps $U_f (\cdot - p (f))$ are pairwise disjoint. Then
\[
\sum_{f\in \mathcal{F}} \gamma_f U_f (\lambda (x-p(f)))
\]
is a stationary solution of the incompressible Euler equations on $\mathbb T^3$. 
\end{lem}

Note that the supports of the functions $U_f (\cdot - p)$ and $\psi_f (\cdot - p)$ are contained in a $\frac{d_0}{4}$-neighborhood of 
\begin{equation}
l_f+p := \left\{x\in \mathbb T^3: \left(x -\sigma f - p\right) \in 2\pi \mathbb Z^3 \quad \mbox{for some $\sigma \in \mathbb R$}\right\}\, .
\end{equation} 
If $d_0$ is sufficiently small, depending on $f$, the latter is a ``thin tube'' winding around the torus a finite number of time (this inspired the authors of \cite{DaSz2016} to call $U_f$ a Mikado flow, inspired by the classical game originating in Hungary). 
In a first approximation we wish to define our perturbation $v_{q+1} - v_q$ as
\[
\sum_{f\in \mathcal{F}} \gamma_f (R_q(t,x), \varphi_q (t,x)) U_f (\lambda (x-p(f))
\]
where the coefficients $\gamma_f$ are appropriately chosen smooth functions (later on called ``weights''), $\lambda$ is a very large parameter and the $p(f)$ are appropriately chosen shifts to ensure the disjoint support condition of Lemma \ref{l:Mikado} (namely that the $\frac{d_0}{4}$-neighborhoods of $l_{f} + p(f)$ are pairwise disjoint). As already pointed out such Ansatz must be corrected and we need to modify the perturbation 
so that it is approximately advected by the velocity $v_q$. Note that on large time-scales the flow of the velocity $v_q$ does not satisfy good estimates, while it satisfies good estimates on a sufficiently small scale $\tau_q$. Following \cite{BuDLeIsSz2015} and \cite{Is2013} this issue is solved by introducing a partition of unity in time and ``restarting' the flow at a discretized set of times, roughly spaced according to the parameter $\tau_q$. However, unlike \cite{BuDLeIsSz2015} and \cite{Is2013} one has to face the delicate problem of keeping the supports of the various Mikado flows disjoint. This is done by discretizing the construction in space too, taking advantage of the fact that for sufficiently small space and time scales, the supports of the transported Mikado flows remain roughly straight thin tubes: the argument requires then a subtle combinatorial choice of the ``shifts''. As in \cite{DLSz2013} the introduction of the space-time discretization deteriorates the estimates and accounts for the H\"older threshold $\frac{1}{7}$. This is certainly better than the exponent achieved by Isett in \cite{Is17}, however it comes short of the conjectural exponent $\frac{1}{3}$ precisely because of the several additional errors introduced by the discretization.

We break down the exact description of the perturbation in the following steps:
\begin{itemize}
\item In Section \ref{ss:directions} we describe the choice of the directions $\mathcal{F}$ and the respective functions $\psi_f$;
\item In Section \ref{ss:regularization} we describe an appropriate regularization of the $(v_q, p_q, R_q, \kappa_q, \varphi_q)$ and we introduce the drift of the regularized velocity;
\item In Section \ref{ss:discretization} we describe the main part of the velocity perturbation, which involves the space-time discretization, the combinatorial choice of the shifts and the drifted Mikado flows;
\item In Section \ref{ss:weights} we detail the choice of the weight of each Mikado flow;
\item In Section \ref{ss:correction} we specify a further correction of the main velocity perturbation, which ensures that $v_{q+1}$ is solenoidal. 
\end{itemize}

\subsection{Mikado directions}\label{ss:directions} To determine a set of suitable directions $f$, we recall two geometric lemmas. In the first we denote by $\mathbb{S}$ the subset of $\mathbb R^{3\times 3}$ of all symmetric matrices, let $\I$ be the identity matrix and set $|K|_\infty := |(k_{lm})|_\infty = \max_{l,m} |k_{lm}|$ for $K\in \mathbb R^{3\times 3}$.
\begin{lem}[Geometric Lemma I]\label{lem:geo1} 
Let $\cF= \{f_i\}_{i=1}^6$ be a set of vectors in $\Z^{3}$ and $C$ a positive constant such that
\begin{equation}\label{con.FIR}
\sum_{i=1}^6 f_i \otimes f_i = C\I, \quad\mbox{and}\quad
\{f_i\otimes f_i\}_{i=1}^6 \text{ forms a basis of }\mathbb{S}\, .
\end{equation}
Then, there exists a positive constant $N_0 = N_0(\cF)$ such that for any $N\leq N_0$, we can find functions $\{\Ga_{f_i}\}_{i=1}^6\subset C^\infty(S_{N}; (0,\infty))$, with domain $S_{N}:= \{\I-K: K \text{ is symmetric, } |K|_\infty\leq N\}$,
satisfying 
\begin{align*}
\I -K = \sum_{i=1}^6 \Ga_{f_i}^2(\I - K) (f_i\otimes f_i), 
\quad\forall (\I-K) \in S_{N}\, .
\end{align*}
\end{lem}

\begin{proof}
Since $f_i\otimes f_i$ is a basis for $\mathbb{S}$, there are unique linear maps $L_i : \mathbb{S} \to \mathbb R$ such that
\[
A = \sum_i L_i (A)\, f_i\otimes f_i \qquad \forall A\in \mathbb{S}\, . 
\]
On the other hand by \eqref{con.FIR} and the uniqueness of such maps, $L_i (\I) = \frac{1}{C}$ for all $i$. Choose now $N_0$ so that $L_i (A)\geq \frac{1}{2C}$ for all $A\in S_{N_0}$. It thus suffices to set $\Ga_{f_i} (A) := \sqrt{L_i (A)}$ for all $A\in S_{N_0}$ to find the desired functions.
\end{proof}

\begin{lem}[Geometric Lemma II] \label{lem:geo2}
Suppose that 
\begin{equation}\label{con.FIph}
\{f_1, f_2, f_3\}\subset \Z^3\setminus\{0\} \text { is an orthogonal frame and } 
f_4=-(f_1+ f_2+f_3).
\end{equation}
Then, for any $N_0>0$, there are affine functions $\{\Ga_{f_k}\}_{1\leq k \leq 4} \subset C^\infty(\cal{V}_{N_0}; [N_0, \infty))$ with domain $\cal{V}_{N_0} := \{u\in \R^3: |u| \leq N_0\}$ such that
\[
u = \sum_{k=1}^4 \Ga_{f_k} (u) f_k  \qquad\forall u \in \cal{V}_{N_0}\, .
\] 
\end{lem}
\begin{proof} Since $\{f_k\}_{i=1}^3$ is orthogonal, any vector $u$ in $\R^3$ can be written as $u=\sum_{i=1}^3 \frac{u\cdot f_k}{|f_k|^2}f_k$. Define
\[
\Ga_{f_k} (u) 
=\begin{cases} 
2N_0+\frac{u\cdot f_k}{|f_k|^2},  &k=1,2,3\\
2N_0					&k=4.
\end{cases}	
\]
Clearly the functions are affine (and thus smooth), whereas a direct computation gives $u = \sum_{k=1}^4 \Ga_{f_k} (u) f_k$. {If $u\in \cal{V}_{N_0}$, using $|u| \leq N_0$ and $|f_k|\geq 1$, we get $|\Ga_{f_k}(u)|\geq 2N_0 -\frac{|u|}{|f_k|} \geq N_0$ when $k=1,2,3$. Therefore, it is obvious to have $|\Ga_{f_k} (u) |\geq N_0$ for any $u\in \cal{V}_{N_0}$.}
\end{proof}

Based on these lemmas, we choose 27 pairwise disjoint families $\cF^j$ indexed by $j\in \mathbb{Z}_3^3$, where each $\cF^j$ consists further of two (disjoint) subfamilies $\cF^{j,R}\cup\cF^{j,\ph} $ with cardinalities $|\cF^{j,R}|=6$ and  $|\cF^{j,\ph}|=4$, chosen so that $\cF^{j,R}$ and $\cF^{j,\ph}$ satisfy \eqref{con.FIR} and \eqref{con.FIph}, respectively. For example, for $j= (0,0,0)$ we can choose
\begin{align*}
\cF^{j,R} = \{(1,\pm 1, 0), (1, 0, \pm 1), (0, 1, \pm1)\}, \quad
\cF^{j,\ph} = \{(1,2, 0), (-2, 1, 0), (0, 0, 1), (1,-3,-1)\}
\end{align*}
and then we can apply $26$ suitable rotations (and rescalings). 
Next, the function $\vartheta$ will be chosen for each $f$ in two different ways, depending on whether $f\in \cF^{j,R}$ or $f\in \cF^{j,\ph}$. Introducing the shorthand notation
$\langle u \rangle = \dint_{\T^3} u (x)\, dx$, we impose the moment conditions
\begin{equation}\label{con.psi2}
\begin{split}
\langle \psi_f \rangle = \langle \psi_f^3 \rangle = 0, \quad \langle \psi_f^2\rangle =1  \qquad \forall f\in \cF^{j,R},\\
\langle \psi_f\rangle = 0, \quad \langle \psi_f^3 \rangle=1
%\gg \int_{\T^3} \psi_f^2(x) dx, 
\qquad \forall f\in \cF^{j,\ph}.
\end{split}
\end{equation} 
The main point is that the Mikado directed along $f\in \cF^{j,R}$ will be used to ``cancel the error $R_q$'', while the ones directed along $f\in \cF^{j, \ph}$ will be used to ``cancel the error $\ph_q$'' and the different moment conditions will play a major role. In both cases we assume also that
\begin{equation}\label{e:eta}
\supp (\psi_f) \subset B \left(l_f, \frac{\eta}{10}\right)
{ := \{x\in \R^3: |x-y|<\textstyle{\frac {\eta}{10}} \text{ for some } y\in l_f\} }
\, , 
\end{equation}
where $\eta$ is a geometric constant which will be specified later, cf. Proposition \ref{p:supports}. 

\subsection{Regularization and drift}\label{ss:regularization} We start by suitably smoothing the tuple $(v_q, p_q, R_q, \kappa_q, \varphi_q)$. To this aim we first introduce the parameters $\ell$ and $\ell_t$, defined by
\[
\ell = \frac 1{\la_q^\frac34\la_{q+1}^\frac14} \left(\frac{\de_{q+1}}{\de_q}\right)^\frac38, \quad
\ell_t = \frac 1{\la_q^{\frac12-3\ga} \la_{q+1}^\frac12 \de_q^\frac14\de_{q+1}^\frac14 }\, .
\]
The space regularizations of $v_q$ and $p_q$ are defined by applying a ``low-pass filter'' which roughly speaking eliminates all the waves larger than $\ell^{-1}$. 
In order to do so we first introduce some suitable notation. First of all, for a function $f$ in the Schwartz space $\cal{S}(\R^3)$, the Fourier transform of $f$ and its inverse on $\R^3$ are denoted by 
\[
\widehat{f} (\xi) = \frac 1{(2\pi)^3}\int_{\R^3} f(x) e^{-ix \cdot \xi} dx,\quad
\widecheck{f} (x) = \int_{\R^3} f(\xi) e^{ix \cdot \xi} d\xi. 
\]   
As usual, we understand the Fourier transform on more general functions as extended by duality to $\mathcal{S}' (\R^3)$. Since practically all the objects considered in this note are functions, vectors and tensors defined on $I\times \mathbb T^3$ for some time domain $I\subset \mathbb R$, regarding them as spatially periodic functions on $I\times \mathbb R^3$, we will consider their Fourier transform as time-dependent elements of $\mathcal{S}' (\R^3)$. We then follow the standard convention on Littlewood-Paley operators. We let $m(\xi)$ be a radial smooth function supported in ${B(0,2)}$ which is identically $1$ on $\overline{B(0,1)}$. For any number $j\in \Z$ and distribution $f$ in $\R^3$, we set
\begin{align*}
\widehat{P_{\leq 2^j} f}(\xi) &:=m\left(\frac{\xi}{2^j}\right) \hat{f}(\xi),\quad
\widehat{P_{> 2^j} f}(\xi) := \left(1-m\left(\frac{\xi}{2^j}\right)\right) \hat{f}(\xi), 
\end{align*}
and for $j\in \Z$
\begin{align*}
\widehat{P_{2^j}f}(\xi) :=\left( m\left(\frac{\xi}{2^j}\right)
-m\left(\frac{\xi}{2^{j-1}}\right)\right) \hat{f}(\xi).
\end{align*}
For a positive real number $S$, we finally let $P_{\leq S}$ equal the operator $P_{\leq 2^J}$ for the largest $J$ such that $2^J\leq S$. 
We are thus ready to introduce the coarse scale velocity $v_\ell$ and pressure $p_\ell$ defined by 
\begin{align}\label{def.vl}
v_\ell = P_{\le \ell^{-1}}v_{{q}}, \quad p_\ell = P_{\le \ell^{-1}} p_{{q}}\, .
\end{align}
Note that, regarding $v$ as a spatially periodic function on $I\times \mathbb T^3$, $P_{\leq \ell^{-1}} v$ can be written as the space convolution of $v$ with the kernel $2^{3J} \widecheck{m} (2^J \cdot)$, which belongs to $\mathcal{S} (\R^3)$. In particular $v_\ell$ is also spatially periodic and will be in fact regarded as a function on $I\times \mathbb T^3$. Similar remarks apply to several other situations in the rest of this note.

The regularization of the errors $R_q, \kappa_q$ and $\varphi_q$ is more laborious and follows the intuition that, while we need to regularize them in time and space, we want such regularization to give good estimates on their advective derivatives along $v_\ell$, for which we introduce the ad hoc notation
\[
D_{t, \ell} := \pa_t + v_\ell \cdot \na\, .
\]
First of all we let $\Phi(\tau,x;t)$ be the forward flow map with the drift velocity $v_\ell$ defined on some time interval $[a,b]$ starting at the initial time $t\in [a,b)$:
\begin{align}\label{e:fflow_first_instance}
\begin{cases}
\pa_\tau \Phi(\tau, x;t)= v_\ell (\tau, \Phi(\tau, x;t))\\
\Phi(t, x;t) = x\, .
\end{cases}
\end{align}

\begin{rem}\label{r:periodicity}
Strictly speaking the map above is defined on $[a,b] \times \mathbb R^3$. Note however that the periodicity of $v_\ell$ implies that $\Phi$ induces a well-defined map from $[a,b]\times \mathbb T^3$ into $\mathbb T^3$. From now on we will implicitly identify both maps.
\end{rem}

We then take a standard mollifier $\rho$ on $\R$, namely a nonnegative smooth bump function satisfying $\norm{\rho}_{L^1(\R)}=1$ and $\supp\rho\subset (-1,1)$. As usual we set $\rho_\delta (s) = \delta^{-1} \rho( \delta^{-1} s)$ for any $\de>0$. We can thus introduce the mollification along the trajectory
\begin{align*}
(\rho_\delta \ast_{\Phi} F) (t,x)
= \int_{\R} F(t+s,\Phi(t+s,x;t))  \rho_\delta (s) ds. 
\end{align*}
(Note that if $F$ and $v_\ell$ are defined on some time interval $[a,b]$, then $\rho_\delta \ast_{\Phi} F$ is defined on $[a,b]-\delta$.)
This mollification can be found in \cite{Is2013}
%{p.129}
and is designed to satisfy
\begin{align*}
D_{t,\ell} (\rho_\delta \ast_{\Phi} F)(t,x)
= \int (D_{t,\ell}F)(t+s, \Phi(t+s, x;t)) \rho_\delta (s) ds
= - \int F(t+s, \Phi(t+s, x;t)) \rho'_\delta (s) ds\, .
\end{align*}
The regularized errors are then given by
\begin{align}\label{def.me}
R_\ell = \rho_{\ell_t} \ast_{\Phi} P_{\leq \ell^{-1}} R_{{q}}, \quad
\ph_\ell = \rho_{\ell_t} \ast_{\Phi} P_{\leq \ell^{-1}}\ph_{{q}}, {\quad \ka_\ell = \tr(R_\ell).}
\end{align}
These errors can be defined on $[0,T]+2\tau_q$ by the choice of sufficiently large $\la_0$. We will need later quite detailed estimates on the difference between the original tuple and the regularized one and on higher derivative of the latter. Such estimates are in fact collected in Section \ref{s:estimates_mollification}.

\subsection{Partition of unity and shifts}\label{ss:discretization}  We first introduce nonnegative smooth functions $\{\chi_n\}_{n\in\Z^3}$ and $\{\th_m\}_{m\in \Z}$ whose sixtth powers give suitable partitions of unity in space $\R^3$ and in time $\R$, respectively:
\[
\sum_{n\in \Z^3} \chi_n^6(x) =1, \quad
\sum_{m\in \Z} \th_m^6(t) =1.   
\]
Here, $\chi_n(x) = \chi_0(x-{2}\pi n)$ where $\chi_0$ is a non-negative smooth function supported in $Q(0, {9/8}\pi)$ satisfying $\chi_0 = 1$ on $\overline{Q(0, {7/8}\pi)}$, where from now on $Q (x, r)$ will denote the cube $\{y: |y-x|_\infty < r\}$ (with $|z|_\infty := \max \{|z_i|\}$). Similarly, $\th_m(t) = \th_0(t-m)$ where $\th_0\in C_c^\infty(\R)$ satisfies $\th_0 = 1$ on $[1/8, 7/8]$ and $\th_0 =0$ on $(-1/8,9/8)^c$. Then, we divide the integer lattice $\Z^3$ into 27 equivalent families $[j]$ with $j \in \mathbb{Z}_3^3$ via the usual equivalence relation
\[
n=(n_1,n_2,n_3) \sim \td n = (\td n_1,\td n_2,\td n_3) 
\iff n_i \equiv \td n_i\; \mod 3\quad \text{for all }i=1,2,3.
\]
We use these classes to define the set of indices 
\[
\mathscr{I} := \{(m,n, f): (m,n)\in \mathbb Z\times \mathbb Z^3 \quad\mbox{and}\quad f\in \mathcal{F}^{[n]}\}\, .
\]
For each $I$ we denote by $f_I$ the third component of the index and we further subdivide $\mathscr{I}$ into $\mathscr{I}_R\cup \mathscr{I}_\varphi$ depending on whether $f_I\in \mathcal{F}^{[n],R}$ or $f_I\in \mathcal{F}^{[n], \varphi}$. 
Next we introduce the parameters $\tau = \tau_q$ and $\mu = \mu_q$ with $\tau_q^{-1}>0$ and $\mu_q^{-1}\in \mathbb N \setminus \{0\}$, which are explicitly given by
\begin{align}\label{mu.tau}
\mu_q^{-1} = 3 \ceil{\la_q^{\frac12}\la_{q+1}^{\frac12} \de_q^\frac14\de_{q+1}^{-\frac14} /3}, \quad 
\ta_q^{-1} = 40\pi {M}\eta^{-1}\cdot {\la_q^\frac12\la_{q+1}^\frac12 \de_q^\frac14\de_{q+1}^\frac14 },
\end{align}
(note that $M$ in $\tau_q$ is required in order to satisfy the last condition in \eqref{con.par}).
we define 
\begin{align*}
\th_I(t) = 
\begin{cases}
\th_m^3(\tau^{-1}t), \quad I\in \mathscr{I}_R\\
\th_m^2(\tau^{-1}t), \quad I\in \mathscr{I}_\ph,
\end{cases}
\quad
\chi_I(x)=
\begin{cases}
\chi_n^3(\mu^{-1}x), \quad I\in \mathscr{I}_R\\
\chi_n^2(\mu^{-1}x), \quad I\in \mathscr{I}_\ph\, .
\end{cases}
\end{align*}
Next, for each $I$ let $U_{f_I}$ be the corresponding Mikado flow. Moreover, given $I= (m,n,f)$, denote by $t_m$ the time $t_m = m\tau$ and let $\xi_I = \xi_m$ be the solution of the following PDE (which we understand as a map on $\mathbb  R \times \mathbb T^3$ taking values in $\mathbb T^3$, cf. Remark \ref{r:periodicity}):  
\begin{align}\label{def.bflow}
\begin{cases}
\pa_t \xi_m + (v_\ell \cdot \na)\xi_m =0\\
\xi_m(t_m,x) =x 
\end{cases}
\end{align}
In the rest of the paper $\nabla \xi_I$ will denote the Jacobi Matrix of the partial derivatives of the components of the vector map $\xi_I$ and we will use the shorthand notations $\nabla \xi_I^\top$, $\nabla \xi_I^{-1}$ and $\nabla \xi_I^{-\top}$ for, respectively, its transpose, inverse and transport of the inverse. Moreover, for any vector $f\in \mathbb R^3$ and any matrix $A\in \mathbb R^{3\times s}$ the notation $\nabla \xi_I f$ and $\nabla \xi_I A$ (resp. $\nabla \xi_I^{-1} f$, etc.) will be used for the usual matrix product, regarding $f$ as a column vector (i.e. a $\mathbb R^{3\times 1}$-matrix). 
 
For each $I = (m,n,f)$ we will also choose a shift 
\[
z_I= z_{m,n} + \bar p_f \in \mathbb R^3
\] 
and, setting $\lambda= \lambda_{q+1}$, we are finally able to introduce the main part of our perturbation, which is achieved using the following ``master function''
\begin{equation}\label{e:master}
W (R, \varphi, t,x) := \sum_{I\in \mathscr{I}} \theta_I (t) \chi_I (\xi_I (t,x)) {\gamma}_I (R, \varphi,t,x) \nabla \xi_I^{-1} (t,x) U_{f_I} (\lambda (\xi_I (t,x) - z_I))\, ,
\end{equation}
where the $\gamma_I$'s are smooth scalar functions (the ``weights'') whose choice will be specified in the next section.
In order to simplify our notation we will use $U_I$ for $U_{f_I} (\cdot - z_I)$, $\psi_I$ for $\psi_{f_I} (\cdot - z_I)$ and $\tilde{f}_I$ for $\nabla \xi_I^{-1} f_I$. We therefore have the writing
\begin{equation}\label{e:master2}
W := \sum_{I\in \mathscr{I}} \theta_I \chi_I (\xi_I) {\gamma}_I \tilde{f}_I \psi_I (\lambda \xi_I)\, .
\end{equation}
Note that, since we want $W$ to be a periodic function of $x$, we will impose 
\begin{equation}\label{e:periodicity}
z_{m,n} = z_{m, n'} \qquad \mbox{if ${\mu}(n-n')\in 2\pi \mathbb Z^3$.}
\end{equation}
Finally, the main part of the correction $v_{q+1}-v_q$ will take the form
\begin{equation}\label{e:wo}
w_o (t,x) := W (R_\ell (t,x), \varphi_\ell (t,x), t,x)\, 
\end{equation}
which is well-defined on $[0,T]+ 2\tau_q$. (Indeed, it is possible to have $[0,T]+ 3\tau_q\subset [0,T]+ \tau_{q-1}$ by the choice of sufficiently large $\la_0$). In the rest of Section \ref{sec:correction.const}, without mentioning, our analysis is done in the time interval $[0,T]+ 2\tau_q$.
Given the complexity of several formulas and future computations, it is convenient to break down the functions $ W$ and $w_o$ in more elementary pieces. To this aim we introduce the scalar maps
\[
w_I (t,x) := \theta_I (t) \chi_I (\xi_I (t,x)) \psi_I (\lambda (\xi_I (t,x)))\, ,
\]
using which we can write
\[
W (R, \varphi, t, x) = \sum_{I\in \mathscr{I}} \gamma_I (R, \varphi,t,x) \nabla \xi_I (t,x)^{-1} f_I w_I (t,x) 
= \sum_{I\in \mathscr{I}} \gamma_I (R, \varphi, t,x) \tilde{f}_I (t,x) w_I (t,x)\, 
\]
and
\[
w_o = \sum_{I\in \mathscr{I}} \gamma_I  \nabla \xi_I^{-1} f_I w_I =  \sum_{I\in \mathscr{I}} \gamma_I  \tilde{f}_I w_I\, .
\]
The crucial point in our construction is the following proposition, whose proof will be given in Section \ref{p:supports}.

\begin{prop}\label{p:supports}
There is a constant $\eta = \eta (\mathcal{F})$ in \eqref{e:eta} such that it allows a choice of the shifts $z_I= z_{m,n} + \bar p_f$ which ensure that {for each $(\mu_q, \tau_q, \la_{q+1})$,} the conditions $\supp (w_I) \cap \supp (w_J)=\emptyset$ for every $I\neq J$ and that \eqref{e:periodicity} for every $m,n$ and $n'$. 
\end{prop}

\subsection{Choice of the weights}\label{ss:weights} We next detail the choice of the functions $\gamma_I$, subdividing it into two cases.

\subsubsection{Energy weights} The weights $\ga_I$ for $I\in \mathscr{S}_\ph$ will be chosen so that the low frequency part of $\frac 12|w_o|^2w_o$ makes a cancellation with the mollified unsolved current $\ph_\ell$. 
Because of Proposition \ref{p:supports}, we have
\begin{align*}
|w_o|^2 w_o  =&\ \sum_{I\in \mathscr{I}}  \th_I^3 \chi_I^3(\xi_I) \ga_I^3 \psi_I^3(\la_{q+1}\xi_I) |\na \xi_I^{-1} f_I|^2\na \xi_I^{-1}  f_I\\
=&\underbrace{\sum_{I\in \mathscr{I}}  \th_I^3 \chi_I^3(\xi_I) \ga_I^3 \langle \psi_f^3\rangle   |\td f_I|^2\td f_I}_{=: (|w_o|^2 w_o)_L}
+\underbrace{\sum_{I\in \mathscr{I}}  \th_I^3 \chi_I^3(\xi_I) \ga_I^3 \left(\psi_I^3(\la_{q+1}\xi_I)-\langle \psi_I^3\rangle
\right) |\td f_I|^2\td f_I}_{=:(|w_o|^2 w_o)_H}\, .
\end{align*}
In order to find the desired $\ga_I$, we introduce the notation $\mathscr{I}_{m,n,\ph}$ for $\{I\in \mathscr{I}_\ph : I = (m,n,f)\}$ and we observe that, by \eqref{con.psi2}, it suffices to achieve
\begin{align}\label{W3L}
(|w_o|^2w_o)_L
=\sum_{m,n} \th_m^6\left(\frac t{\tau_q}\right) \chi_n^6\left(\frac{\xi_m}{\mu_q}\right)\sum_{I \in \mathscr{I}_{m,n,\ph}} \ga_I^3 |\tilde{f}_I|^2 \tilde{f}_I\, .
\end{align}
Next we look for our coefficients in the following form:
\[
\ga_I =\frac{\la_q^{-\ga}\de_{q+1}^\frac12 \Ga_I}{|\tilde{f}_I|^\frac23}\, ,
\]
where $\Gamma_I$ will be specified in a moment.

Recall that $\xi_I$ is a solution to \eqref{def.bflow} and satisfies $\na \xi_I |_{t=t_m} = \I$ and 
\[
\na \xi_I^{-1} (t,x) = \na \Phi_m (t, \xi_I(t,x))\, ,
\] 
where $\Phi_m$ is the ``forward flow''
$\Phi (t,x;t_m)$ introduced in \eqref{e:fflow_first_instance} and thus solves
\begin{equation}\label{eqn.fflow}
\begin{cases}
\pa_t \Phi_m(t,x) =v_\ell (t,\Phi_m(t,x))\\
\Phi_m(t_m, x) = x. 
\end{cases}
\end{equation}
This implies that  
\begin{align*}
\norm{\na \xi_I}_{C^0(\cal{I}_m\times \R^3)} &\leq  \exp(2\tau_q \norm{\na v_q}_0)
\leq  \exp(2{M}\tau_q\la_q\de_q^\frac12),
\end{align*}
\begin{equation}\begin{split}
\norm{\I - \na \xi_I^{-1}}_{C^0(\cal{I}_m\times \R^3)} 
&{= \norm{\I - \na \Phi_m}_{C^0(\cal{I}_m\times \R^3)}}\\
&\le 2\tau_q \norm{\na \Phi_m}_{C^0(\cal{I}_m\times \R^3)}\norm{\na v_q}_0\\
&\leq 2M\tau_q\la_q\de_q^\frac12\exp(2M\tau_q \la_q\de_q^\frac12 ) \label{bf.Id}
\end{split}\end{equation}
for the time interval $\cal{I}_m = [t_m -\frac 12\tau_q, t_m + \frac 32\tau_q]  \cap  [0,T]+2\tau_q$. 
Therefore, for sufficiently large $\la_0$, we have 
\begin{align*}
|\tilde{f}_I| = |\na \xi_I^{-1}f_I|&\geq \frac 34\\
\norm{2\la_q^{3\ga}\de_{q+1}^{-\frac32}(\na\xi_I)\ph_\ell}_{C^0_x}&\leq 3C_1
\end{align*} 
on the support of $\th_I$ for some $C_1$.  
Since $\{f_I : I \in \mathscr{I}_{m,n, \ph}\}= \mathcal{F}^{[n], \ph}$ satisfies \eqref{con.FIph}, we can apply Lemma \ref{lem:geo2} with $N_0= 3C_1$ to solve
\begin{align*}
-2\ph_\ell
=\sum_{I\in \mathscr{I}_{m,n,\ph}} \ga_I^3|\tilde{f}_I|^2 \tilde{f}_I
\iff 
-2\la_q^{3\ga}\de_{q+1}^{-\frac32}(\na\xi_I)\ph _\ell = \sum_{I\in \mathscr{I}_{m,n\ph}}\Ga_I^3 f_I
\end{align*}
on each support of $\th_I$ (observe that we have crucially used that $\xi_I=\xi_m$ is independent of $f_I$ for $I\in \mathscr{S}_{m,n,\ph}$). We are thus in the position to apply Lemma \ref{lem:geo2} to the set $\mathcal{F}^{[n],\ph} = \{f_I: I \in \mathscr{I}_{m,n,\ph}\}$ and we let $\Gamma_{f_I}$, $I\in \mathscr{I}_{m,n,\ph}$ be the corresponding functions. As a result, we can set
\begin{align}\label{Ga.ph}
\Ga_I(t,x) 
={\Ga}_{f_I}^{\frac{1}{3}}
({-2
{\la_q^{3\ga}\de_{q+1}^{-\frac 32}(\na\xi_I)\ph_\ell}} 
)\,.
\end{align}
Note that the smoothness of the selected functions $\Gamma_{f_I}$ depends only on $C_1$ and that in fact Lemma \ref{lem:geo2} is just applied $27$ times, taking into consideration that $[n]\in \mathbb{Z}_3^3$. For later use we record here the important ``cancellation property'' that the choice of our weights achieves:
\begin{align}\label{can.ph}
\frac 12 (|w_o|^2w_o)_L
=- \ph_\ell.
\end{align}

\subsubsection{Reynolds weights}\label{subsec:R}
Similarly to the previous section we decompose $w_o\otimes w_o$ into the low and high frequency parts,
\begin{align*} 
w_o \otimes w_o  
=&\ \sum_I \th_I^2 \chi_I^2(\xi_I) \ga_I^2 \psi_I^2(\la_{q+1}\xi_I)
\tilde f_I\otimes \tilde f_I\\
=&\underbrace{\sum_I \th_I^2 \chi_I^2(\xi_I) \ga_I^2 \langle\psi_I^2\rangle \tilde f_I\otimes \tilde f_I}_{=: (w_o\otimes w_0)_L}
+\underbrace{\sum_I \th_I^2 \chi_I^2(\xi_I) \ga_I^2 \left(\psi_I^2(\la_{q+1}\xi_I)-\langle\psi_I^2\rangle \right) \tilde f_I\otimes \tilde f_I}_{(w_o\otimes w_o)_H}.
\end{align*}
Since the weights for $I\in \mathscr{I}_\ph$ have already been established, for each fixed $(m,n)$ we denote by $I (m,n)$ the sets of indices $(m'n')$ such that $\max \{|m-m'|_\infty, |n-n'|_\infty\}\leq 1$ (where $|u|_\infty := \max \{|u_1|, |u_2|, |u_3|\}$ for any $u\in \mathbb R^3$) and {rewrite
\begin{align*}
& (w_o\otimes w_o)_L
= \sum_{m,n} \th^6_m\left(\frac t{\tau_q}\right) \chi^6_n\left(\frac{\xi_m}{\mu_q}\right) 
\sum_{I\in \mathscr{I}_{m,n,R}} \ga_I^2 \td f_I \otimes \td f_I
+\sum_{J\in  \mathscr{I}_{\ph}} \th_J^2 \chi_J^2(\xi_I) \ga_J^2 \langle\psi_J^2\rangle \tilde f_J\otimes \tilde f_J\\
=& \sum_{m,n} \th^6_m\left(\frac t{\tau_q}\right) \chi^6_n\left(\frac{\xi_m}{\mu_q}\right) \left[
\sum_{I\in \mathscr{I}_{m,n,R}} \ga_I^2 \td f_I \otimes \td f_I
+\sum_{\substack{J\in  \mathscr{I}_{m',n',\ph}\\ (m',n')\in I(m,n)}}\th_J^2 \chi_J^2(\xi_I) \ga_J^2 \langle\psi_J^2\rangle \tilde f_J\otimes \tilde f_J\right].
\end{align*}}
To make $w_o\otimes w_o$ cancel out $\de_{q+1}\I - R_l$, we recall that $\tilde{f}_I\otimes \tilde{f}_I = \nabla \xi_I^{-1} (f_I\otimes f_I) \nabla \xi_I^{-\top}$ and set
\begin{align}\label{eq.Ga}
&\sum_{I\in \mathscr{I}_{m,n,R}} \ga_I^2 f_I \otimes f_I
 = \na \xi_I 
\Bigg[\de_{q+1} \I  -R_\ell -\sum_{(m'n') \in I (m,n)} \sum_{J\in \mathscr{I}_{m',n',\ph}} \theta_J^2 \chi_J^2 \ga_{J}^2 (\xi_J) \langle\psi_{J}^2\rangle  \tilde{f}_J \otimes  \tilde{f}_J \Bigg]
\na \xi_I^{\top}\, . 
\end{align}
We now define $\mathcal{M}_I$ as 
\begin{align*}
\mathcal{M}_I &= \de_{q+1} [\na \xi_I\na \xi_I^{\top}-\I ] -
\na \xi_I R_\ell \na \xi_I^{\top}  \\
&\quad
-\na\xi_I \left[\sum_{(m'n') \in I (m,n)} \sum_{J\in \mathscr{I}_{m',n',\ph}}\th_{J}^2\chi_{J}^2(\xi_{J}) \ga_{J}^2 \langle \psi_{J}^2\rangle  \tilde f_J \otimes  \td f_J \right] \na\xi_I^{\top}
\end{align*}
and $\gamma_I = \delta_{q+1}^{\frac{1}{2}} \Gamma_I$ (for $I\in \mathscr{I}_{m,n,R}$) and impose
\begin{equation} \label{eq.Ga1}
\sum_{I\in \mathscr{I}_{m,n,R}} \Ga_I^2 f_I  \otimes  f_I
= \I + \de_{q+1}^{-1} \mathcal{M}_I
\end{equation}
In order to show that such a choice is possible observe that we can make $\norm{\de_{q+1}^{-1}\mathcal{M}_I}_{C^0(\supp(\th_I)\times \R^3)}$ sufficiently small, provided that $\la_0$ is sufficiently large, because of \eqref{bf.Id}, $\norm{\de_{q+1}^{-1}R_\ell}_0\lec \la_q^{-3\ga}$, and, $\norm{\de_{q+1}^{-1}\ga_J^2}_0\lec \la_q^{-2\ga}$ when $J \in \mathscr{I}_{m'n',\ph}$. We can thus apply Lemma \ref{lem:geo1} to $\{f_I : I \in \mathscr{I}_{m,n,R}\}= \mathcal{F}^{[n],R}$ and, denoting by $\Gamma_{f_I}$ the corresponding functions, we just need to set 
\[
\Ga_I=\Ga_{f_I} ( \I - \de_{q+1}^{-1} \mathcal{M}_I)\, .
\] 
Observe once again that this means applying Lemma \ref{lem:geo1} just $27$ times, given that there are $27$ different families $\mathcal{F}^{[n], R}$. 
We finally record the desired ``cancellation property'' that the choice of the weights achieves:
\begin{equation}\begin{split} \label{can.R}
( w_o\otimes w_o)_L
&=\sum_{m,n}\th_m^6\left(\frac t{\tau_q}\right) \chi_n^6\left(\frac{\xi_m}{\mu_q}\right)  (\de_{q+1}\I-R_\ell)
=\de_{q+1}\I-R_\ell.
\end{split}\end{equation}

\subsection{Fourier expansion in fast variables and corrector $w_c$}\label{ss:correction} In the rest of this article, we use a representation of $w_o$, $w_o\otimes w_o$, and $\frac12 |w_o|^2w_o$ based on the Fourier series of $\psi_I$, $\psi_I^2$ and $\psi_I^3$.
Indeed, since $\psi_I$ is a smooth function on $\T^3$ with zero-mean, we have
\begin{equation}\label{rep.psi}
\psi_I (x) = \sum_{k\in \Z^3\setminus \{0\}} \dot{b}_{I,k}  e^{ i k\cdot x},\quad
\psi_I^2 (x) = \dot c_{I,0} 
 +\sum_{k\in \Z^3\setminus \{0\}} \dot c_{I,k}   e^{ i k\cdot x},\quad
\psi_I^3 (x) = \dot d_{I,0}  + 
\sum_{k\in \Z^3\setminus \{0\}} \dot d_{I,k}   e^{ i k\cdot x}
\end{equation}
In particular, 
\[
\dot c_{I,0} =\langle \psi_I^2\rangle, \quad
\dot d_{I,0} = \langle \psi_I^3\rangle.
\]
Since $\psi_I$ is in $C^\infty(\T^3)$, we have  
\begin{equation}\label{est.k}
\sum_{k\in \Z^3} |k|^{n_0+2}|\dot{b_{I,k}}|
+\sum_{k\in \Z^3} |k|^{n_0+2}|\dot{c_{I,k}}|
+\sum_{k\in \Z^3} |k|^{n_0+2}|\dot{d_{I,k}}|
 \lec 1 , \quad 
\sum_{k\in\Z^3} |\dot{c}_{I,k}|^2\lec 1. 
\end{equation}
for $n_0 = \ceil{\frac{2b(2+\al)}{(b-1)(1-\al)}}$. Also,  it follows from $f_I \cdot \na \psi_I = f_I \cdot \na \psi_I^2=f_I \cdot \na \psi_I^3=0 $ that
\begin{align}\label{coe.van}
\dot b_{I,k} (f_I\cdot k )=\dot c_{I,k} (f_I\cdot k )=\dot d_{I,k} (f_I\cdot k )=0.
\end{align}
Next, as a consequence of \eqref{can.ph}, \eqref{can.R}, and \eqref{rep.psi}, we have 
\begin{align}
&w_o = \sum_{m} \sum_{k\in \Z^3\setminus \{0\}} \de_{q+1}^\frac 12b_{m,k} e^{i\la_{q+1} k\cdot \xi_I} \label{rep.W}\\
&w_o\otimes w_o
=\de_{q+1}\I -R_\ell
+\sum_{m} \sum_{k\in \Z^3\setminus \{0\}}  \de_{q+1} c_{m,k} e^{ i\la_{q+1} k\cdot \xi_I} \label{alg.eq}\\
&\frac12 |w_o|^2 w_o =-\ph_\ell
+\frac12\sum_{m} \sum_{k\in \Z^3\setminus \{0\}}  \de_{q+1}^\frac 32 d_{m,k} e^{ i\la_{q+1} k\cdot \xi_I}\label{e:rep3}\\
&\frac12|w_o|^2
=-\ka_\ell +\frac32 \de_{q+1}  
+\frac 12\sum_{m} \sum_{k\in \Z^3\setminus \{0\}}  \de_{q+1} \tr(c_{m,k}) e^{ i\la_{q+1} k\cdot \xi_I}.\label{e:rep4}
\end{align}
where the relevant coefficients are defined as follows:
\begin{equation}\begin{split}\label{coef.def}
&b_{m,k} = \sum_{I: m_I=m} \th_I\chi_I(\xi_I) \de_{q+1}^{-\frac 12}\ga_I \dot b_{I,k} \tilde f_I =: \sum_{I:m_I =m} B_{I,k} \tilde{f}_I,\\
&c_{m,k}
=
\sum_{I: m_I=m} \th_I^2\chi_I^2(\xi_I)
\de_{q+1}^{-1}\ga_I^2 \dot c_{I,k}
\tilde f_I\otimes \tilde f_I,\\
&d_{m,k}
= \sum_{I: m_I=m} \th_I^3\chi_I^3(\xi_I)
\de_{q+1}^{-\frac 32}\ga_I^3 \dot d_{I,k}
|\tilde f_I|^2 \tilde f_I.
\end{split}\end{equation}
Observe that,  by the choice of $\th_I$, if $|m-m'|> 1$, then
\begin{align*}
\supp_{t,x} (b_{m,k})\cap \supp_{t,x} (b_{m',k'})
&=\supp_{t,x} (c_{m,k})\cap \supp_{t,x} (c_{m',k'})\\
&=\supp_{t,x} (d_{m,k})\cap \supp_{t,x} (d_{m',k'})
 = \emptyset
\end{align*}
for any $k, k'\in \Z^3\setminus\{0\}$.  

We next prescribe an additional correction $w_c$ to make $w=w_o+w_c$ divergence-free. 
Since we have \eqref{coe.van} and the identity $\na \times (\na \xi_I^{\top}U(\xi_I)) = \na \xi_I^{-1}(\na \times U)(\xi_I)$ for any smooth function $U$ (see for example \cite{DaSz2016}), we have 
\[
\frac 1{\la_{q+1}}\na \times \left( \dot{b}_{I,k}\na\xi_I^{\top} \frac{ik \times f_I}{|k|^2}  e^{i\la_{q+1} k\cdot {\xi_I}}\right) =\dot{b}_{I,k} \na \xi_I^{-1}f_I e^{i\la_{q+1} k\cdot \xi_I}.
\]
Using this, the ``preponderant part'' $w_o$ of the velocity correction can be written as
\begin{align*}
w_o 
&= \sum_{\substack{m\in \Z\\k\in \Z^3\setminus \{0\}}}  \de_{q+1}^\frac 12 \sum_{I:m_I=m} B_{I,k} \na \xi_I^{-1}f_I e^{i\la_{q+1} k\cdot \xi_I}\\
& =  \sum_{\substack{m\in \Z\\k\in \Z^3\setminus \{0\}}}  \frac{ \de_{q+1}^\frac 12}{\la_{q+1}}\sum_{I:m_I=m}  
B_{I,k} \na \times\left( \na \xi_I^{\top}\frac{ik \times f_I}{|k|^2} e^{i\la_{q+1} k\cdot \xi_I}\right).
\end{align*}
Therefore, we define
\begin{align}\label{def.Wc}
w_c = \frac {\de_{q+1}^\frac 12 } {\la_{q+1}} \sum_{\substack{m\in \Z\\k\in \Z^3\setminus\{0\}}} \sum_{m_I =m} \na B_{I,k} \times \left(\na\xi_I^{\top} \frac{ik \times f_I}{|k|^2} \right) e^{i\la_{q+1} k\cdot \xi_I}
=: \frac {\de_{q+1}^\frac 12}{\la_{q+1}\mu_q}\sum_{m,k} e_{m,k} e^{i\la_{q+1} k\cdot \xi_I}
\end{align}
where
\begin{align}\label{def.e}
e_{m,k} 
= {\mu_q}\sum_{I:m_I=m} \na(\th_I\chi_I(\xi_I) \de_{q+1}^{-\frac 12}\ga_I \dot b_{I,k}) \times  \left((\na\xi_I)^{\top} \frac{ik \times f_I}{|k|^2} \right).
\end{align}
In this way, the final velocity correction $v_{q+1}-v_q =: w= w_o+w_c$ can be written as
\begin{align*}
w = \na\times \left( \frac {\de_{q+1}^\frac12}{\la_{q+1}}
\sum_{I,k}  
B_{I,k} \na\xi_I^{\top} \frac{ik \times f_I}{|k|^2}e^{i\la_{q+1} k\cdot \xi_I} \right),
\end{align*}
and hence it is divergence-free. For later use, we remark that if $|m-m'| {>} 1$, $\supp_{t,x} (e_{m,k})\cap \supp_{t,x} (e_{m',k'})
 = \emptyset$ holds
for any $k, k'\in \Z^3\setminus\{0\}$.
 Also, by its definition, the correction $w$ has the representation
\begin{align}\label{def.w}
w= \sum_{m\in \Z} \sum_{k\in \Z^3\setminus \{0\}} \de_{q+1}^\frac 12 (b_{m,k} + (\la_{q+1}\mu_q)^{-1} e_{m,k} )e^{i\la_{q+1} k\cdot \xi_I}.
\end{align}

\section{Definition of the new errors}

\subsection{Preliminaries}
To define the new triple $(\kappa_q, R_q, \ph_q)$ we need to ``invert the divergence'' of vector fields and tensors. For this purpose, we recall the inverse divergence operator introduced in \cite{DLSz2013}.

\begin{defn}[Inverse divergence operator]\label{idv.defn}
For any $f\in C^\infty(\T^3;\R^3)$, the inverse divergence operator is defined by
\[
(\cal{R}f)_{ij}= \cal{R}_{ijk} f_k 
= -\frac 12 \De^{-2} \pa_{ijk} f_k + \frac 12\De^{-1} \pa_k f_k \de_{ij} - \De^{-1}\pa_i f_j - \De^{-1}\pa_jf_i. 
\]
\end{defn}
\begin{rem} The image of the divergence free operator $\cal{R}f(x)$ is designed to be a trace-free symmetric matrix at each point $x$ and to solve 
\[
\div(\cal{R}f) = f - \langle f \rangle\, .
\]
\end{rem}
To define unsolved currents, we also need an inverse divergence operator which maps a mean-zero scalar function to a mean-zero vector-valued one. To this end, we abuse the notation and define 
\[
(\cR g)_i = \De^{-1} \pa_i g. 
\]
Indeed, $\div \cR g =g - \langle g \rangle$. 

\subsection{New Reynolds stress}
Having defined the correction $w$ of the velocity as in the previous steps and after setting $0$ for the correction on the pressure (namely $p_{q+1} = p_q$), 
we can reorganize the Euler-Reynolds system and the relaxed local energy equality as the equations for the new Reynolds stress $R_{q+1}$ and for the unsolved current $\ph_{q+1}$, respectively (while we simply impose $\kappa_{q+1} = \frac{1}{2} \textrm{tr}\, R_{q+1}$.
  
We first define $R_{q+1}$. Using the given relation $\pa_t v_q + \na\cdot(v_q\otimes v_q) + \na p_q = \na\cdot R_q$ (note that we are dropping $\frac{E(t)}{3} \I$ from the equation as the latter, being just a function of time, disappears after we apply the divergence) , we can write the equation for $R_{q+1}$ as 
\begin{align*}
\na \cdot R_{q+1}
&= \pa_t v_{q+1} + \na \cdot (v_{q+1} \otimes v_{q+1}) + \na p_{q+1} \\
&=\underbrace{(\pa_t+v _\ell\cdot \na) w}_{=\na \cdot \as R_T} + \underbrace{w\cdot \na v_\ell}_{=:\na\cdot \as R_N} 
+ \underbrace{\na \cdot (w\otimes w + R_\ell)}_{=:\na \cdot \as R_O} \\
&\quad+ \underbrace{\na \cdot ((v_q-v_\ell)\otimes w + w\otimes (v_q-v_\ell) + R_q-R_\ell) }_{=:\na\cdot \as R_M},
\end{align*}
and $\asR_O$ can be decomposed further as 
\[
\na\cdot \asR_O=\underbrace{\na \cdot (w_o\otimes w_o+ R_\ell )}_{=\na \cdot \asR_{O1}}
+ \underbrace{\na \cdot (w_o \otimes w_c + w_c \otimes w_o + w_c\otimes w_c) }_{=\na \cdot \asR_{O2}}.
\]
In fact we will define $R_{q+1}$ as
\begin{equation}\label{e:splitting_of_Reynolds}
R_{q+1} = \as R_T + \as R_N + \as R_{O1} +\as R_{O2}+\as R_M + \frac{2}{3} \varrho (t) \I\, .
\end{equation}
Note that the last summand is not impacting the divergence of $R_{q+1}$, since it is a function of time only. The reason for introducing this extra term, which at the moment is not of any relevance, will be clear once we get to the definition of the new current.

First, our choice of $\asR_{O2}$ and $\as R_M$ are
\begin{equation}\begin{split}\label{def.RO2}
\asR_{O2} =  w_o \otimes w_c + w_c \otimes w_o + w_c\otimes w_c\, .
\end{split}\end{equation}
and
\begin{equation}\begin{split}\label{def.RM}
\asR_M = \underbrace{R_q-R_\ell}_{=\asR_{M1}}+\underbrace{(v_q-v_\ell)\otimes w + w\otimes (v_q-v_\ell)}_{=\asR_{M2}},
\end{split}\end{equation}
which are the only two Reynolds stress errors which might have nonzero trace. For the other errors, we solve the divergence equation by using the inverse divergence operator $\cal{R}$ in Definition \ref{idv.defn} to get trace-free errors, namely we set
\begin{align*}
\asR_{O1} &= \mathcal{R} (\nabla \cdot (w_o\otimes w_o+ R_\ell ))\\
\asR_{N} &= \mathcal{R} (w\cdot \nabla v_\ell)\\
\asR_{T} &= \mathcal{R} (\pa_t w + v_\ell \cdot \nabla w)\, . 
\end{align*}
Observe that all the tensors to which we apply the operator $\mathcal{R}$ have zero average, because $w\cdot \nabla v_\ell = \nabla \cdot (w\otimes v_\ell)$, $v_\ell \cdot \nabla w = \nabla \cdot (v_\ell \otimes w)$ and $w$ (and therefore $\pa_t w$) have zero average.

As a result, we also have
\[
\tr R_{q+1} = \tr (\as R_{O2} + \as R_{M}) + 2 \varrho\, ,
\]
which gives 
\begin{equation}\label{askR}
\kappa_{q+1} =\frac{1}{2} \tr R_{q+1}= \frac 12  \tr(\as R_{O2} + \as R_{M}) + \varrho. 
\end{equation}
In fact, for later use we will split the function $\varrho$ into the sum of three functions, $\varrho_0 + \varrho_1+\varrho_2$ and we thus have
\begin{align*}
\kappa_{q+1} &= \frac 12  \tr(\as R_{O2} + \as R_{M}) + \varrho_0 + \varrho_1+\varrho_2\, .
\end{align*}

\subsection{New current}
Applying the frequency cut-off $P_{\leq \ell^{-1}}$ to the Euler-Reynolds system, we have
\[
\pa_t v_\ell + \na \cdot (v_\ell\otimes v_\ell) + \na p_\ell
= \na \cdot P_{\le \ell^{-1}} R_q +  Q(v_q,v_q).
\]
where, upon setting $(v_q\otimes v_q)_\ell = P_{\le \ell^{-1}}(v_q\otimes v_q)$, the term $Q(v_q, v_q)$ is the following commutator:
\begin{equation}\label{e:quadratic_form_Q}
Q(v_q,v_q) := \na \cdot (v_\ell \otimes v_\ell - (v_q\otimes v_q)_\ell)\, .
\end{equation}
Also, we recall that the tuple $(v_q,p_q,R_q,\ka_q,\ph_q)$ solves 
\[
\pa_t \left(\frac 12 |v_q|^2 \right) + \na \cdot \left(\left(\frac 12| v_q|^2 +  p_q\right)  v_q\right)
= (\pa_t + v_q\cdot \na) \ka_q + {\textstyle{\frac{1}{2}}} E'(t) + \na \cdot (R_q v_q) + \na \cdot \ph_q. 
\]
Using these equations, the relaxed local energy equality for $(v_{q+1},p_{q+1},R_{q+1},\ka_{q+1},\ph_{q+1})$ can be reorganized as 
\begin{align*}
\pa_t &\left(\frac 12 |v_{q+1}|^2 \right) + \na \cdot \left(\left(\frac 12|v_{q+1}|^2 + p_{q+1}\right) v_{q+1}\right)\\
=& {\textstyle{\frac{1}{2}}} E' (t) + \
\underbrace{(\pa_t+ v_q\cdot \na) \left(\frac 12 |w|^2 + \ka_q+ (v_q-v_\ell)\cdot w \right)}_{=: (\pa_t + v_{q+1}\cdot \nabla) (\ka_{q+1}-\varrho_1-\varrho_2) + \na \cdot \as\ph_T} 
+ \underbrace{\na\cdot \left(\left(\frac 12|w|^2 \right) w + \ph_\ell \right)}_{=\na\cdot\as\ph_O} \\
& + \div (R_{q+1} v_{q+1}) \underbrace{-\div(R_{q+1} w)}_{= \na\cdot\as\ph_R}    \\
&+ \underbrace{\na \cdot \left(\left( \frac12|v_q-v_\ell|^2+ (p_q-p_\ell)\right)w\right)+ \na \cdot (\ph_q-\ph_\ell) 
+ \na\cdot  ((w\otimes w+R_q -R_{q+1}) (v_q-v_\ell))}_{=\na\cdot\as\ph_M} 
\\
&+\underbrace{(\div P_{\le \ell^{-1}}R_q + Q(v_q,v_q)) \cdot w}_{=\nabla \cdot \as\ph_{H1} + \pa_t \varrho_1}\\
&+ \underbrace{\left(w\otimes w-\de_{q+1}\I+R_q-R_{q+1} + (v_q-v_\ell)\otimes w+ w\otimes (v_q-v_\ell) \right): \na v_\ell^\top}_{= \pa_t \varrho_2 + \na\cdot\as\ph_{H2}}\, . 
\end{align*}
The functions {$\varrho_0$}, $\varrho_1$, and $\varrho_2$ will be defined so to ensure that the divergence equations can be solved. Indeed observe that in order to solve $\nabla \cdot z = f$ on the periodic torus, it is necessary and sufficient that the average of $f$ equals $0$. In fact under such assumption a solution is provided by the operator $\mathcal{R}$ introduced above. More precisely, first of all we observe that we can set
\begin{align}
\as\varphi_O &:= \as\varphi_{O1} + \as\varphi_{O2} := \mathcal{R} \left(\na \cdot \left(\left({\textstyle{\frac 12}}|w_o|^2 \right) w_o 
+ \ph_\ell \right)\right) 
+\mathcal{R} \left(\na \cdot \left({\textstyle{\frac 12}} (|w|^2w - |w_o|^2 w_o) \right)\right)\\
\as\varphi_R &:= (R_{q+1}-{\textstyle{\frac23}}\varrho \I) w\\
\begin{split}
\as\varphi_M &:=  \left( {\textstyle{\frac12}}|v_q-v_\ell|^2+ (p_q-p_\ell)\right) w + \ph_q-\ph_\ell\\
&\quad+ (w\otimes w {-\de_{q+1}\I}+R_q -R_{q+1}+ {\textstyle{\frac23}}\varrho \I) (v_q-v_\ell)\, .
\end{split}
\end{align}
Next, recall that $\kappa_{q+1} = \frac 12  \tr(\as R_{O2} + \as R_{M}) + \varrho$ and that
\begin{equation}\begin{split}\label{reo.tra}
\frac 12 |w|^2 + \ka_q + (v_q-v_\ell) \cdot w
&= \frac 32 \de_{q+1}  + \frac 12 \tr(\asR_{O2} + \asR_{M})+ 
\frac 12\tr\left(  w_o\otimes w_o - \de_{q+1} \I + R_\ell \right)
\\
&=\frac 32 \de_{q+1} + \ka_{q+1} -\varrho + \frac 12\tr\left(  w_o\otimes w_o - \de_{q+1}\I + R_\ell \right).
\end{split}\end{equation}
In particular, the equation for $\as\ph_T$ becomes:
\begin{align*}
\nabla \cdot \as\ph_T +\varrho _0' = & \div (-(\ka_{q+1} - \varrho) w+ {\textstyle{\frac12}} \tr\left( w_o\otimes w_o -\de_{q+1}I+ R_\ell \right)(v_q-v_\ell)
\\ &
+ {\textstyle{\frac12}} D_{t,\ell}  \tr\left( w_o\otimes w_o -\de_{q+1}I+ R_\ell \right) \, 
\end{align*}
(where we have used that $\nabla \cdot w =0$).
Hence we can define $\as\ph_T = \as\ph_{T1} + \as\ph_{T2}$ and $\varrho_0$ as 
\begin{align}
\as\ph_{T1} &= -(\ka_{q+1} - \varrho) w+ \frac12 \tr\left( w_o\otimes w_o -\de_{q+1}I+ R_\ell \right)(v_q-v_\ell)\label{e:phi_T1}\\
\varrho_0 (t) &= \int_0^t \langle D_{t,\ell} {\textstyle{\frac12}} \tr\left( w_o\otimes w_o -\de_{q+1}I+ R_\ell \right)\rangle (s)\, ds\\
\as\ph_{T2} &= \mathcal{R} \left(D_{t,\ell} {\textstyle{\frac12}} \tr\left( w_o\otimes w_o -\de_{q+1}I+ R_\ell\right)\right)\, .
\end{align}
Observe that $\varrho_0$ is defined in such a way that $\div \as\ph_{T2} + \varrho_0' = D_{t,\ell} {\textstyle{\frac12}} \tr\left( w_o\otimes w_o -\de_{q+1}I+ R_\ell\right)$. 
Similarly, we set
\begin{align}
\varrho_1 (t) &:= \int_0^t \langle (\div P_{\le \ell^{-1}}R_q + Q(v_q,v_q)) \cdot w \rangle (s)\, ds\\
\varrho_2 (t) &:= \int_0^t \langle\left(w\otimes w-\de_{q+1}\I+R_q-R_{q+1} + (v_q-v_\ell)\otimes w+ w\otimes (v_q-v_\ell) \right): \na v_\ell^\top\rangle (s)\, ds\label{e:define_varrho_2}
\end{align}
and
\begin{align*}
\as\ph_{H1} &:= \mathcal{R} \left((\div P_{\le \ell^{-1}}R_q + Q(v_q,v_q)) \cdot w - \partial_t \varrho_1\right)\\
\as\ph_{H2} &:=\mathcal{R} \left(\left(w\otimes w-\de_{q+1}\I+R_q-R_{q+1} + (v_q-v_\ell)\otimes w+ w\otimes (v_q-v_\ell) \right): \na v_\ell^\top - \partial_t \varrho_2 \right)\, .
\end{align*}
Observe that, while $R_{q+1}$ has been defined in terms of $\varrho$, we have
\begin{align*}
R_{q+1} : \nabla v_\ell^\top &= (\asR_T + \asR_N + \asR_{O1} +\asR_{O2}+\asR_M + \frac{2}{3} \varrho (t) \I) : \nabla v_\ell^\top\\
&= (\asR_T + \asR_N + \asR_{O1} +\asR_{O2}+\asR_M ) : \nabla v_\ell^\top
\end{align*}
since $\I : \nabla v_\ell^\top = \nabla \cdot v_\ell = 0$. In particular the right hand side of \eqref{e:define_varrho_2} is independent of $\varrho_2$.

\section{Preliminary estimates}\label{s:estimates_mollification}

We now start detailing the estimates which will lead to the proof of the inductive propositions. In this section, we set $\norm{\cdot}_N = \norm{\cdot}_{C^0([0,T]+\tau_q; C^N(\T^3))}$.

\subsection{Regularization} First of all we address a series of a-priori estimates on the regularized tuple and on their differences with the original one.
By its construction, we can easily see that
\begin{align*}
&\norm{v_\ell}_N \lec_N \ell^{-N} \de_{q}^\frac12, \quad
\norm{p_\ell}_N \lec_N \ell^{-N}\de_q, \quad\qquad\qquad \forall N\geq 1,\\
\norm{D_{t,\ell}^s R_\ell}_0 &\lec_{s} \ell_t^{-s}\la_q^{-3\ga} \de_{q+1}, \quad
\norm{D_{t,\ell}^s \ph_\ell}_0 \lec_{s} \ell_t^{-s}  \la_q^{-3\ga} \de_{q+1}^\frac32, \quad \forall s\geq 0. 
\end{align*}
{Also, there exists $\bar{b}(\al)>1$ such that
for any $b\in (1,\bar b(\al))$ we can find $\La_0=\La_0(\al,b,M)$ with the following property: if $\la_0\geq \La_0$, then $|\na^{N+1} \Phi(t+s,x;t)|\lec_M  \ell^{-N}$ holds for $N\geq 0$ and $s\in [-\ell_t,\ell_t]$. This implies}
\begin{align}
\ell_t^s\norm{D_{t,\ell}^s R_\ell}_N {\lec_{s,N,M}} \ \ell^{-N} \la_q^{-3\ga} \de_{q+1}
\label{est.mR}
\\
\ell_t^s\norm{D_{t,\ell}^s \ph_\ell}_N 
{\lec_{s,N,M}} \ \ell^{-N} \la_q^{-3\ga} \de_{q+1}^\frac32
\label{est.mph}.
\end{align}
For the detailed computation, see \cite[Section18]{Is2013}. 

On the other hand, the differences between the regularized objects $(v_\ell, p_\ell, R_\ell, \ph_\ell)$ and their original counterparts satisfy the following estimates.
 
\begin{lem}{There exists $\bar{b}(\al)>1$ such that
for any $b\in (1,\bar b(\al))$ we can find $\La_0(\al,b,M)$ with the following property. If $\la_0\geq \La_0$}
%sufficiently close to $1$ there exists $\La_0=\La_0(b)$ such that the following estimates hold for $\la_0\geq \La_0$ and for for 
and $N\in \{0,1,2\}$ (recall that we follow the notational convention explained in Remark \ref{r:estimates_material_derivative}) then:
\begin{align}
&\norm{v_q-v_\ell}_N 
+ \de_q^{-\frac12}\norm{D_{t,\ell}(v_q-v_\ell)}_{N-1}  
\lec \ell^{2-N} \la_q^2 \de_q^\frac12, \label{est.v.dif} \\
&\norm{p_q-p_\ell}_N
+ \de_q^{-\frac12}\norm{D_{t,\ell}(p_q-p_\ell)}_{N-1} 
\lec \ell^{2-N}\la_q^2 \de_q,
 \label{est.p.dif}  \\ 
&\norm{R_q-R_\ell}_N+ 
\de_{q+1}^{-\frac12}\norm{D_{t,\ell} (R_q-R_\ell)}_{N-1} 
\lec
\la_{q+1}^N {\la_q^\frac12}{\la_{q+1}^{-\frac12}}
\de_q^\frac14 \de_{q+1}^\frac34
\label{est.R.dif} \\
&\norm{\ph_q-\ph_\ell}_N+\de_{q+1}^{-\frac12}\norm{D_{t,\ell} (\ph_q-\ph_\ell)}_{N-1}
 \lec  
\la_{q+1}^N{\la_q^\frac12}{\la_{q+1}^{-\frac12}}\de_q^\frac14 \de_{q+1}^\frac54.\label{est.ph.dif} 
\end{align}
{Here, we allow the implicit constants to be depending on $M$.}
\end{lem}

\begin{proof} Set $P_{>\ell^{-1}}F := F- P_{\leq \ell^{-1}} F = P_{>2^J}F$, where $J\in \N$ is the largest natural number such that $2^J\le \ell^{-1}$. By Bernstein's inequality, we have $\norm{F-P_{\le \ell^{-1}}F}_0 = \norm{P_{> \ell^{-1}}F}_0 \lec \ell^{j}\norm{\na^j F}_0$ for any $F\in C^j(\T^3)$. Using \eqref{est.vp} we then get
\begin{align*}
&\norm{v_q-v_\ell}_N 
\lec \ell ^{2-N}\norm{\na^2 v}_0
\lec {\ell^{2-N}}\la_q^2 \de_q^\frac12, \\
&\norm{p_q-p_\ell}_N 
\lec  \ell^{2-N}\norm{\na^2 p}_0
 \lec {\ell^{2-N}}\la_q^2 \de_q.
\end{align*}
Also, we have
\begin{align*}
(F- \rho_{\ell_t} \ast_{\Phi} F)(t,x)
&=  \int_{\R} (F(t+s,\Phi(t+s,x;t)) - F(t,x)) \rho_{\ell_t}(s) ds\\
&= \int_{\R} \int_0^s D_{t,\ell} F(t+\tau, \Phi(t+ \tau,x;t)) d\tau  \rho_{\ell_t}(s) ds, 
\end{align*}
from which we conclude $ \norm{F- \rho_{\ell_t} \ast_{\Phi} F}_{C^0([a,b]\times \T^3)} \lec \ell_t \norm{D_{t,\ell} F}_{C^0([a,b]+\ell_t\times \T^3)}$ because of $\supp(\rho_{\ell_t}) \subset (-\ell_t,\ell_t)$. 
In addition, we have a following decompositions,
\begin{align}
F-\rho_{\ell_t} \ast_\Phi P_{\le \ell^{-1}} F &= (F-P_{\le \ell^{-1}}F) + (P_{\le \ell^{-1}}F-\rho_{\ell_t} \ast_\Phi P_{\le \ell^{-1}} F),  \label{rep.dif}\\
D_{t,\ell}P_{\le \ell^{-1}} F 
&= P_{\le \ell^{-1}}D_{t,\ell} F + [v_\ell\cdot \na, P_{\le \ell^{-1}}]F\, ,
\label{rep.DtF}
\end{align}
where as usual $[A,B]$ denotes the commutator $AB - BA$ of the two operators $A$ and $B$.
Note that $D_{t,\ell} F$ can be further decomposed as  $D_{t,\ell} F = D_t F + (v_q-v_\ell)\cdot \na F$. Then, using \eqref{est.R}, \eqref{est.ph}, and \eqref{est.com1}, we obtain
\begin{align*}
\norm{R_q-R_\ell}_0 
&\lec \norm{P_{> \ell^{-1}} R_q}_0
+ \ell_t\norm{D_{t,\ell}  P_{\le \ell^{-1}}R_q}_{C(\mathcal{I}^q \times \T^3)}\\
&\lec \ell^2\norm{R_q}_2 +\ell_t(  \norm{D_{t,\ell}  R_q}_{C(\mathcal{I}^q \times \T^3)} + \ell \norm{\na v_q}_{C(\mathcal{I}^q \times \T^3)}\norm{\na R_q}_{C(\mathcal{I}^q \times \T^3)})\\
&\lec  ( (\ell\la_q)^2 + \ell_t \la_q\de_q^\frac12)\la_q^{-3\ga} \de_{q+1}
\lec  {\la_q^\frac12}{\la_{q+1}^{-\frac12}}\de_q^\frac14 \de_{q+1}^\frac34,
\end{align*}
and
\begin{align*}
\norm{\ph_q-\ph_\ell}_0
&\lec
\norm{P_{> \ell^{-1}} \ph_q}_0 + \ell_t \norm{D_{t,\ell}P_{\le \ell^{-1}}\ph_q}_{C(\mathcal{I}^q \times \T^3)} \\
&\lec   \ell^2\norm{\ph_q}_2 
+ \ell_t(\norm{D_{t,\ell}\ph_q}_{C(\mathcal{I}^q \times \T^3)} + \ell^{1} \norm{\na v_q}_{C(\mathcal{I}^q \times \T^3)} \norm{\na \ph_q}_{C(\mathcal{I}^q \times \T^3)}) \\
&\lec 
((\ell\la_q)^2 + \ell_t \la_q\de_q^\frac12)
\la_q^{-3\ga} \de_{q+1}^\frac32
\lec  {\la_q^\frac12}{\la_{q+1}^{-\frac12}}\de_q^\frac14 \de_{q+1}^\frac54,
\end{align*}
where $\mathcal{I}^q = [0,T]+\tau_{q-1}$.
Furthermore, we use $|\na^{N} \Phi(t+s,x;t)|\lec_M  \la_q^{N-1}$ for $N=1,2$, $s\in [-\ell_t,\ell_t]$, $t\ge 0$, and combine it with  \eqref{est.R}, \eqref{est.ph}, and \eqref{chain}, to obtain for $N=1,2$,
\begin{align*}
\norm{R_q-R_\ell}_N 
&\lec \norm{P_{> \ell^{-1}} R_q}_N
+ {\norm{R_q}_N + \norm{R_\ell}_N }
%&\lec \la_{q+1}^N  \ell^2\norm{R_q}_2 +\cmb{ \norm{R_\ell}_{N}+ \la_q^{N-1}\norm{\na R_\ell}_0}\\
\lec (\la_{q+1}^N  (\ell\la_q)^2 + {\la_q^N} )\la_q^{-3\ga} \de_{q+1}\\
&\lec \la_{q+1}^N {\la_q^\frac12}{\la_{q+1}^{-\frac12}}\de_q^\frac14 \de_{q+1}^\frac34
\end{align*}
and
\begin{align*}
\norm{\ph_q-\ph_\ell}_N
&\lec
\norm{P_{> \ell^{-1}} \ph_q}_N 
+ {\norm{\ph_q}_N + \norm{\ph_\ell}_N} 
%&\lec \la_{q+1}^N  \ell^2\norm{\ph_q}_2 
%+ (\ell^{-N} \norm{D_{t,\ell}\ph_q}_0 + \ell^{1-N} \norm{\na v_q}_0 \norm{\na \ph_q}_0) \\
\lec (\la_{q+1}^N  
(\ell\la_q)^2 +{\la_q^N})
\la_q^{-3\ga} \de_{q+1}^\frac32\\
&\lec  \la_{q+1}^N {\la_q^\frac12}{\la_{q+1}^{-\frac12}}\de_q^\frac14 \de_{q+1}^\frac54,
\end{align*}
provided that sufficiently small $b-1>0$ and sufficiently large $\la_0$.

Now, we consider the advective derivatives. We remark that for $F_\ell= P_{\le l^{-1}} F$, we can write
\begin{align*}
D_{t,\ell} (F-F_\ell) 
= D_{t,\ell} P_{> \ell^{-1}}F 
=  P_{> \ell^{-1}} D_{t,\ell} F + [v_\ell \cdot \na,  P_{> \ell^{-1}}] F.
\end{align*}
Then, we apply this to $F=v$ and $F=p$ and use \eqref{est.com2} to obtain
\begin{align*}
\norm{D_{t,\ell} (v_q-v_\ell)}_{N-1}
&\lec \norm{P_{>\ell^{-1}}D_{t,\ell} v_q}_{N-1} + \norm{[v_\ell \cdot \na,  P_{> \ell^{-1}}] v_q}_{N-1}\\
&\lec {\ell\norm{D_t v_q}_{N} + \ell\norm{(v-v_\ell)\cdot \na v}_{N}} + \ell^{2-N} \norm{\na v_q}_0^2
\lec \ell^{2-N} (\la_q \de_q^\frac 12)^2 
\end{align*}
and
\begin{align*}
\norm{D_{t,\ell} (p_q-p_\ell)}_{N-1}
&\lec \norm{P_{>\ell^{-1}}D_{t,\ell} p_q}_{N-1} + \norm{[v_\ell \cdot \na,  P_{> \ell^{-1}}] p_q}_{N-1}\\
&\lec {\ell\norm{D_t p_q}_{N} + \ell\norm{(v_q-v_\ell)\cdot \na p_q}_{N}} + \ell^{2-N} \norm{\na v}_0\norm{\na p_q}_0
\lec \ell^{2-N} (\la_q \de_q^\frac 12)^2 \de_q^\frac 12.
\end{align*} 
In a similar way, we have
\begin{equation}\begin{split}\label{est.hDt}
\norm{D_{t,\ell} P_{> \ell^{-1}}R_q}_{N-1}
\lec \norm{D_{t,\ell} R_q}_{N-1} + \ell^{2-N}\norm{\na v_q}_0 \norm{\na R_q}_0 
\lec {\la_{q+1}^{N-1} \la_q} \de_q^\frac12 \la_q^{-3\ga} \de_{q+1}\\
\norm{D_{t,\ell} P_{> \ell^{-1}}\ph}_{N-1}
\lec \norm{D_{t,\ell} \ph_q}_{N-1} 
+ \ell^{2-N}\norm{\na v_q}_0 \norm{\na \ph_q}_0 
\lec {\la_{q+1}^{N-1} \la_q} \de_q^\frac12 \la_q^{-3\ga} \de_{q+1}^\frac32.
\end{split}\end{equation}
Simply applying the triangle inequality, it can be easily shown that
\begin{equation}\begin{split}\label{est.Dtdif}
\norm{D_{t,\ell}[P_{\le \ell^{-1}} R_q -\rho_{\ell_t}\ast_{\Phi} P_{\le \ell^{-1}} R_q]}_{N-1}
\le 2\norm{D_{t,\ell} P_{\le \ell^{-1}} R_q}_{N-1} 
\lec \la_{q+1}^{N-1} \la_q\de_q^\frac12 \la_q^{-3\ga}\de_{q+1}\\
\norm{D_{t,\ell}[P_{\le \ell^{-1}} \ph_q -\rho_{\ell_t}\ast_{\Phi} P_{\le \ell^{-1}} \ph_q]}_{N-1}
\le 2\norm{D_{t,\ell} P_{\le \ell^{-1}} \ph_q}_{N-1} 
\lec \la_{q+1}^{N-1} \la_q\de_q^\frac12 \la_q^{-3\ga}\de_{q+1}^\frac32.
\end{split}\end{equation}
Combining \eqref{rep.dif}, \eqref{est.hDt}, and \eqref{est.Dtdif}, it follows that
\begin{align*}
\norm{D_{t,\ell}(R_q-R _\ell)}_{N-1} &\lec  \la_{q+1}^N\de_{q+1}^\frac12\cdot {\la_q^\frac12}{\la_{q+1}^{-\frac12}}\de_q^\frac14 \de_{q+1}^\frac34\\
\norm{D_{t,\ell}(\ph_q-\ph _\ell)}_{N-1} &\lec  \la_{q+1}^N\de_{q+1}^\frac12 \cdot{\la_q^\frac12}{\la_{q+1}^{-\frac12}}\de_q^\frac14 \de_{q+1}^\frac54.
\end{align*}
%for sufficiently large $\la_0$. 
\end{proof}
 
 \subsection{Quadratic commutator} We next deal with a quadratic commutator estimate, which is a version of the estimate in \cite{CoETi1994} leading to the proof of the positive part of the Onsager conjecture.
 
\begin{lem}\label{lem:est.Qvv} For any $N\geq 0$, $Q(v,v) = \na \cdot (v _\ell\otimes v _\ell - (v\otimes v) _\ell) $ satisfies 
\begin{align}\label{est.Qvv}
\norm{Q(v_q,v_q)}_N\lec \ell^{1-N} (\la_q\de_q^\frac12)^2, \quad
\norm{D_{t,\ell} Q(v_q,v_q)}_{N} \lec \ell^{-N} \de_q^{\frac12}(\la_q\de_q^\frac12)^2. 
\end{align}
{Here, we allow the implicit constants to be depending on $M$.}
\end{lem}

\begin{proof} In order to simplify our notation we drop the subscript $q$ fom $v_q$.
The estimate for $\norm{Q(v,v)}_N$ easily follows from \eqref{est.com}. To estimate $D_{t,\ell}Q(v,v)$, we first decompose $Q(v,v)$ into
\begin{align*}
Q(v,v)
=  (v _\ell-v)\cdot \na v _\ell + [v\cdot \na, P_{\le \ell^{-1}}] v.
\end{align*}
Recall that $\widehat{P_{\le \ell^{-1}} f}(\xi) =\widehat{P_{\le 2^J} f}(\xi) =  m \left(\frac{\xi}{2^J}\right) \widehat{f}(\xi)$ for some radial function $m \in \cal{S}$, where $J\in \N$ is the maximum number  satisfying $2^J\le \ell^{-1}$. For the convenience, we set $\widecheck{m}_{\ell}(x) = 2^{3J}\widecheck{m}(2^Jx)$. Then, by Poison summation formula, $P_{\le \ell^{-1}} f(x) = \int_{\R^3} f(x-y) \widecheck{m} _\ell(y)dy$ holds.  
Using this, the advective derivative of the commutator term can be written as follows,
\begin{align*}
D_{t,\ell}[v\cdot \na, P_{\le \ell^{-1}}] v
&=  (\pa_t +v _\ell(x)\cdot \na) \int ((v(x)-v(x-y))\cdot \na)v(x-y) \widecheck{m} _\ell(y) dy\\
&=  \int  ((D_{t,\ell}v(x)-D_{t,\ell}v(x-y)) \cdot  \na )v(x-y) \widecheck{m} _\ell(y) dy\\
&\quad-\int (v _\ell(x)-v _\ell(x-y))_a\na_a v_b(x-y)  \na_b v(x-y) \widecheck{m} _\ell(y) dy\\
&\quad+ \int ((v(x)-v(x-y))\cdot D_{t,\ell}\na) v(x-y) \widecheck{m} _\ell(y) dy\\
&\quad + \int (v(x)-v(x-y))_a (v _\ell(x)-v _\ell(x-y))_b (\pa_{ab}v)(x-y) \widecheck{m} _\ell(y) dy. 
\end{align*} 
Based on the decompositions, we use \eqref{est.vp}, \eqref{est.v.dif}, and $\norm{|y|^n\widecheck{m} _\ell}_{L^1(\R^3)}\lec \ell^n$, $n\geq 0$, to get
\begin{align*}
\norm{D_{t,\ell}Q(v,v)}_0
&\lec \norm{D_{t,\ell}(v-v _\ell)}_0  \norm{\na v_\ell}_0 + \norm{v-v _\ell}_0\norm{D_{t,\ell}\na v _\ell}_0 \\
&\quad+ \ell\norm{\na D_{t,\ell} v}_0\norm{\na v}_0
+\ell\norm{\na v}_0^3 + \ell \norm{\na v}_0 \norm{D_{t,\ell} \na v}_0 + \ell^2 \norm{\na v}_0^2 \norm{\na^2 v}_0 
\lec \ell (\la_q\de_q^\frac12)^3
\end{align*}
Here, we use $\norm{\na D_{t,\ell} v _\ell}_0 \leq \norm{\na D_{t,\ell} (v _\ell-v)}_0+\norm{\na D_{t} v}_0 +\norm{\na (((v-v _\ell)\cdot \na) v _\ell)}_0\lec (\la_q\de_q^\frac12)^2$, and
\begin{align*}
\norm{D_{t,\ell} \na v _\ell }_0 \leq \norm{\na D_{t,\ell} v _\ell}_0 + \norm{\na v}_0^2 \lec (\la_q\de_q^\frac12)^2.
\end{align*}
In the case of $N\geq 1$, we remark that $D_{t,\ell}Q(v,v)$ has frequency localized to $\lec \ell^{-1}$, so that the remaining estimates follows from the Bernstein inequality. Similarly, we also have
\begin{align}\label{est.Dtvl}
\norm{D_{t,\ell} \na v _\ell }_N \lec \ell^{-N} (\la_q\de_q^\frac12)^2.
\end{align} 
\end{proof}

\subsection{Estimates on the backward flow} Finally we address the estimates on the backward flow $\xi_I$.

\begin{lem}\label{lem:est.flow} {For every $b>1$ there exists $\La_0=\La_0(b)$ such that for $\la_0\geq \La_0$} the backward flow map $\xi_I$ satisfies the following estimates on 
the time interval $\cal{I}_m = [t_m - \frac 12 \tau_q , t_m + \frac 32\tau_q]\cap [0,T]+2\tau_q$
\begin{align}
&\norm{\I - \na \xi_I}_{C^0(\cal{I}_m \times \R^3)} \leq \frac 15 \label{est.flow1}\\
&\norm{D_{t,\ell}^s \na \xi_I}_{C^0(\cal{I}_m; C^N(\R^3))}\lec_{N,M} \ell^{-N} (\la_q \de_q^\frac12)^s \label{est.flow2} \\
&\norm{D_{t,\ell}^s (\na \xi_I)^{-1}}_{C^0(\cal{I}_m; C^N(\R^3))} 
\lec_{N,M} \ell^{-N} (\la_q \de_q^\frac12)^s, \label{est.flow3}
\end{align}
for any $N\geq 0$ and $s=0,1,2$. Note that the implicit constants in the inequalities are independent of the index $I=(m,n,f)$. {In particular,
\begin{align}\label{est.flow.indM}
\norm{\na \xi_I}_{C^0(\cal{I}_m; C^N(\R^3))} 
+ \norm{(\na \xi_I)^{-1}}_{C^0(\cal{I}_m; C^N(\R^3))}
\lec_N \ell^{-N}.
\end{align}
The implicit constant in this inequality is also independent of $M$.}
\end{lem}
\begin{proof} {First, we can find $\La_0(b)$ such that for any $\la_0\geq \La_0(b)$, $\tau_q\norm{\na v}_0\leq \frac 1{10}$ holds.} Then, \eqref{est.flow1} easily follows from \eqref{bf.Id}. Also, 
\begin{align}
&\norm{\na \xi_I}_{C^0(\cal{I}_m; C^N(\R^3))} \lec_N 1 + \tau_q
\norm{\na v_\ell}_N \lec 1+ \ell^{-N} \lec \ell^{-N} \label{est.flow01}
%&\norm{\I - \na \xi_I }_0 \leq 2 \tau_q \norm{\na \xi_I}_0 \norm{\na v_\ell}_0\leq 2\tau_q \norm{\na v_\ell}_0  \exp\left( 2 \norm{\na v_\ell}_0 \tau_q\right) \leq \frac 14, \nonumber
\end{align}
which follows $\norm{(\na \xi_I)^{-1}}_{C^0(\cal{I}_m; C^N(\R^3))} \lec_N \ell^{-N}$. 
Since we have
\begin{align*}
D_{t,\ell} \na \xi_I = -(\na\xi_I)(\na v_\ell), 
\quad D_{t,\ell}^2 \na \xi_I = (\na\xi_I)(\na v_\ell)^2-(\na \xi_I)D_{t,\ell}\na v_\ell,
\end{align*}
using \eqref{est.Dtvl} and \eqref{est.flow01}, $\norm{D_{t,\ell}^s \na \xi_I}_{C^0(\cal{I}_m; C^N(\R^3))} \lec_{N,M} \ell^{-N}(\la_q\de_q^\frac12)^s $ easily follows. Lastly, we have 
\begin{align*}
D_{t,\ell}(\na\xi_I)^{-1}
= \na v_\ell (\na \xi_I)^{-1}, \quad
D_{t,\ell}^2 (\na\xi_I)^{-1}
= D_{t,\ell}\na v_\ell (\na \xi_I)^{-1} + (\na v_\ell)^2 (\na\xi_I)^{-1}.
\end{align*}
Therefore, \eqref{est.flow3} can be obtained similarly. 
\end{proof}

\section{Choice of shifts: proof of Proposition \ref{p:supports}}

This section is perhaps the most crucial in our note, as it ensures the key property in the construction of $w_o$, namely the disjointness of the supports of the single blocks $w_I$ in its definition. 

\subsection{An elementary geometric observation} The basic tool is an elementary fact about closed geodesics in the three-dimensional torus. In order to state it efficiently we introduce the following notation. Given a vector $f\in \mathbb Z^3 \setminus \{0\}$ and a point $p\in \mathbb R^3$, we consider the line $\{\lambda f +p: \lambda \in \mathbb R\}\subset \mathbb R^3$. With a slight abuse of notation, we then denote by %$\ell_{f,p}$ 
$l_{f,p}$ the ``periodization'' of such line, namely
\begin{equation}\label{e:periodization}
l_{f,p}:= \{\lambda f +p: \lambda\in \mathbb R\} + 2\pi \mathbb Z^3\, ,
\end{equation}
and the corresponding closed geodesic in $\mathbb T^3$. Next, given two closed geodesics $s$ and $\sigma$ in the torus (or, equivalently, the periodizations of the corresponding lines in $\R^3$), we define
\[
\dis (s, \sigma):= \min \{|x-y|: x\in s, y\in \sigma\}
\]

 \begin{lem}\label{lem:no.int.line} 
Let $\cF$ be a given family of vectors in $\Z^3$ with a finite cardinality and $d$ be a given positive real number. 
Then, we can find $\eta=\eta(\cF, d)>0$ such that the following holds.
For any two sets of closed geodesics $\{s_f\} = \{l_{f,p_f}\}_{f\in \cF}$ and $\{\tilde{s}_f\} = \{l_{f,q_f}\}_{f\in \cF}$ satisfying
\begin{equation}\label{d:assumption}
\dis(s_f,s_{g}) \geq 2d, \quad \dis(\td s_f,\td s_{g}) \geq 2d, \quad
\forall f\neq g\in \cF,
\end{equation} 
we can always find $z\in \mathbb R^3$ with $|z|\le \frac 14 d$ such that the shifted geodesics $\{\bar{s}_f\}=\{l_{f, q_f+z}\}$ satisfy
\begin{equation}\label{e:eta-estimate}
\min \{\dis (s_f, \bar{s}_{g}) : f, g\in \cF\} \geq \eta.
\end{equation}
\end{lem}

\begin{proof}
The proof is based on a contradiction argument. Suppose that there exist a family $\cF$ with a finite cardinality and a positive number $d$ for which the statement fails no matter how small $\eta>0$ is chosen. Considering then for each $\eta=\frac{1}{k}$ with $k\in \mathbb N\setminus \{0\}$ a pair of families of geodesics which contradict the statement. We then achieve $2|\cF|$ sequences of closed geodesics 
$\{s_f^{(k)}\}$ and $\{\td s_g^{(k)}\}$, $f,g\in \cF$, $k\in \N$ such that:
\begin{itemize}
\item[(a)] For any $f\in \cF$, $s_f^{(k)}$ and $\td s_f^{(k)}$ are given by $l_{s, p_f (k)}, l_{s,q_f (k)}$ for some choice of vectors $p_f (k), q_f (k)$, which without loss of generality we can assume to satisfy the bounds $|p_f (k)|_\infty, |q_f (k)|_\infty\leq \pi$.
\item[(b)] The geodesics satisfy the bound
\[
\dis(s_f^{(k)}, s_{g}^{(k)}) \geq 2d, \quad \dis(\td s_f^{(k)}, \td s_{g}^{(k)}) \geq 2d, \quad
\forall f\neq g\in \cF\, .
\]
\item[(c)] For each $k\in \N$ 
\[
\max_{|z|\leq \frac{1}{4}}\, \left[ \min \left\{\dis(l_{f, p_f (k)+z}, l_{g, q_k (g)}) : f,g\in \cF\right\}\right] \leq \frac{1}{k}.
\]
\end{itemize}
Clearly, by extraction of a subsequence we can assume that all the sequences $\{p_f (k)\}$, $\{q_f (k)\}$ converge to some limits $p_f$ and $q_f$. We thus can consider the corresponding geodesics $s_f = l_{f, p_f}$ and $\td{s}_f = l_{f, q_f}$. The simple inequalities $\dis(s_f^{(k)}, s_f) \leq |p_f (k)-p_f|$ and $\dis (\tilde{s}_f^{(k)}, \tilde{s}_f)\leq |q_f^{(k)} - q_f|$ imply
\[
\lim_{k\to\infty} \left[\max_{f\in \cF}\dis(s_f^{(k)}, s_f) + \max_{f\in \cF}\dis(\td s_f^{(k)}, \td s_f)\right] = 0\, .   
\]
This implies that
\begin{align}
&\min_{f,g\in \cF} \dis(l_{f, p_f +z}, l_{g, q_g}) = 0, \qquad\forall |z|\leq \frac 14d,\label{e:contradiction}\\
\dis(l_{f,p_f}, l_{g,p_g}) &\geq 2d, \quad \dis(l_{f,q_f}, l_{g, q_g}) \geq 2d,  \quad
\forall f\neq g \in \cF. \label{mutual.dis}
\end{align}
Denote by $\cT$ the collection of $(f,g)\in \cF\times \cF$ such that $l_{f,p_f}\cap l_{g, q_g} \neq\emptyset$ and set
\begin{align*}
\de = \min \{ \dis (l_{f,p_f}, l_{g,q_g}): (f,g) \not \in \cT\} > 0\, .
\end{align*} 
Clearly, as long as $|z|<\frac {\de}2$, we have 
\begin{align}\label{no.T0}
\dis(l_{f,p_f+z}, l_{g, q_g})\geq \frac{\de}2 >0 \qquad \forall (f,g)\not \in \cT\, .
\end{align}
Consider next that, by \eqref{mutual.dis}, if $p\in l_{f, p_f} \cap l_{g,q_g}$, then $p$ cannot belong to any other geodesic $l_{f',p_f'}$ or $l_{g',q_g'}$. {Furthermore, for any $|z|\leq d/4$, we have $l_{f,p_f+z}\cap  l_{g', q_g'} =\emptyset$ for all $g'\neq g \in \cF$ and $ l_{f',p_f'+z}\cap l_{g, q_g}= \emptyset$ for all $f'\neq f \in \cF$.} Consider that, for $(f,g)\in \mathcal{T}$, either $l_{f,p_f}\cap l_{g,q_g}$ has finite cardinality (which occurs when $f$ and $g$ are not colinear) or else $l_{f, p_f} = l_{g,q_g}$ (which occurs when $f$ and $g$ are colinear). In both cases, let $L (f,g)$ be the linear space spanned by $\{f,g\}$ and observe that, if $z\in \mathbb S^2$ is any vector such that $\zeta\not\in L (f,g)$, then there is $\delta (f,g,\zeta)>0$ such that
\[
l_{f,p_f+\tau \zeta}\cap l_{g,q_g} = \emptyset \qquad \forall \tau\in (0, \delta (f,g,\zeta))\, .
\]
Since $\{L(f,g): (f,g)\in \mathcal{T}\}$ is a set with finite cardinality, it is clear that we can choose a vector $\zeta\in \mathbb S^2$ such that
\[
\zeta\not \in \bigcup_{(f,g)\in \mathcal{T}} L (f,g)\, . 
\]
Having fixed such a $\zeta$, if $z=\tau \zeta$ and $\tau$ is a sufficiently small positive number, we conclude from the considerations above that $|z|\leq \frac 14 d$ and
\[
l_{f,p_f+z} \cap l_{g,q_g} = \emptyset \qquad \forall f,g\in \mathcal{F}\, .
\]
Given that $\mathcal{F}$ is a finite set, the latter statement clearly contradicts \eqref{e:contradiction}. 
\end{proof}

\subsection{Proof of Proposition \ref{p:supports}: Set up} First of all we wish to determine the constant $\eta$ of the Proposition. Recall that a family $\mathcal{F} {= \cup_{j\in \Z^3_3} \cF^j}\subset \mathbb Z^3\setminus\{0\}$ has been fixed in Section \ref{ss:directions} and it consists of 270, pairwise noncolinear, elements. We first notice that we can choose a finite family $\{\bar{p}_f\}_{f\in \mathcal{F}}$ of shifts with the property that 
\[
l_{f,\bar{p}_f} \cap l_{g,\bar{p}_g} = \emptyset \qquad \forall f\neq g \in \mathcal{F}\, . 
\]
Hence we denote by $d_0$ the positive number 
\[
3 d_0 := \min \left\{\dis (l_{f,\bar p_f}, l_{g,\bar p_g}): f\neq g\in \mathcal{F}\right\} 
\]
and we apply Lemma \ref{lem:no.int.line} to $\mathcal{F}$ and $d=d_0$ to get the corresponding $\eta (\mathcal{F}, d_0)$. The resulting positive constant $\eta$ of Proposition \ref{p:supports} is then %$\eta = \min \{\eta (\mathcal{F}, d_0/2) $
$\eta = \min \{\eta (\mathcal{F}, d_0),d_0/2 \}$. Therefore we will now proceed to prove the claim of the Proposition. 

In order to simplify our notation we will use $\mu, \tau$ and $\lambda$ in place of $\mu_q, \tau_q$ and $\lambda_{q+1}$ and $v$ in place of $v_\ell$. 
We recall the following consequences of our choice of the parameters, which will play a fundamental role in the proof:
\begin{align}\label{con.par}
\mu^{-1} \ll \la \in \N, \quad \tau\norm{\na v}_{0} \leq \frac{1}{10}, \quad\mu\tau \norm{\na v}_{0}  \leq  \frac {\eta} {10\pi\la}.
\end{align}
(Here, $\norm{\cdot}_0 = \norm{\cdot}_{C([0,T]+\tau_{q-1}\times \T^3)}$.)
Next, the choice of the ${z}_{m,n}$ will be made inductively in the time discretization parameter $m$, so that
\begin{equation}\label{e:small_shift} 
|z_{m,n}|\leq \frac{d_0}{4}\, .
\end{equation}
Before coming to the specific choice, we argue that the condition \eqref{e:small_shift} guarantees 
\[
\supp (w_I) \cap \supp (w_J) = \emptyset \qquad \mbox{for all $I\neq J$ with $m_I = m_J = m$.}
\]
Indeed, observe that the last claim is implied by the disjointness of the supports of the functions $\chi_I (\xi_I) \psi_I (\lambda \xi_I)$
and $\chi_J (\xi_J) \psi_J (\lambda\xi_J)$. However $\xi_I=\xi_J = \xi_m$ and thus it suffices to show the disjointness of the supports
of $\chi_I (\cdot) \psi_I(\la \cdot)$, which depends only on the $x$ variables. Moreover observe that $\chi_I \chi_J=0$ if $|n_I -n_J|_\infty >1$. Hence it suffices to show 
\begin{equation}\label{e:disjoint_space_neighbors}
\supp (\psi_I) \cap \supp (\psi_J) = \emptyset \qquad \mbox{for all $I=(m,n,f)\neq (m,n',g)=J$ when $|n-n'|_\infty \leq 1$}.
\end{equation}  
Under such assumption, by \eqref{e:eta} 
\begin{align*}
\supp (\psi_I) &\subset B (l_{f,\bar{p}_f+z_{m,n}},\eta/10)\\
\supp (\psi_J) &\subset B(l_{g,\bar{p}_g+z_{m,n'}},\eta/10)
\end{align*}
However, since $I\neq J$ and either $[n]\neq [n']$ or $n=n'$, we necessarily have $f\neq g$. This means that $\dis (l_{f,\bar{p}_f}, l_{g, \bar{p}_g})\geq 2 d_0$ and thus that
\[
\dis (l_{f,\bar{p}_f} + z_{m,n}, l_{g, \bar{p}_g}+z_{m,n'}) \geq \frac{3}{2} d_0\, .
\]
Since $\eta\leq \frac{d_0}{2}$, the latter inequality ensures %\eqref{e:eta} 
\eqref{e:disjoint_space_neighbors}.

We are now left with the inductive specification of the extra shifts $z_{m,n}$. At the initial step $m=-2$, we just set $z_{m,n}=0$. Next note that $\supp (\theta_I) \cap \supp (\theta_J) = \emptyset $ when $|m_I - m_J|> 1$. Hence in the inductive step we fix $m$ and, assuming to have chosen $z_{m',n}$ for all $m'\leq m$ and all $n$, we wish to choose the ``next generation'' of $z_{m+1,n}$ such that
\[
\supp (\theta_I \chi_I (\xi_I) \psi_I (\lambda \xi_I)) \cap \supp (\theta_J \chi_J (\xi_J) \psi_J (\lambda \xi_J)) = \emptyset \quad \forall I=(m+1,n,f), \forall J=(m,n',g)\, .
\]  
We will not deal with condition \eqref{e:periodicity}, as it will be clear from the algorithm for choosing $z_{m+1,n}$ that it will be automatically satisfied. Recall next that $\supp (\theta_I) \cap \supp (\theta_J) \subset (\tau (m+1) - \frac{\tau}{8}, \tau (m+1)+\frac{\tau}{8}) = (t_{m+1}-\frac{\tau}{8}, t_{m+1} + \frac{\tau}{8})$, while $\xi_I = \xi_{m+1}$ and $\xi_J=\xi_m$ (defined in \eqref{def.bflow}). The above condition thus implied by
\begin{align}
&\supp (\chi_I (\xi_{m+1} (t, \cdot)) \psi_I (\lambda \xi_{m+1} (t, \cdot))) \cap \supp (\chi_J (\xi_{m} (t, \cdot)) \psi_J (\lambda \xi_{m} (t, \cdot)) = \emptyset \nonumber\\ 
&\quad \mbox{for all $t_{m+1}-\frac{\tau}{8} < t < t_{m+1} + \frac{\tau}{8}$, $I=(m+1,n,f)$ and $J=(m,n',g)$.} \label{e:inductive}
\end{align}
Moreover, the choice of each $z_{m+1,n}$ will be independent of the choice of other $z_{m+1,\bar n}$ except for the condition $z_{m+1,n}= z_{m+1, \bar n}$ when $\mu (n-\bar n) \in 2\pi\mathbb Z^3$, which will be enforced by the fact that we will only specify the choice when $n\in [0
, 2\pi \mu^{-1}]^3$. 

\subsection{Proof of Proposition \ref{p:supports}: conclusion} 

Let $m_{e}$ be the smallest integer satisfying $T\leq t_{m_e}$. From now on $n\in \Z^3$ and $m\in  \{-1,\cdots, m_e\} $ are thus fixed and we wish to show that for a suitable choice of $z_{m,n}$ satisfying \eqref{e:small_shift}, condition \eqref{e:inductive} holds.
We recall the flow map $\Phi_m$ introduced in \eqref{eqn.fflow} and observe that $\xi_m (t, \cdot) = [\Phi_m (t, \cdot)]^{-1}$. Moreover, by the semigroup property of flows, 
\begin{align}\label{rel.phi.m}
\Phi_{m+1} (s, \Phi_m (t_{m+1}, x)) = \Phi_m (s,x)\, .
\end{align}
and note that the latter can be equivalently written as 
\begin{align}\label{rel.phi.m2}
\Phi_{m+1} (s, y) = \Phi_m (s,\xi_{m} (t_{m+1}, y))\, . 
\end{align}
These relations imply that 
\begin{align}
\supp (\chi_J (\xi_m (t, \cdot))\psi_J (\lambda \xi_m (t, \cdot))) &= \Phi_m (t, \supp (\chi_J(\cdot)\psi_J (\lambda \cdot))\label{e:transport_support_1}\\
\supp (\chi_I (\xi_{m+1} (t, \cdot))\psi_{I} (\lambda \xi_{m+1} (t, \cdot))) &= \Phi_m (t, \xi_{m} (t_{m+1}, \supp (\chi_I(\cdot)\psi_I (\lambda\cdot)))\, .\label{e:transport_support_2}
\end{align}
In particular, \eqref{e:inductive} is reduced to show
\begin{align}
&\supp (\chi_J (\cdot) \psi_J (\lambda \cdot)) \cap \xi_{m} (t_{m+1}, \supp (\chi_I (\cdot) \psi_I (\lambda\cdot)) = \emptyset\nonumber\\
&\mbox{for all $I=(m+1,n,f)$ and $J=(m,n',g)$}.\label{e:inductive2}
\end{align}
Consider now $x_{m+1}:= \xi_{m} (t_{m+1},{2} \pi \mu n)$ and choose $\b n$ such that $x_{m+1}\in \overline{Q}_{\b n} = Q({2} \pi \mu \b n, \pi\mu)$. We claim that 
$\xi_{m} (t_{m+1}, \supp (\chi_I))$ cannot intersect $\supp (\chi_J)$ if $|n_J - \b n|>1$, which follows from the fact that $\Phi_m (t, \cdot)$ is a diffeomorphism for every $t$ and the claim that 
\begin{equation}\label{e:inclusion}
\xi_{m} (t_{m+1}, \supp (\chi_I)) \subset \xi_{m} (t_{m+1}, Q ({2} \pi \mu n, {\textstyle{\frac{{9} \pi\mu}{8}}})) \subset Q (x_{m+1}, {\textstyle{\frac{{13}\pi \mu}{8}}})\, .
\end{equation}
The first inclusion is obvious because by definition $\supp \chi_I \subset Q ({2} \pi \mu n, {\textstyle{\frac{{9} \pi\mu}{8}}})$. As for the second inclusion, recall the estimates
(cf. \eqref{bf.Id})
\begin{align*}
\|\nabla \xi_{m} (t_{m+1}, \cdot)\|_0 &\leq \exp (\tau_q \|\nabla v_q\|_0)\\
\|\nabla \xi_{m} (t_{m+1}, \cdot) - \I\|_0 &\leq \tau_q\|\nabla v_q\|_0 \exp (\tau_q\|\nabla v_q\|_0)\leq \frac{1}{5}\, .
\end{align*}
Thus, for every $x\in Q ({2} \pi \mu n, {\textstyle{\frac{{9} \pi\mu }{8}}})$ we can estimate
\begin{align*}
\xi_{m} (t_{m+1}, x) -x_{m+1} &= \xi_{m} (t_{m+1}, x) - \xi_{m} (t_{m+1}, {2} \pi \mu n)\\
&= \int_0^1 \nabla \xi_{m} (t_{m+1}, \lambda x + (1-\lambda) {2} \pi \mu n) \cdot (x-{2} \pi \mu n)\, d\lambda\\
&=x - {2} \pi \mu n + \int_0^1 (\nabla \xi_{m} (t_{m+1}, \lambda x + (1-\lambda) {2} \pi \mu n) - \I)\cdot  (x-{2} \pi \mu n)\, d\lambda\\
&=: x - {2} \pi \mu n + E\, .
\end{align*}
Hence, in order to show the second inclusion in \eqref{e:inclusion} it suffices to estimate
\[
|E|\leq \|\nabla \xi_{m} (t_{m+1}, \cdot) - \I\|_0 |x - {2} \pi \mu_n|\leq \frac{{9} \sqrt{3}}{40} \pi\mu \leq \frac{\pi\mu }{{2} }\, .  
\]
Given the argument above, \eqref{e:transport_support_1} and \eqref{e:transport_support_2} imply that
\[
\supp (\chi_J (\cdot)) \cap \supp \xi_{m+1} (t_m, \supp (\chi_I))=\emptyset
\]
for every $I=(m+1,n,f)$ and $J=(m,n',f)$ with $|n'-\b n|>1$. Therefore, in order to show 
\eqref{e:inductive2}, we will focus on the following remaining cases: 
\begin{align}
&\supp (\psi_J (\lambda, \cdot)) \cap \xi_{m} (t_{m+1}, \chi_I (\cdot) \psi_I (\lambda\cdot))=\emptyset\nonumber\\
&\mbox{for all $I=(m+1,n,f)$ and $J=(m,n',g)$ with $|n'-\bar n|_\infty\leq 1$.} \label{e:inductive2.5}
\end{align}
In particular note that $\{n':|n'-\bar n|_\infty\leq 1\}$ consists of 27 points in the integer lattice $\mathbb Z^3$, containing exactly one element for each equivalence class in $\mathbb Z_3^3$.

In order to deal with these last 270 cases of indices for $J$ together with the $10$ cases of possibility for $I$, observe first that, inserting $s=t_m$ in \eqref{rel.phi.m2}, we have $\Phi_{m+1} (t_m, \cdot) = \xi_m (t_{m+1}, \cdot)$ and thus \eqref{e:inductive2} becomes in fact
\begin{equation}\label{e:inductive3}
\supp (\psi_J (\lambda, \cdot)) \cap \Phi_{m+1} (t_{m}, \chi_I (\cdot) \psi_I (\lambda\cdot))=\emptyset
\end{equation}
Introduce now the ``frozen flow'' $\Psi$ given by
\begin{align*}
\begin{cases}
\pa_t \Psi (t,x) = v (t,\Phi_{m+1} (t, {2} \pi\mu n ))\\
\Psi (t_{m+1}, x) = x.
\end{cases}
\end{align*}
Observe that $\Psi (t,x)$ translates $x$ by some vector depending on time and so
\begin{align}\label{rep.Psi}
\Psi (t,x) = x + \int_{t_{m+1}}^t v (s,\Phi_{m+1} (s,{2} \pi\mu n)) ds= x +u (t).
\end{align}
Moreover, by definition $\Psi (t_m, {2} \pi \mu n) = x_{m+1}$, which means that, upon introducing $\b x := x_{m+1} - {2} \pi \mu n$, 
\begin{equation}\label{e:frozen=shift}
\Psi (t_m,x ) = x + \b x\, .
\end{equation}
Observe next that $\Phi_{m+1} (t_m, {2} \pi n \mu) = \Psi (t_m,{2}  \pi n \mu)$. Hence for $x\in %Q_n
Q({2} \pi \mu n, \frac {{9} \pi \mu}{8})\supset \supp (\chi_I)$ we can estimate
\begin{equation}\label{frozen.flow}
\begin{split}
|\Phi_{m+1} (t_{m},x) - \Psi (t_{m}, x) |_\infty
&\leq \int_{t_m}^{t_{m+1}} |\pa_s\Phi_{m+1} (s,x)- \pa_s\Psi (s,x)|_\infty ds\\
&\leq \int_{t_m}^{t_{m+1}} |v (s,\Phi_{m+1} (s,x))- v (s,\Phi_{m+1} (s,{2} \pi\mu n )) |_\infty ds \\
&\leq  \tau\norm{\na v}_0 \norm{\na \Phi_{m+1}}_{C^0([t_m,t_{m+1}]\times \T^3)} |x-{2} \pi\mu n |_\infty\\
&\leq \frac{{9}\pi\mu  \tau }{4} \|\na v\|_0 \leq \frac {\eta}{{4} \la}. 
\end{split}
\end{equation}
In particular we conclude that 
\[
|\Phi_{m+1} (t_m, x) - (x+ \b x)| \leq  \frac {\eta}{{4} \la}\, .
\]
If we introduce $\bar x := \lambda \b x$, we conclude that
\[
 \Phi_{m+1} (t_m, \supp (\chi_I (\cdot)\psi_I (\lambda \cdot) \subset %B (\supp (\psi_I (\lambda (\cdot - \bar x)), 5\lambda)^{-1} \eta)
B (\supp (\psi_I (\lambda \cdot - \bar x)), \eta/({4} \lambda) )
\, .
\]
Hence, %\eqref{e:inductive}
\eqref{e:inductive3} is satisfied if we have
\begin{equation}\label{e:inductive4} 
\supp (\psi_J) \cap B (\supp (\psi_I (\cdot - \bar x)), \eta/{4} ) =\emptyset
\end{equation} 
for the set 270 indices $J=(m,n', g)$ with $|\b n - n'|_\infty \leq 1$ and for the $10$ indices $I=(m,n,g)$. Observe now that, by the construction of the $\psi_J$'s, we know that:
\begin{itemize}
\item The map $\{J= (m,n', g): |n'-\b n|_\infty \leq 1, \}\ni J \mapsto g\in \mathcal{F}$ is a one-to-one map and
the map $\{I=(m+1,n,f)\}\ni I\to f\in \mathcal{F}$ is injective. 
\item We have that for each $g\in \mathcal{F}$ there is a point $\tilde{p}_g$ such that
\begin{equation}\label{e:support_10}
\supp (\psi_J) \subset B (l_{g,\tilde p_g}, \eta/10) 
\end{equation}
while for each $I= (m+1,n,f)$, if we let $p_f:= \bar{p}_f+\bar x$, then
\begin{equation}\label{e:support_11}
B  (\supp (\psi_I (\cdot - \bar x)), \eta/{4} ) \subset B (l_{f, p_f+z_{m+1,n}}, {7} \eta/{20} ).
\end{equation}
\item Finally
\begin{equation}\label{e:distant_lines}
\dis (l_{f,p_f}, l_{f',p_{f'}}) \geq 2 d_0 \qquad \dis (l_{g,\tilde p_g}, l_{g',\tilde p_{g'}}) \geq 2d_0 \qquad \forall f\neq f', g\neq g'\, .
\end{equation}
\end{itemize}
In particular, we are in the position to apply Lemma \ref{lem:no.int.line} and thus find a shift $z_{m+1,n}$ with $|z_{m+1,n}|\leq \frac{d_0}{4}$ such that
\begin{equation}\label{e:support_12}
\dis (l_{f,p_f+z_{m+1,n}}, l_{g, \tilde p_g}) \geq \eta \qquad \forall f,g\, .
\end{equation}
Clearly, \eqref{e:support_10}, \eqref{e:support_11} and \eqref{e:support_12} imply \eqref{e:inductive4} and thus completes the proof of the proposition.

\section{Estimates in the velocity correction}

The main point of this section is to get the estimates on the velocity correction. {In this section, we set $\norm{\cdot}_N = \norm{\cdot}_{C^0([0,T]+\tau_q; C^N(\T^3))}$.}

The following proposition provides the estimates for the perturbation $w$. 
\begin{prop}\label{p:velocity_correction_estimates} For $N=0,1,2$ and $s=0,1,2$, the following estimates hold for $w_o$, $w_c$, and $w=w_o+w_c$:
\begin{align}
& \tau_q^s \norm{D_{t,\ell}^s w_o}_{N}
{\lec_M} \la_{q+1}^{N} \de_{q+1}^\frac12 	\label{est.W}	\\
&\tau_q^s\norm{D_{t,\ell}^s w_c}_{N}
 {\lec_M} \la_{q+1}^{N}\frac { \de_{q+1}^\frac12}{\la_{q+1}\mu_q} 		\label{est.Wc} \\
&\tau_q^s\norm{D_{t,\ell}^s w}_{N}
 {\lec_M} \la_{q+1}^{N}\de_{q+1}^\frac12, \label{est.w}
\end{align}
where the implicit constants are independent of $s$, $N$, and $q$ in $w= w_{q+1}$. {Moreover, 
\begin{align}\label{est.w.indM}
\norm{w}_N \lec \la_{q+1}^N \de_{q+1}^\frac12,
\end{align}
where the implicit constant is additionally independent of $M$. }
\end{prop}

The latter estimates are in fact a simple consequence of estimates on the functions $b,c,d$ and $e$ defined in \eqref{coef.def} and \eqref{def.e}

\begin{lem}\label{lem:est.coe} For any $N\geq 0$ and $s=0,1,2$, the coefficients $b_{m,k}$, $c_{m,k}$, $d_{m,k}$, and $e_{m,k}$ defined by \eqref{coef.def} and \eqref{def.e} satisfy the following,
\begin{align}
\ta_q^s\norm{D_{t,\ell}^s b_{m,k}}_N
&{\lec_{N,M}} \ \mu_q^{-N}\max_I |\dot{b}_{I,k}| \label{est.b}\\
\ta_q^s\norm{D_{t,\ell}^s c_{m,k}}_N
&{\lec_{N,M}} \ \mu_q^{-N}\max_I |\dot{c}_{I,k}| \label{est.c}\\
\ta_q^s\norm{D_{t,\ell}^s d_{m,k}}_N
&{\lec_{N,M}}\  \mu_q^{-N}\max_I |\dot{d}_{I,k}| \label{est.d}\\
\ta_q^s\norm{D_{t,\ell}^s e_{m,k}}_N
&{\lec_{N,M}}\  \mu_q^{-N}\max_I |\dot{b}_{I,k}|.\label{est.e}
\end{align}
{Moreover, for $N=0,1,2$,
\begin{align}\label{est.be.indM}
\norm{b_{m,k}}_N + \norm{e_{m,k}}_N
\lec \mu_q^{-N} \max_I |\dot{b}_{I,k}|,
\end{align}
where the implicit constant is independent of $M$ and $N$.}
\end{lem}

\begin{rem}\label{r:Fourier_coefficients}
Observe that, by the definition of the respective coefficients, the moduli $|\dot{b}_{I,k}|$,  $|\dot{c}_{I,k}|$ and $|\dot{d}_{I,k}|$ just depend on the third component of the index $I=(m,n,f)$, since they involve the functions $\psi_f$, but not the ``shifts'' $z_{m,n}$. In particular, the set of their possible values is a finite number, independent of  $q$ and just depending on the collection of the family of functions $\psi_f$ and on the frequency $k$.
\end{rem}

\begin{proof}
First of all, it is easy to see that for any $s\geq 0$ and $N\geq 0$, 
\begin{equation}\begin{split}
\label{est.th.chi}
\norm{D_{t,\ell}^s\th_I}_{C^0(\R)} 
&= \norm{\pa_t^s\th_I}_{C^0(\R )} \lec_s \tau_q^{-s}, \\ 
\norm{\chi_I(\xi_I)}_{C^0(\cal{I}_m; C^N(\R^3))} 
&\lec_N \mu_q^{-N}, \  D_{t,\ell}^s [\chi_I(\xi_I)] = 0,
\end{split}\end{equation}
where $\cal{I}_m = [t_m - \frac 12\tau_q, t_m + \frac 32 \tau_q]$. Indeed, the estimate of $\chi_I(\xi_I)$ follows from \eqref{est.flow2}, Lemma \ref{lem:est.com}, and $\ell^{-1}\leq \mu_q^{-1}$. We remark that the implicit constants are independent of $I$. 

Recall that when $f \in \mathscr{I}_\ph$, 
\[
\ga_I=\frac{\la_q^{-\ga}\de_{q+1}^\frac12\Ga_I}{|\tilde{f}_I|^{\frac{2}{3}}} =
\frac{\la_q^{-\ga}\de_{q+1}^\frac12\Ga_I}{|\nabla \xi_I^{-1} f_I|^{\frac{2}{3}}}
\] 
for 
\[
\Ga_I(x) 
=\Gamma_{f_I}^{\frac{1}{3}}
(-2 \la_q^{3\ga}\de_{q+1}^{-\frac 32}(\na\xi_I)\ph_\ell)
\]
where $\Gamma_f$'s are the functions given by Lemma \ref{lem:geo2}.  
First it is easy to see that \eqref{est.flow3} implies 
\begin{align}\label{est4}
\norm{D_{t,\ell}^s [(\na\xi_I)^{-1}f_I]}_{C^0(\cal{I}_m; C^N(\R^3))}
{\lec_{N,M}} (\la_q \de_q^\frac 12)^s \ell^{-N}.
\end{align}
Also, using \eqref{est.flow2} and \eqref{est.mph},  
\begin{align}\label{est5}
\norm{D_{t,\ell}^s(2\la_q^{3\ga} \de_{q+1}^{-\frac32}(\na \xi_I)\ph_\ell)}_{C^0(\cal{I}_m; C^N(\R^3))}
{\lec_{N,M}} (\ell_t^{-s} + (\la_q \de_q^\frac12)^s) \ell^{-N}\lec \ta_q^{-s}\ell^{-N}.
\end{align}
Next, for any smooth functions $\Ga=\Ga(x)$ and $g=g(t,x)$ we have
\begin{equation}\begin{split}\label{formula0}
&\norm{D_{t,\ell} \Ga(g)}_{C^N_x} \lec \sum_{N_1+N_2=N} \norm{D_{t,\ell} g}_{C^{N_1}_x} \norm{(\na\Ga)(g)}_{C^{N_2}_x},\\
&\norm{D_{t,\ell}^2 \Ga(g)}_{C^N_x} 
\lec \sum_{N_1+N_2=N} \norm{D_{t,\ell}^2 g}_{C^{N_1}_x} \norm{(\na\Ga)(g)}_{C^{N_2}_x}
+ \norm{D_{t,\ell} g\otimes D_{t,\ell} g}_{C^{N_1}_x} \norm{(\na^2\Ga)(g)}_{C^{N_2}_x},
\end{split}\end{equation}
and therefore we obtain by Lemma \ref{est.com}
\begin{align*}
&\norm{D_{t,\ell}^s |(\na\xi_I)^{-1}f_I|^{-\frac 23}}_{C^0(\cal{I}_m; C^N(\R^3))}
{\lec_{N,M}} (\la_q \de_q^\frac 12)^s \ell^{-N}\\
&\norm{D_{t,\ell}^s[ (\Ga_j^\ph)^\frac13(-2\la_q^{3\ga} \de_{q+1}^{-\frac32}(\na \xi_I)\ph_\ell)]}_{C^0(\cal{I}_m; C^N(\R^3))}
{\lec_{N,M}} \ta_q^{-s} \ell^{-N}.  
\end{align*}
Here, we used \eqref{est4} and \eqref{est5} which we can apply thanks to the fact that $|(\na \xi_I)^{-1} f_I| \geq  \frac 34$ and $\Ga_{f_I}\geq 3$ (according to our choice of $N_0$ in applying Lemma \ref{lem:geo2}). 
Also the implicit constant in the second inequality can be chosen to be independent of $I$ because of the finite cardinality of the functions $f_I$. {On the other hand, in the case of $s=0$ and $N=0,1,2$, because of \eqref{est.ph} and \eqref{est.flow.indM}, the implicit constants in both inequalities can be chosen to be independent of $N$ and $M$.}
Therefore, it follows that, when $I \in \mathscr{I}_\ph$, 
\begin{align}\label{est.Ga.ph}
\norm{D_{t,\ell}^s \de_{q+1}^{-\frac12}\ga_I}_N {\lec_{N,M}} \ta_q^{-s} \ell^{-N}.
\end{align}
{In particular, for $N=0,1,2$,
\begin{align*}
\norm{\de_{q+1}^{-\frac12}\ga_I}_N \lec \ell^{-N}.
\end{align*}}

On the other hand, when $I\in \mathscr{I}_R$, recall that $\de_{q+1}^{-\frac 12} \ga_I = \Ga_I = \Ga_{f_I} (\I-\de_{q+1}^{-1}\mathcal{M}_I)$ for a finite collection of smooth functions $f_I$ chosen through Lemma \ref{lem:geo1}. First we obtain the estimate for $\mathcal{M}_I$, 
\begin{align*}
&\norm{\mathcal{M}_I}_{C^0(\cal{I}_m; C^N(\R^3))} \\
&\lec_N\ \de_{q+1}  \norm{(\na \xi_I)(\na \xi_I)^{\top} - \I}_{C^0(\cal{I}_m; C^N(\R^3))}  \\
&\ + \sum_{N_1+N_2+N_3=N}\norm{\na \xi_I}_{C^0(\cal{I}_m; C^{N_1}_x)} \norm{R_\ell}_{N_2} \norm{\na \xi_I}_{C^0(\cal{I}_m; C^{N_3}_x)}\\
&\ + \sum_{N_1+N_2+N_3+N_4=N} \sum_{{J}: f\in \cF_{{J}, \ph}} \norm{\na \xi_I}_{C^0(\cal{I}_m; C^{N_1}_x)} \norm{\chi_{J}^2(\xi_{J})}_{C^0(\cal{I}_m; C^{N_2}_x)} \norm{\ga_{J}^2}_{C^0(\cal{I}_m; C^{N_3}_x)}  \norm{\na \xi_I}_{C^0(\cal{I}_m; C^{N_4}_x)}\\
&\lec \de_{q+1} \mu_q^{-N}, 
\end{align*}
using \eqref{est.flow1}, \eqref{est.flow2}, \eqref{est.mR}, \eqref{est.th.chi}, and \eqref{est.Ga.ph}, 
Similarly, we have 
\begin{align*}
\norm{D_{t,\ell}^s \mathcal{M}_I}_N \lec_{N,M} \de_{q+1}\ta_q^{-s} \mu_q^{-N},
\end{align*}
{but $\norm{ \mathcal{M}_I}_N \lec \de_{q+1} \mu_q^{-N}$ for $N=0,1,2$.}
Then, \eqref{formula0} and Lemma \ref{lem:est.com} imply that when $f_I \in \cF_{I, R}$, for $s=0,1,2$ and $N\geq 0$, 
\begin{align}\label{est.Ga.R}
\norm{D_{t,\ell}^s \de_{q+1}^{-\frac 12} \ga_I}_N = \norm{D_{t,\ell}^s(\Ga_{f_I} (\I-\de_{q+1}^{-1}\mathcal{M}_I))}_N
\lec_{N,M} \ta_q^{-s} \mu_q^{-N}.
\end{align}
{In particular, for $N=0,1,2$, the implicit constant can be chosen to be independent of $M$ and $N$;
\begin{align*}
\norm{\de_{q+1}^{-\frac 12} \ga_I}_N 
\lec  \mu_q^{-N}.
\end{align*}}
Finally, recall the definition of $b_{m,k}$, $c_{m,k}$, $d_{m,k}$, and $e_{m,k}$. Then, the estimates \eqref{est.b}-\eqref{est.be.indM} follows from \eqref{est.th.chi}, \eqref{est.Ga.ph}, and \eqref{est.Ga.R}.
\end{proof}

\begin{proof}[Proof of Proposition \ref{p:velocity_correction_estimates}]
Using \eqref{est.b}, \eqref{est.e}, \eqref{est.flow2}, and \eqref{est.k}, we easily have the estimates $\la_{q+1}^{-N}\norm{w_o}_{N} \lec_{N} \de_{q+1}^\frac12$ and  $\la_{q+1}^{-N}\norm{w_c}_N \lec_{N} (\la_{q+1}\mu_q)^{-1}\de_{q+1}^\frac12$, recalling 
Remark \ref{r:Fourier_coefficients}. On the other hand, we observe that $D_{t,\ell} e^{i\la_{q+1}k\cdot \xi_I} = 0$ because of $D_{t,\ell} \xi_I =0$. Hence the remaining inequalities in \eqref{est.W} and \eqref{est.Wc} are obtained in a similar fashion. Finally, \eqref{est.w} follows from \eqref{est.W} and \eqref{est.Wc}. Note that all estimates used in the proof have implicit constants independent of $q$. Moreover, the finite cardinalities of the range of $N$ and $s$ make it possible to choose the implicit constants in \eqref{est.W}, \eqref{est.Wc}, and \eqref{est.w} independent of $N$ and $s$ too. {Furthermore, when $s=0$, we can also make the implicit constants independent of $M$.}
\end{proof}

\section{A microlocal lemma} 

We will need in the sequel a suitable extension of \cite[Lemma 4.1]{IsVi2015}, where we will use the notation
\[
\crF[f] (k) = \dint_{\T^3} f(x) e^{-ix \cdot k} dx,\quad
f (x) = \sum_{k\in \Z^3}\crF[f](k)   e^{ik \cdot x} 
\]   
for the Fourier series of periodic functions.

\begin{lem}[Microlocal Lemma]\label{mic} Let $T$ be a Fourier multiplier defined on $C^\infty(\T^3)$ by
\[
\crF[Th](k) = \mathfrak{m}(k)\crF[h](k), \quad \forall k\in \Z^3
\]
for some $m$ which has an extension in $\cal{S}(\R^3)$ (which for convenience we keep denoting by $m$). Then, for any $n_0\in \N$, $\la>0$, and any scalar functions $a$ and $\xi$ in $C^\infty(\T^3)$, 
$T(a e^{i\la \xi})$ can be decomposed as
\begin{equation*}
\begin{split}
T(a e^{i\la \xi}) 
=&\   \left[a \mathfrak{m}(\la\na \xi)
+\sum_{k=1}^{2{n_0}} C_{k}^\la(\xi,a): (\na^k \mathfrak{m})(\la \na\xi)
+\e_{n_0}(\xi,a)\right]e^{i\la \xi}
\end{split}
\end{equation*}
for some tensor-valued coefficient $C_{k}^\la(\xi,a)$ and a remainder $\e_{n_0}(\xi, a)$ which is specified in the following formula: 
\begin{equation}\label{def.eN}\begin{split}
\e_{n_0}(\xi, a)(x)
&= 
\sum_{\substack{n_1+n_2\\=n_0+1}} \frac{(-1)^{n_1}c_{n_1,n_2}}{n_0!} \\
&\cdot\int_0^1\int_{\R^3} \widecheck{\mathfrak{m}}(y) e^{-i\la \na\xi(x)\cdot y} ((y\cdot\na)^{n_1}a)(x-ry) e^{i\la Z[\xi](r)}\be_{n_2} [\xi](r) (1-r)^{n_0} dydr,
\end{split}
\end{equation}
where $c_{n_1,n_2}$ is a constant depending only on $n_1$ and $n_2$, and the function $\be_n[\xi]$ is
\begin{align*}
\be_{n} [\xi](r)
&=B_{n} (i\la Z'(r), i\la Z''(r), \cdots, i\la Z^{(n)}(r)),\\
Z{[\xi]}(r) &=Z{[\xi]}_{x,y}(r) = r\int_0^1(1-s) (y\cdot\na )^2 \xi(x-rsy) ds,
\end{align*}
with $B_n$ denoting the $n$th complete exponential Bell polynomial (cf. \eqref{e:Bell} for its definition).
\end{lem}

Before coming to its proof, we collect an important consequence on the operator $\mathcal{R}$.

\begin{cor}\label{cor.mic2} Let $N=0,1,2$ and $F = \sum_{k\in \Z^3\setminus\{0\}}\sum_{m\in \Z} a_{m,k}e^{i\la_{q+1}k\cdot\xi_m}$. Assume that a function $a_{m,k}$ fullfills the following requirements. 
\begin{enumerate}[(i)]
\item The support of $a_{m,k}$ satisfies $\supp(a_{m,k}) \subset (t_m - \frac 12\tau_q, t_m + \frac 32\tau_q)\times \R^3$. In particular, 
for $m$ and $m'$ neither same nor adjacent, we have
\begin{equation}\label{dis.amk}
\supp(a_{m,k}) \cap \supp(a_{m',k'}) = \emptyset, \quad \forall k, k' \in \Z^3\setminus\{0\}.
\end{equation}
\item For any $j\geq 0$ and $(m,k)\in \Z\times \Z^3$,  
\[
\norm{a_{m,k}}_j+ (\la_{q+1}\de_{q+1}^\frac 12)^{-1}\norm{D_{t,\ell} a_{m,k}}_j\lec_j  \mu_q^{-j} |\dot{a}_k|, \quad
\sum_k |k|^{n_0+2} |\dot{a}_k| \leq  a_F, 
\]
for some $a_F>0$, where $ n_0 = {\ceil{\frac{2b(2+\al)}{(b-1)(1-\al)}}}$ and {$\norm{\cdot}_{j} = \norm{\cdot}_{C(\mathcal{I}; C^j(\T^3))}$ on some time interval $\mathcal{I}\subset \R$. }
\end{enumerate}
Then, {for any $b>1$, we can find $\La_0(b)$ such that for any $\la_0\geq \La_0(b)$,} $\cR F$ satisfies the following inequalities: 
\begin{align*}
\norm{\cR F}_{N} \lec \la_{q+1}^{N-1} a_F, \quad
\norm{\as D_t \cR F}_{N-1} \lec \la_{q+1}^{N-1}\de_{q+1}^\frac12 a_F\,
\end{align*}
upon setting $\as D_t = \pa_t + v_{q+1} \cdot \na $.
\end{cor}

\subsection{Proof of Lemma \ref{mic}}
Recall that $\widecheck{m}$ is the inverse Fourier transform of $m$ in $\R^3$. By the Poisson summation formula, we have 
\begin{equation}\label{Tf.exp}
\begin{split}
T(ae^{i\la \xi})(x) 
=&\  e^{i\la \xi(x)} \int_{\R^3} \widecheck{\mathfrak{m}}(y) a(x-y) e^{i\la[\xi(x-y)-\xi(x)]} dy\\
=&\ e^{i\la \xi(x)} \int_{\R^3} \widecheck{\mathfrak{m}}(y) a(x) e^{-i\la\na \xi(x)\cdot y} dy\\
&+ e^{i\la \xi(x)} \int_{\R^3} \widecheck{\mathfrak{m}}(y)e^{-i\la\na \xi(x)\cdot y} 
(H_{x,y}(1) - H_{x,y}(0)) dy
\end{split}\end{equation}
where 
\[
H_{x,y}(r) 
= a(x-ry)e^{i\la Z_{x,y}(r)},
\quad
Z[\xi]_{x,y}(r) = r\int_0^1(1-s) (y\cdot \na)^2 \xi(x-rsy) ds. 
\]
Indeed, it follows from
\[
\xi(x-y) - \xi(x) +y\cdot \na \xi(x) =\int_0^1(1-s) (y\cdot\na)^2 \xi(x-sy) ds.
\]
In order to avoid a cumbersome notation, from now we drop the index $x$,$y$ in $H$ and $Z$ and the dependence on $\xi$ in $Z$. The decomposition of $T(ae^{i\la\xi})$ follows from Taylor's theorem applied to $H$ at $r=0$: 
\begin{equation}\label{H.exp}
H(1) - H(0) 
= \sum_{n=1}^{n_0} \frac{H^{(n)}(0) }{n!} + \int_0^1 \frac{H^{(n_0+1)}(r) }{n_0!} (1-r)^{n_0} dr.
\end{equation}
The $n$th derivative of $H$ can be computed by the Fa\`{a} di Bruno's formula,
\begin{equation}\label{der.H}
\begin{split}
H^{(n)}(r) 
&= \sum_{n_1+n_2=n}c_{n_1,n_2} \pa_r^{n_1} (a(x-ry)) \pa_r^{n_2} e^{i\la Z(r)} \\
&=\sum_{n_1+n_2=n} c_{n_1,n_2}(-1)^{n_1}((y\cdot\na)^{n_1}a)(x-ry) e^{i\la Z(r)}B_{n_2} (i\la Z', i\la Z'', \cdots, i\la Z^{(n_2)}) 
\end{split}\end{equation}
where $Z^{(n)}$ is the $n$th derivative of $Z$, and $B_n$ is the $n$th complete exponential Bell polynomial given by 
\begin{align}\label{e:Bell}
B_n(x_1,\dots,x_n)
=\sum_{k=1}^n B_{n,k}(x_1,x_2,\dots,x_{n-k+1}),
\end{align} 
where
\begin{align*}
B_{n,k}(x_1,x_2,\dots,x_{n-k+1}) =
 \sum\frac{n!}{j_1!  j_2! \cdots j_{n-k+1}!}
\left(\frac{x_1}{1!}\right)^{j_1}
\left(\frac{x_2}{2!}\right)^{j_2}\cdots
\left(\frac{x_{n-k+1}}{(n-k+1)!}\right)^{j_{n-k+1}},
\end{align*}
and the summation is taken over $\{j_k\}\subset \N\cup\{0\}$ satisfying
\begin{align}\label{sum.bel}
j_1 + j_2 + \cdots + j_{n-k+1} = k, \quad
j_1 + 2 j_2 + 3 j_3 + \cdots + (n-k+1)j_{n-k+1} = n.
\end{align}
Observe that $Z$ has the form $Z(r) = rZ_0(r)$, which follows
\begin{align}\label{der.Z}
Z^{(n)} (r) = nZ_0^{(n-1)}(r) + rZ_0^{(n)} (r), \quad Z^{(n)} (0) = nZ_0^{(n-1)}(0).
\end{align}
The $n$th derivative of $Z_0$ is
\begin{align}\label{der.Z0}
Z_0^{(n)}(r) = \int_0^1 (1-s) (-s)^n ((y\cdot\na)^{n+2}\xi)(x-rsy) ds.
\end{align}
In particular, 
\begin{equation*}
Z_0^{(n)}(0) = \int_0^1 (1-s) (-s)^n (y\cdot\na)^{n+2}\xi(x) ds = \frac{(-1)^n}{(n+1)(n+2)}  (y\cdot\na)^{n+2}\xi(x).
\end{equation*}
Therefore, we obtain 
\begin{align*}
H^{(n)}(0) 
&=\sum_{n_1+n_2=n}c_{n_1,n_2} (-1)^{n_1}(y\cdot\na)^{n_1}a(x) B_{n_2} (i\la Z_0(0), \cdots, i\la n_2  Z_0^{(n_2-1)}(0)) \\
&=  \sum_{n_1+n_2 =n}\sum_{k=1}^{n_2}  c_{n_1,c_2} (-1)^{n_1} \la^k y^{\otimes n+k}: \na^{n_1}a(x) \otimes 
\td\be_{n_2,k}[\xi](x),
\end{align*}
for some function $\td\be_{n_2,k}[\xi](x)$, and hence
\begin{align*}
\sum_{n=1}^{n_0} \int_{\R^3}\frac {H^{(n)}_{x,y}(0) }{n!} \widecheck{\mathfrak{m}}(y) e^{-i\la\na \xi(x) \cdot y} dy
= \sum_{k=1}^{2n_0} C_{k}^\la(\xi,a): (\na^k \mathfrak{m})(\la \na\xi)
\end{align*}
for some tensor-valued coefficient $C_{k}^\la(\xi,a)$.
Indeed, the factor $y^{\otimes n+k}$ gives the $(n+k)$th derivatives of $\mathfrak{m}$.  

Considering the remainder, we use \eqref{der.H} to define $\e_{n_0}(\xi,a)$ by 
\begin{equation*}\begin{split}
&\e_{n_0}(\xi, a)(x)\\
&=\frac 1{n_0!}\int_0^1\int_{\R^3} \widecheck{\mathfrak{m}}(y) e^{-i\la \na\xi(x)\cdot y} (H_{x,y})^{(n_0+1)}(r) dy (1-r)^{n_0}dr\\
&= 
\sum_{n_1+n_2=n_0+1} \frac{(-1)^{n_1}c_{n_1,n_2}}{n_0!}\\
&\quad\qquad \cdot
\int_0^1\int_{\R^3} \widecheck{\mathfrak{m}}(y) e^{-i\la \na\xi(x)\cdot y} ((y\cdot\na)^{n_1}a)(x-ry) e^{i\la Z(r)}\be_{n_2} [\xi](r) (1-r)^{n_0} dydr
\end{split}
\end{equation*}
where in the second equality
\begin{align*}
\be_{n_2} [\xi](r)
&=B_{n_2} (i\la Z'(r), i\la Z''(r), \cdots, i\la Z^{(n_2)}(r))\\
&=\sum_{k=1}^n (i\la)^k B_{n,k}(Z'(r), \cdots, Z^{n_2-k+1}(r)).
\end{align*}
This completes the proof. 

\subsection{Proof of Corollary \ref{cor.mic2}}
\

\noindent\texttt{Step 1.} Decomposition of $F$.

We first claim that $F$ can be written as 
\begin{align}\label{dec.mic}
F 
= \cP_{\gtrsim \la_{q+1}} \left(\sum_{m,k} a_{m,k} e^{i\la_{q+1} k\cdot \xi_m} \right) - \sum_{m,k} \e_{n_0}^{\la_{q+1}}(k\cdot \xi_m, a_{m,k}) e^{i\la_{q+1} k\cdot \xi_m},
\end{align}
where $\cP_{\gtrsim \la_{q+1}}$ is defined by
\[
\cP_{\gtrsim \la_{q+1}} = \sum_{2^j \geq \frac 38\la_{q+1}} P_{2^j}
\]
and
\[
\e_{n_0}^{\la_{q+1}} (k\cdot \xi_m, a_{m,k})
 = \sum_{2^j\geq \frac 38\la_{q+1}} \e_{n_0, j} (k\cdot \xi_m, a_{m,k}).
\]
The remainder $ \e_{n_0, j}(\xi,a)$ is obtained by applying Lemma \ref{mic} to $P_{2^j}$ and $n_0= {\ceil{\frac{2b(2+\al)}{(b-1)(1-\al)}}}$. 
To prove the claim, we first decompose $\cP_{\gtrsim \la_{q+1}} $ into 
\begin{align*}
\cP_{\gtrsim \la_{q+1}} 
= \sum_{\frac 38\la_{q+1}\le 2^j \le 2^J } P_{2^j} + P_{>2^J}
=: P_{\la_{q+1} \lec \cdot \le 2^J} + P_{>2^J}.
\end{align*}
and denote the multipliers of $P_{\la_{q+1} \lec \cdot \le 2^J}$ and $P_{2^j}$ by $m_{\le J}$ and $m_{j}$.   Indeed, they satisfy
\begin{align*}
\supp(m_{\le J}) &\subset B(0, 2^{J+1}), \quad
m_{\le J} = 1 \text{ on  } \overline{B(0,2^J)}\setminus B(0, \frac 34\la_{q+1}) \\
&\supp(m_{j}) \subset B(0, 2^{j+1})\setminus B(0,2^{j-1}).
\end{align*}
Then, since for any $k\in \Z^3\setminus\{0\}$, 
\[
\frac 34 |k|
\leq |\na (k\cdot \xi_m)| 
\leq \frac54 |k|, \quad  \text{on }\  \cal{I}_m\times \R^3
\] 
by $\norm{\I - \na \xi_m}_{C^0(\cal{I}_m\times \R^3)} \leq \frac 14$, where $\cal{I}_m = (t_m - \frac12\tau_q, t_m + \frac 32\tau_q)\cap \mathcal{I}$, we have for any $2^{J_k}\geq \frac 54\la_{q+1} |k|$, 
\[
m_{\le J_k}(\la_{q+1} \na (k\cdot \xi_m)) = 1, \quad 
m_{j}(\la_{q+1} \na (k\cdot \xi_m)) = 0, \quad \forall j>J_k, 
\]
on the support of $a_{m,k}$
Therefore, applying Lemma \ref{mic} to $P_{\la_{q+1}\lec \cdot \le 2^{J_k}}$ and $P_{2^j}$ for $j>J_k$, we obtain
\begin{align*}
P_{\la_{q+1}\lec \cdot \le 2^{J_k}}(a_{m,k} e^{i\la_{q+1} k\cdot \xi_m} ) 
&=   a_{m,k} e^{i\la_{q+1} k\cdot \xi_m}+\sum_{\frac 38\la_{q+1} \leq 2^j \leq 2^{J_k}} \e_{n_0,j}(k\cdot \xi_m,a_{m,k} )e^{i\la_{q+1} k\cdot\xi_m}\\
{P}_{2^j}(a_{m,k} e^{i\la_{q+1} k\cdot \xi_m} ) 
&= \e_{n_0, j}(k\cdot \xi_m,a_{m,k} )e^{i\la_{q+1} k\cdot\xi_m}.
\end{align*} 
Indeed, the remainder in the first equality follows from $m_{\leq J} =\sum_{\frac 38\la_{q+1} \leq 2^j \leq 2^{J_k}} m_j$. Summing them up in $j$ and reorganizing the terms, we have 
\begin{align}\label{dec.comp}
a_{m,k} e^{i\la_{q+1} k\cdot \xi_m}
=\cP_{\gtrsim \la_{q+1}} \left(a_{m,k} e^{i\la_{q+1} k\cdot \xi_m} \right) -  \e_{n_0}^{\la_{q+1}}(k\cdot \xi_m, a_{m,k}) e^{i\la_{q+1} k\cdot \xi_m}.
\end{align}
Taking summation again in $m$ and $k$, the desired decomposition follows. 

\

\noindent\texttt{Step 2.} The estimates for the remainder. 

We aim to obtain the following estimates. 
\begin{align}
\norm{\sum_{m,k}\e_{n_0}^{\la_{q+1}}(k\cdot \xi_m, a_{m,k})}_0 
&\lec_{n_0} (\la_{q+1}\mu_q)^{-(n_0+1)} a_F, \label{est.eN}
%\sum_{k=1}^{n_0}\sum_{i=0}^k \la_{q+1}^k \mu_q^{-(i+n_0)} \norm{|y|^{k+i+n_0} K}_{L^1_y(\R^3)}
\\
\norm{\sum_{m,k}D_{t,\ell} \e_{n_0}^{\la_{q+1}}(k\cdot \xi_m, a_{m,k})}_0 
&\lec_{n_0}   \la_{q+1}\de_{q+1}^\frac12 (\la_{q+1}\mu_q)^{-(n_0+1)} a_F. \label{est.DteN}
\end{align}
First, we remark the following relations between parameters;
\[
\la_q \leq \ell^{-1} \lec \mu_q^{-1}\lec \la_{q+1},\quad
\la_q\de_q^\frac12 \lec \la_{q+1}\de_{q+1}^\frac12.
\]
Recall the definition of $\e_{n_0,j}(\xi, a)$ from \eqref{def.eN}:
\begin{equation*}\begin{split}
&\e_{n_0,j}(k\cdot\xi_m, a_{m,k})(x)\\
&= 
\sum_{n_1+n_2=n_0+1} \frac{(-1)^{n_1}c_{n_1,n_2}}{n_0!}
\int_0^1\int_{\R^3} \widecheck{m_j}(y) e^{-i\la_{q+1} (y\cdot\na)(k\cdot\xi_m)(x)}(y\cdot\na)^{n_1}a_{m,k}(x-ry)  \\
&\qquad\qquad\qquad\qquad\qquad\qquad\qquad\qquad e^{i\la_{q+1} Z[k\cdot\xi_m](r)}\be_{n_2} [k\cdot\xi_m](r) (1-r)^{n_0} dydr.
\end{split}
\end{equation*}
It is obvious that
\begin{align*}
&|e^{-i\la_{q+1} (y\cdot\na)(k\cdot\xi_m)(x)}| \leq 1, \quad
|e^{i\la_{q+1} Z(r)} | \leq 1, \\
&|(y\cdot\na)^{n_1}a_{m,k}(x-ry)|\leq |y|^{n_1} \norm{a_{m,k}}_{n_1}\lec |y|^{n_1} \mu_q^{-n_1}|\dot{a}_k|.
\end{align*}
On the other hand, recall 
\begin{align}
\be_{n_2} [k\cdot\xi_m](r)
&=\sum_{l=1}^{n_2} (i\la_{q+1})^l B_{n,l}(Z'(r), \cdots, Z^{n_2-l+1}(r)), \nonumber\\
Z^{(n)} (r) &= Z[k\cdot\xi_m]^{(n)}(r) = nZ_0^{(n-1)}(r) + rZ_0^{(n)} (r), \nonumber\\
Z_0^{(n)}(r) &= \int_0^1 (1-s) (-s)^n (y\cdot\na)^{n+2}(k\cdot\xi_m)(x-rsy) ds. \label{defn.Z0n}
\end{align}
Using \eqref{est.flow2}, it can be easily seen that for $r\in [0,1]$ and $t\in \cal{I}_m$,
\begin{align*}
\norm{Z^{(n)}}_{C^0_x}
&\leq  \norm{nZ_0^{(n-1)}}_{C^0_x} + \norm{rZ_0^{(n)}}_{C^0_x}\\
&\lec_n |y|^{n+1}\norm{\na(k\cdot\xi_m)}_{{C^n_x}} 
+ |y|^{n+2}\norm{\na (k\cdot \xi_m)}_{{C^{n+1}_x}}
\lec |y|^{n+1}|k|\mu_q^{-n} ( 1+ |y|\mu_q^{-1}).
\end{align*}
Therefore, in the same range of $r$ and $t$, 
\begin{align*}
\norm{\be_{n_2}[k\cdot \xi_m](r)}_{C^0_{x}} 
&\lec_{ n_2} \sum \la_{q+1}^l \norm{ (Z')^{j_1} \cdots 
(Z^{(n_2-l+1)})^{j_{n_2-l+1}}}_{C^0_x}\\
&\lec_{ n_2} \sum_{l=1}^{n_2} \la_{q+1}^l |y|^l |k|^l (|y|\mu_q^{-1})^{n_2} \sum_{i=0}^l (|y|\mu_q^{-1})^i\\
&\lec_{ n_2} (|y|\mu_q^{-1})^{n_2} |k|^{n_0+1}\sum_{l=1}^{n_2} (\la_{q+1} |y|)^l \sum_{i=0}^l (|y|\mu_q^{-1})^i,
\end{align*}
where the summations in the first inequality is taken over 
\begin{align*}
&1\le l \le n_2, \quad  j_1+\cdots+j_{n_2-l+1}=l, \quad
j_1+\cdots +(n_2-l+1)j_{n_2-l+1}=n_2
\end{align*}
and the last inequality follows from $|k|^l\leq|k|^{n_2}\leq |k|^{n_0+1}$ for any $k\in \Z^3\setminus\{0\}$.   
Combining all estimates, the remainder $\e_{n_0,j}(k\cdot\xi_m, a_{m,k})$ satisfies
\begin{align*}
\norm{\e_{n_0,j}(k\cdot\xi_m, a_{m,k})}_{C^0(\cal{I}_m\times \R^3)}
&\lec_{n_0}  |k|^{n_0+1} |\dot{a}_k|
\sum_{l=1}^{n_0+1}\sum_{i=0}^l \la_{q+1}^l \mu_q^{-(i+n_0+1)} \norm{|y|^{l+i+n_0+1} \widecheck{m}_j}_{L^1_y(\R^3)},
\end{align*}
which implies that
\begin{align*}
\norm{\e_{n_0}^{\la_{q+1}}(k\cdot\xi_m, a_{m,k})}_{C^0(\cal{I}_m\times \R^3)}
&\lec \sum_{2^j\geq \frac 38 \la_{q+1}} 
\norm{\e_{n_0,j }(k\cdot\xi_m, a_{m,k})}_{C^0(\cal{I}_m\times \R^3)}\\
&\lec_{n_0} |k|^{n_0+1}|\dot{a}_k| (\la_{q+1}\mu_q)^{-(n_0+1)}.
\end{align*}
Therefore, {using $\supp(\e_{n_0}^{\la_{q+1}}(k\cdot\xi_m, a_{m,k}))\subset \supp(a_{m,k})$ and \eqref{dis.amk},} we obtain the first part of the claim:
\begin{align*}
\norm{\sum_{m,k}\e_{n_0}^{\la_{q+1}}(k\cdot\xi_m, a_{m,k})}_0
&\leq \sum_k\norm{ \sum_m \e_{n_0}^{\la_{q+1}}(k\cdot\xi_m, a_{m,k})}_{C^0}\\
&\leq {3}\sum_k \sup_m \norm{  \e_{n_0}^{\la_{q+1}}(k\cdot\xi_m, a_{m,k})}_{C^0(\cal{I}_m\times \R^3)}\\
&\lec_{n_0} \sum_{k\in \Z^3\setminus\{0\}} |k|^{n_0+1}|\dot{a}_k| (\la_{q+1}\mu_q)^{-(n_0+1)}
\lec (\la_{q+1}\mu_q)^{-(n_0+1)} a_F. 
\end{align*}
Indeed, we use $\norm{|y|^{n} \widecheck{m_j}}_{L^1_y(\R^3)}\lec (2^{-j})^n $ and $\mu_q^{-1}\leq \la_{q+1}$. 

To find the estimate for $D_{t,\ell} \e_{n_0}(k\cdot\xi_m,a_{m,k}) $, we compute the advective derivatives of each piece of the integrand of the integral in $\e_{n_0,j}(k\cdot\xi_m,a_{m,k})$ as follows;
\begin{align*}
D_{t,\ell} e^{-i\la_{q+1} (y\cdot\na) (k\cdot\xi_m)(x)} 
&= (\pa_t + v_\ell (x)\cdot \na_x)e^{-i\la_{q+1} \na(k \cdot\xi_m)(x)\cdot y} \\
&= -i\la_{q+1} D_{t,\ell}\na(k\cdot \xi_m)(x)\cdot y e^{-i\la_{q+1} \na (k\cdot\xi_m)(x)\cdot y},
\end{align*}
\begin{align*}
D_{t,\ell} [(y\cdot\na)^{n_1} a_{m,k}(x-ry) ]
=\ & (v_\ell(x)-v_\ell(x-ry))\cdot \na ((y\cdot\na)^{n_1} a_{m,k}(x-ry))\\
&+ (D_{t,\ell}(y\cdot\na)^{n_1} a_{m,k})(x-ry),
\end{align*}
and
\begin{align*}
&(D_{t,\ell}e^{i\la_{q+1} Z_{x,y}(r)} )
= i\la_{q+1} (\pa_t+v_\ell (x)\cdot \na_x)Z_{x,y}(r)e^{i\la_{q+1} Z_{x,y}(r)} \\
&\  = i\la_{q+1} r \int_0^1 (1-s) y\otimes y: (D_{t,\ell} \na^2(k\cdot \xi_m))(x-ry)ds e^{i\la_{q+1} Z_{x,y}(r)}\\
&\quad +
i\la_{q+1} r \int_0^1 (1-s) y\otimes y:  [(v_\ell(x)-v_\ell(x-y))\cdot\na] \na^2 (k\cdot\xi_m)(x-ry)ds e^{i\la_{q+1} Z_{x,y}(r)}
\end{align*}
For the last piece, we recall $\be_{n_2}[k\cdot \xi_m]$;
\begin{align*}
D_{t,\ell} \be_{n_2}[k\cdot \xi_m]
= D_{t,\ell}\left[ \sum_{l=1}^{n_2} (i\la_{q+1})^l \sum_{j_i} 
\frac{n_2!}{j_1!\cdots j_{n_2-l+1}!} \left(\frac{Z'_{x,y}(r)}{1!}\right)^{j_1}
% \left(\frac{Z''_{x,y}(r)}{2!}\right)^{j_2} 
\cdots 
\left(\frac{Z^{(n_2-k+1)}_{x,y}(r)}{(n_2-l+1)!}\right)^{j_{n_2-l+1}} \right]. 
\end{align*}
First, using \eqref{est.flow2}, we have on $\cal{I}_m$
\begin{align*}
&\norm{D_{t,\ell} \na (k\cdot\xi_m)}_{C^0_x} \lec \la_q\de_q^\frac12 |k| \\
&\norm{D_{t,\ell} \na^2(k\cdot \xi_m)}_{C^0_x}
\leq |k|\norm{\na D_{t,\ell} \na \xi_m}_{C^0_x}
+ |k|\norm{\na v}_0 \norm{\na^2 \xi_m}_{C^0_x}
\lec \la_q\de_q^\frac12 \mu_q^{-1}|k|,
\end{align*}
which follows that for $r\in [0,1]$ and $t\in \cal{I}_m$, 
\begin{equation}\begin{split}\label{est.6}
&\norm{D_{t,\ell} e^{-i \la_{q+1}(y\cdot\na)(k\cdot \xi_m)}}_{C^0_x}
\lec \la_{q+1}|y|\la_q\de_q^\frac12 |k|\\
&\norm{D_{t,\ell} e^{i\la_{q+1} Z_{x,y}(r)}}_{C^0_x}
\lec \la_{q+1}\mu_q^{-1}  \la_q\de_q^\frac12 |y|^2|k|(1+|y|\ell^{-1}).  
\end{split}\end{equation}
Also, for any smooth function $g$, we can write 
\begin{align*}
D_{t,\ell}(y\cdot \na)^{n_1} g = (y\cdot \na)^{n_1} (D_{t,\ell}g) + [D_{t,\ell}, (y\cdot \na)^{n_1}] g.
\end{align*}
Since the commutator term $[D_{t,\ell}, (y\cdot \na)^{n_1}]g$ has a representation,
\begin{align*}
[D_{t,\ell}, (y\cdot \na)^{n_1}]g = [v_\ell\cdot \na,  (y\cdot \na)^{n_1}]g 
= \sum_{\substack{ m_1+m_2=n_1\\ 1\le m_1\le n_1}} C_{m_1,m_2} [(y\cdot \na)^{m_1} v_\ell \cdot \na] (y\cdot\na)^{m_2} g,
\end{align*}
we have
\begin{align*}
\norm{[D_{t,\ell}, (y\cdot \na)^{n_1}]g}_0
&\lec |y|^{n_1} \sum_{\substack{ m_1+m_2=n_1\\ 1\le m_1\le n_1}}  \norm{\na^{m_1} v_\ell}_{0} \norm{\na^{m_2+1} g}_0,
\end{align*}
and hence
\begin{align}\label{est.Dtlna}
\norm{D_{t,\ell}(y\cdot \na)^{n_1}g}_0
\lec_{n_1} |y|^{n_1} \left[\norm{\na^{n_1} D_{t,\ell}g}_0
+ \sum_{\substack{ m_1+m_2=n_1\\ 1\le m_1\le n_1}}  \norm{\na^{m_1} v_\ell}_{0} \norm{\na^{m_2+1} g}_0\right].
\end{align}
In particular, we have 
\begin{align}
&\norm{D_{t,\ell}(y\cdot \na)^{n_1}a_{m,k}}_0\lec_{n_1}  \la_{q+1}\de_{q+1}^\frac12 |y|^{n_1} \mu_q^{-n_1}|\dot{a}_k| \label{est.7}\\
&\norm{D_{t,\ell}(y\cdot \na)^{n+2}(k\cdot\xi_m)}_{C^0(\cal{I}_m\times \R^3)} \lec |y|^{n+2}\la_q\de_q^\frac12 \mu_q^{-(n+1)} |k|.  \nonumber
\end{align}
Here, the second inequality uses $D_{t,\ell} \xi_m=0$.
This suggests that $Z_0$ in \eqref{defn.Z0n} satisfies for $r\in [0,1]$ and $t\in \cal{I}_m$
\begin{align*}
\norm{D_{t,\ell} Z_0^{(n)}}_{C^0_{x}} 
&\leq\Norm{\int_0^1(1-s)(-s)^{n} (D_{t,\ell})_x [(y\cdot\na)^{n+2} (k\cdot\xi_m)(x-rsy)] ds}_{C^0_{x}}\\
&\lec
 |k|\sup_{s\in [0,1]}\norm{v_\ell-v_\ell(\cdot-rsy)}_{C^0_{x}}\norm{\na (y\cdot\na)^{n+2} \xi_m}_{C^0_{x}} + |k| \norm{D_{t,\ell} (y\cdot\na)^{n+2} \xi_m}_{C^0_{x}}\\
 &\lec |k||y|^{n+3}\norm{\na v}_{C^0_{x}} \norm{\na^{n+3}\xi_m}_{C^0_{x}}  + |k| \norm{D_{t,\ell} (y\cdot\na)^{n+2} \xi_m}_{C^0_{x}}\\
 &\lec  |y|^{n+1} \la_q\de_q^\frac12 \mu_q^{-n} 
 (|y|^2\mu_q^{-2}+|y|\mu_q^{-1})|k|.
\end{align*}
Therefore, the advective derivative of $Z_{x,y}^{(n)}$ with  $k\cdot \xi_m$ can be estimated as
\begin{equation}\begin{split}\label{est.8}
\norm{D_{t,\ell} Z^{(n)}}_{C^0_x} 
&\lec n \norm{D_{t,\ell} Z_0^{(n-1)}}_{C^0_x}  + r\norm{D_{t,\ell} Z_0^{(n)}}_{C^0_x} \\
&\lec |y|^{n+1} \la_q\de_q^\frac12 \mu_q^{-n} (1+ |y|\mu_q^{-1}+ |y|^2\mu_q^{-2})|k|,
\end{split}\end{equation}
in the same range of $r$ and $t$. Combining the estimates \eqref{est.6}, \eqref{est.7}, and \eqref{est.8}, we can see that $D_{t,\ell}$ generates a factor whose value is bounded by $\la_{q+1}\de_{q+1}^\frac12|k| (1+(\la_{q+1}|y|)^3)$. More precisely, for any $r\in [0,1]$, if  $\norm{F}_{C^0(\cal{I}_m\times \R^3)} \lec b_F$, then $\norm{D_{t,\ell} F}_{C^0(\cal{I}_m\times \R^3)} \lec \la_{q+1}\de_{q+1}^\frac12 |k|(1+(\la_{q+1}|y|)^3) b_F$,  where the possible $F$ are $ e^{-i\la_{q+1}(y\cdot \na) (k\cdot\xi_m)(x)}$, $(y\cdot \na)^{n_1}a_{m,k}(x-ry)$, $e^{i\la_{q+1} Z_{x,y}(r)}$, $Z^{(n)}$, and $\be_{n_2}[k\cdot\xi_m]$. As a result,
\begin{align*}
\norm{\sum_{k,m}D_{t,\ell} \e_{n_0}^{\la_{q+1}}(k\cdot\xi_m,a_{m,k})}_0 
&\lec_{n_0} \la_{q+1}\de_{q+1}^\frac12\sum_{k\in \Z^3\setminus\{0\}}
|k|^{n_0+2} |\dot{a}_k| (\la_{q+1}\mu_q)^{-(n_0+1)} \\
&\lec  \la_{q+1}\de_{q+1}^\frac12 (\la_{q+1}\mu_q)^{-(n_0+1)} a_F.
\end{align*}

\

\noindent\texttt{Step 3.} The estimates for $\cR F$. 

We first note several things; The parameters satisfy the additional relations
\[
\ell\la_q \de_{q}^\frac12 \leq \de_{q+1}^\frac12,  \quad \la_{q+1}^{-1}\leq \de_{q+1}^{\frac 12},\quad
\la_{q+1}^2(\la_{q+1}\mu_q)^{-(n_0+1)}
\lec 1
\]
by the choice of $n_0$. {Furthermore, for any $b>1$, we can find $\La_0(b)$ such that for any $\la_0\geq \La_0(b)$, we have $\ell^{-1}\la_{q+1}^{-1}\leq \frac 3{64}$.} As a consequence of the decomposition \eqref{dec.mic}, we have the frequency localization of the remainder part of $F$, 
\begin{align}\label{rem.low}
\cP_{\lec \la_{q+1}} F := F - \cP_{\gtrsim \la_{q+1}}F
=- \sum_{k,m}\e_{n_0}^{\la_{q+1}}(k\cdot \xi_m, a_{m,k}) e^{i\la_{q+1} k\cdot \xi_m}.
\end{align}
Lastly, using the assumptions on $a_{m,k}$ and $\xi_m$, $F= \sum_{k\in \Z^3\setminus\{0\}} \sum_{m\in \Z} a_{m,k} e^{i\la_{q+1}k\cdot \xi_m}$ satisfies
\begin{align*}
\norm{F}_{N}  \lec \la_{q+1}^N a_F, \quad 
\norm{D_{t,\ell} F}_{N-1}\lec \la_{q+1}^{N} \de_{q+1}^\frac12 a_F,\quad  N=0,1,2.
\end{align*}
Then, we recall the decomposition \eqref{dec.mic} of $F$ and use the Bernstein inequality to get
\begin{align*}
\norm{\cal{R} F}_N
&\leq\norm{\cal{R} \cP_{\gtrsim \la_{q+1}} F}_N +  \norm{\cal{R}\sum_{k,m} \e_{n_0}^{\la_{q+1}}(k\cdot \xi_m, a_{m,k})e^{i\la_{q+1} k\cdot \xi_m}}_N\\
&\lec \frac 1{\la_{q+1}}\norm{ F}_N + {\la_{q+1}^N}\norm{\sum_{k,m} \e_{n_0}^{\la_{q+1}}(k\cdot \xi_m, a_{m,k})}_0
\lec_{n_0} \la_{q+1}^{N-1} a_F
\end{align*}
for $N=0,1,2$. Indeed, the second and third inequalities follows from a crude estimate $\norm{\cal{R} f}_0 \lec \norm{f}_0$ and \eqref{est.eN}. 

To estimate $\as D_t \cR F$, we use the decomposition
\begin{align*}
\as D_t \cR F 
=  \cR D_{t,\ell} F + [v_\ell\cdot \na, \cR]F + ((v_q-v_\ell)+ w)\cdot \na \cR F.
\end{align*}
{For the remaining part, we only consider $N=1,2$. }Since $D_{t,\ell} F = \sum_{k,m} D_{t,\ell}a_{m,k} e^{i\la_{q+1} k \cdot\xi_m}$, the first term can be estimated as above,
\begin{align*}
\norm{\cR D_{t,\ell} F}_{N-1} 
\lec \frac 1{\la_{q+1}}\norm{D_{t,\ell}F}_{N-1} + \la_{q+1}^{N-1} \norm{\sum_{m,k} \e_{n_0}^{\la_{q+1}}(k\cdot \xi_m, D_{t,\ell} a_{m,k})}_0 \lec \la_{q+1}^{N-1}\de_{q+1}^\frac 12 a_F.
\end{align*} 
Also, we recall $\norm{w}_{N-1} + \norm{v-v_\ell}_{N-1} \lec \la_{q+1}^{N-1} \de_{q+1}^\frac12$, so that
\begin{align*}
\norm{((v-v_\ell)+ w)\cdot \na \cR F}_{N-1}
&\lec \la_{q+1}^{N-1} \de_{q+1}^\frac 12 a_F.
\end{align*}
Regarding to the commutator term,  we plug the decomposition \eqref{dec.mic}, 
\begin{align*}
[v_\ell\cdot \na, \cR]F
=[v_\ell\cdot \na, \cR]{\cP}_{\gtrsim \la_{q+1}} F 
- [v_\ell\cdot \na, \cR] \sum_{m,k} \e_{n_0}^{\la_{q+1}}(k\cdot \xi_m, a_{m,k}) e^{i\la_{q+1}k\cdot \xi_m}. 
\end{align*} 
Since $v_\ell$ and $\cP_{\lec \la_{q+1}}F= \e_{n_0}^{\la_{q+1}}(k\cdot \xi_m, a_{m,k}) e^{i\la_{q+1}k\cdot \xi_m}$ have frequencies localized to $\lec \la_{q+1}$,  
\begin{align*}
\norm{[v_\ell\cdot \na, \cR]\cP_{\lec \la_{q+1}}F  }_{N-1}
&\lec \la_{q+1}^{N-1}\norm{[v_\ell\cdot \na, \cR]\cP_{\lec \la_{q+1}}F  }_{0}
\lec \la_{q+1}^{N-1}\norm{v_\ell}_0 \norm{\na \cP_{\lec \la_{q+1}}F  }_{0}\\
&\lec \la_{q+1}^{N}\norm{\sum_{m,k} \e_{n_0}^{\la_{q+1}}(k\cdot \xi_m, a_{m,k})}_0 \lec \la_{q+1}^{N-1}\de_{q+1}^\frac 12 a_F.
\end{align*}
Indeed, the second inequality follows from $\norm{\cR g}_0 \lec \norm{g}_0$. To estimate the other commutator term, we consider $F_j = P_{2^j}F$ for $2^j\geq \frac 38\la_{q+1}$, 
\begin{align}
-&[v_\ell\cdot \na , \cR] F_j (x)
= \sum_{k,\eta\in \Z^3} (\crF[\cR](k)-\crF[\cR](\eta)) i\eta \cdot \crF[v_\ell](k-\eta) \crF[F_j](\eta) e^{ik\cdot x} \nonumber\\
&= \sum_{k,\eta\in \Z^3} \sum_{l=1}^{l_0} \frac 1{l!} [(k-\eta)\cdot \na]^l \crF[\cR](\eta) i\eta \cdot \crF[{v_\ell}](k-\eta) \crF[F_j](\eta) e^{ik\cdot x} \label{eqn.com1}\\
&\ +\frac 1{l_0!}\int_0^1 [(k-\eta)\cdot \na]^{l_0+1} \crF[\cR](\eta+\sigma(k-\eta))(1-\sigma)^{l_0} d\sigma i\eta\crF[{v_\ell}](k-\eta)\crF[F_j](\eta) e^{ik\cdot x} \label{eqn.com2},
\end{align}
where $\crF[\cR f](k) = \crF[\cR ](k)\crF[f](k)$ and $l_0> 2$ is an integer satisfying $\la_{q+1}^3(\ell\la_{q+1})^{-l_0} \leq 1$. 
The estimate for \eqref{eqn.com1} follows from
\begin{align*}
\norm{\sum_{2^j\geq \frac 38\la_{q+1}}\eqref{eqn.com1}}_{N-1}
&\leq  \sum_{l=1}^{l_0} \frac 1{l!}\norm{\na^l v_\ell :\na \cR_l \cP_{\gtrsim \la_{q+1}}F }_{N-1}\\
&\lec \sum_{N_1+N_2=N-1}  \sum_{l=1}^{l_0}
\frac 1{l!}\norm{\na^l v_\ell}_{N_1} \norm{\na \cR_l \cP_{\gtrsim \la_{q+1}}F }_{N_2}\\
&\lec \sum_{N_1+N_2=N-1}  
 \ell \norm{\na v_\ell}_{N_1}\norm{F}_{N_2}
\lec \la_{q+1}^{N-1}  \de_{q+1}^\frac12  a_F,
\end{align*}
where the operator $\cR_l$ has a Fourier multiplier defined by $\crF[\cR_lg](\eta) = \na^{l}_\eta \crF[\cR](\eta)\crF[g](\eta)$ for any $\eta\in \Z^3$. 
To estimate \eqref{eqn.com2}, we use {$|k-\eta|\leq 2 l^{-1} \leq \frac 3{32} \la_{q+1} \leq\frac 12 |\eta|$} and hence $|k|\lec |\eta|$ to get
\begin{align*}
\norm{\na^{N-1}\sum_{2^j\geq \frac 38\la_{q+1}}\eqref{eqn.com2}}_0 
&\lec \sum_j \sum_{k,\eta} 
|k|^{N-1} \frac{|k-\eta|^{l_0+1}}{|\eta|^{l_0+1}}|\crF[ v_\ell](k-\eta)| |\crF[F_j](\eta)|\\
&\lec \sum_j \sum_{k,\eta}  |k-\eta|^{l_0}|\crF[\na v_\ell](k-\eta)| |\eta|^{N-2-l_0} {|\crF[F_j](\eta)|}\\
&\lec \sum_j \frac{\ell^{-l_0}}{(2^j)^{l_0+1}} \sum_{|k|\lec 2^j}  \norm{\na v}_{L^2(\T^3)} \norm{\na^{N-1}F_j}_{L^2(\T^3)}\\
&\lec \left(\frac {\ell^{-1}}{\la_{q+1}}\right)^{l_0} \la_{q+1}^2\norm{\na v}_0 \norm{F}_{N-1}
\lec \la_{q+1}^{N-1}  \de_{q+1}^\frac 12   a_F. 
\end{align*}
Indeed, the last inequality follows from the choice of $l_0$. To summarize, we obtain
\begin{align}\label{est.DtcR}
\norm{[v_\ell \cdot \na, \cal{R}]\cP_{\gtrsim \la_{q+1}} F}_{N-1}
\lec \la_{q+1}^{N-1} \de_{q+1}^\frac12 a_F,
\end{align}
and combine with all other estimates to get the desired one $\norm{\as D_t \cal{R} F}_{N-1}\lec \la_{q+1}^{N-1}\de_{q+1}^\frac12 a_F$ for $N=1,2$.

\section{Estimates on the Reynolds stress}

In this section, we obtain the relevant estimates for the new Reynolds stress and its new advective derivative $\as D_t R_{q+1} = \partial_t R_{q+1} + v_{q+1}\cdot \nabla R_{q+1}$, summarized in the following proposition. For technical reasons it is however preferable to estimate rather $R_{q+1} - \frac{2}{3} \varrho \I$, as indeed the estimates on the function $\varrho (t)$ are akin to those for the new current, which will be detailed in the next section. For the remaining sections, we set $\norm{\cdot}_N = \norm{\cdot}_{C^0([0,T]+\tau_q; C^N(\T^3))}$.

\begin{prop}\label{p:Reynolds}
{There exists  $\bar{b}(\al)>1$ with the following property. For any $1<b<\bar{b}(\al)$ we can find 
$\Lambda_0 = \Lambda_0 (\al,b,M)$} such that the following estimates hold for every  $\lambda_0 \geq \Lambda_0$: 
\begin{equation}\label{est.asR}\begin{split}
\norm{R_{q+1} - {\textstyle{\frac{2}{3}}} \varrho \I}_N &\leq C_M\la_{q+1}^N \cdot
 {\la_q^\frac12}\la_{q+1}^{-\frac12} \de_q^\frac14\de_{q+1}^\frac34 
 \leq \la_{q+1}^{N-3\ga} \de_{q+2}, 
 \ \ \qquad\qquad \forall N=0,1,2 \\
\norm{\as D_t (R_{q+1} - {\textstyle{\frac{2}{3}}} \varrho \I)}_{N-1} &\leq C_M  \la_{q+1}^N \de_{q+1}^\frac12\cdot
 {\la_q^\frac12}\la_{q+1}^{-\frac12} \de_q^\frac14\de_{q+1}^\frac34  \leq \la_{q+1}^{N-3\ga}\de_{q+1}^\frac12 \de_{q+2}, \quad \forall N=1,2\, .
\end{split}\end{equation}
{where $C_M$ depends only upon the $M>1$ of Proposition \ref{ind.hyp} and Proposition \ref{p:ind_technical}.}
\end{prop}

Taking into account \eqref{e:splitting_of_Reynolds}, we will just estimate the separate terms $\as R_T$, $\as R_N$, $\as R_{O1}$, $\as R_{O2}$ and $\as R_M$.
For the errors $\asR_{O2}$ and $\asR_M$, we use a direct estimate, while the other errors, including the inverse divergence operator, are estimated by Corollary \ref{cor.mic2}. In the following subsections, we fix $n_0 = \ceil{\frac{2b(2+\al)}{(b-1)(1-\al)}}$ so that $\la_{q+1}^2({\la_{q+1}\mu_q})^{-(n_0+1)}\lec \de_{q+1}^\frac12$ for any $q$ {and allow the dependence on $M$ of the implicit constants in $\lec$.} Also, we remark that 
\begin{align}\label{rel.par}
\frac {1}{\la_{q+1}\tau_q} +\frac{\de_{q+1}^\frac12}{\la_{q+1}\mu_q} \lec_M  {\la_q^\frac12}{\la_{q+1}^{-\frac12}} \de_q^\frac14\de_{q+1}^\frac14.
\end{align}
For the convenience, we restrict the range of $N$ as in \eqref{est.asR} in this section, without mentioning it further.

\subsection{Transport stress error}
Recall that
\[
\asR_T = \idv{D_{t,\ell} w}.
\]
Since $D_{t,\ell} \xi_I =0$, we have  
\begin{align*}
D_{t,\ell} w 
&= D_{t,\ell} \left( \sum_{m\in\Z} \sum_{k\in \Z^3\setminus \{0\}} \de_{q+1}^\frac 12(b_{m,k} + (\la_{q+1}\mu_q)^{-1}e_{m,k}) e^{i\la_{q+1} k\cdot \xi_I}\right)\\
&=\sum_{m} \sum_{k} \de_{q+1}^\frac 12  D_{t,\ell} (b_{m,k} + (\la_{q+1}\mu_q)^{-1}e_{m,k}) e^{i\la_{q+1} k\cdot \xi_I}.
\end{align*}
Since $b_{m,k}$ and $e_{m,k}$ satisfy
$\supp(b_{m,k}), \supp(e_{m,k})\subset (t_m-\frac 12\tau_q, t_m + \frac 32\tau_q) \times \R^3$
and 
\begin{align*}
&\norm{D_{t,\ell} (b_{m,k} + (\la_{q+1}\mu_q)^{-1}e_{m,k})}_{\bar{N}} + (\la_{q+1}\de_{q+1}^\frac 12)^{-1}\norm{D_{t,\ell}^2 (b_{m,k} + (\la_{q+1}\mu_q)^{-1}e_{m,k})}_{\bar{N}}\\
&\qquad\qquad \lec_{\bar{N},M} \mu_q^{-\bar{N}} \frac {|\dot{b}_{I,k}|}{\tau_q},
\end{align*}
for any $\bar{N}\geq 0$ by \eqref{est.b} and \eqref{est.e}, we can apply Corollary \ref{cor.mic2} to get
\begin{align}\label{est.RT}
\norm{\asR_T}_N
\lec \la_{q+1}^N\frac {\de_{q+1}^\frac 12}{\la_{q+1} \tau_q} , \quad
\norm{\as D_t \asR_T}_{N-1}  \lec \la_{q+1}^N\de_{q+1}^\frac12 \frac {\de_{q+1}^\frac 12}{\la_{q+1} \tau_q}.
\end{align}
 
 \subsection{Nash stress error}
Set $\asR_N = \idv{(w \cdot \na) v_\ell}$ and observe that
\begin{align*}
(w \cdot \na) v_\ell 
&=  \sum_{m} \sum_{k\in \Z^3\setminus \{0\}} \de_{q+1}^\frac 12((b_{m,k} + (\la_{q+1}\mu_q)^{-1}e_{m,k})\cdot \na)v_\ell e^{i\la_{q+1} k\cdot \xi_I}.
\end{align*}
Since $b_{m,k}$ and $e_{m,k}$ satisfy
$\supp(b_{m,k}), \supp(e_{m,k})\subset (t_m-\frac 12\tau_q, t_m + \frac 32\tau_q) \times \R^3$ and
\begin{align*}
\norm{(b_{m,k} + (\la_{q+1}\mu_q)^{-1}e_{m,k})\cdot \na)v_\ell}_{\bar{N}}
&\lec_{\bar N} \mu_q^{-{\bar{N}}} |\dot{b}_{I,k}|\la_q\de_{q}^\frac12\\
\norm{D_{t,\ell}[(b_{m,k} + (\la_{q+1}\mu_q)^{-1}e_{m,k})\cdot \na)v_\ell]}_{\bar{N}}
&\lec_{\bar N} \la_{q+1}\de_{q+1}^\frac12 \mu_q^{-{\bar{N}}} |\dot{b}_{I,k}|\la_q\de_{q}^\frac12
\end{align*}
for any ${\bar{N}}\geq0$ by \eqref{est.vp}, \eqref{est.Dtvl}, \eqref{est.b}, and \eqref{est.e}, we apply Corollary \ref{cor.mic2} and obtain
\begin{align}\label{est.RN}
\norm{\asR_N}_N 
\lec \la_{q+1}^N \frac {\de_{q+1}^\frac 12}{\la_{q+1}\tau_q} , \quad 
\norm{\as D_t\asR_N}_{N-1}  \lec \la_{q+1}^N\de_{q+1}^\frac12 \frac {\de_{q+1}^\frac 12}{\la_{q+1}\tau_q}.
\end{align}

\subsubsection{Oscillation stress error}
Recall that $\asR_O = \asR_{O1} + \asR_{O2}$ where
\begin{align*}
\asR_{O1} 
&= \idv{\na\cdot ( w_o\otimes w_o + R_\ell) }\\
\asR_{O2} 
&=w_o \otimes w_c + w_c \otimes w_o + w_c\otimes w_c.
\end{align*}
We compute
\begin{align*}
\na\cdot ( w_o\otimes w_o + R_\ell)
&=\na\cdot ( w_o\otimes w_o - \de_{q+1}\I + R_\ell)
= \div \left[ \sum_{\substack{m\in \Z\\k\in \Z^3\setminus \{0\}}} \de_{q+1} c_{m,k} e^{ i\la_{q+1} k\cdot \xi_I} \right]
\\
&=\sum_{m,k} \de_{q+1} \div(c_{m,k}) e^{ i\la_{q+1} k\cdot \xi_I},
\end{align*}
because of $\dot{c}_{I,k} (f_I\cdot k) =0$. Also, since we have
\begin{align*}
D_{t,\ell} \div c_{m,k} = \div( D_{t,\ell} c_{m,k}) - (\na v)_{ij}\pa_i (c_{m,k})_{jl},
\end{align*}
it follows from \eqref{est.c} that $\norm{\div c_{m,k}}_{\bar{N}}+(\la_{q+1}\de_{q+1}^\frac12)^{-1}\norm{D_{t,\ell} \div c_{m,k}}_{\bar{N}} \lec_{\bar{N},M}  \mu_q^{-\bar{N}} \frac{|\dot{c}_{I,k}|}{\mu_q}$ for any $\bar{N}\geq 0$. Finally using $\supp(c_{m,k})\subset (t_m-\frac12\tau_q, t_m +\frac32\tau_q)\times \R^3$, we apply Corollary \ref{cor.mic2} to get
\begin{equation}\label{est.RO1}
\begin{split}
\norm{\asR_{O1}}_N
&\lec \la_{q+1}^N\cdot \frac {\de_{q+1}}{\la_{q+1}\mu_q}, \quad
\norm{\as D_t \asR_{O1}}_{N-1}\lec\la_{q+1}^N\de_{q+1}^\frac12\cdot \frac {\de_{q+1}}{\la_{q+1}\mu_q} . 
\end{split}
\end{equation}
On the other hand, we use \eqref{est.W}, \eqref{est.Wc}, \eqref{est.w}, and \eqref{est.v.dif} to estimate $\asR_{O2}$ as follows,
\begin{align*}
\norm{\asR_{O2}}_N
&\lec  \sum_{N_1+N_2=N} \norm{w_o}_{N_1}\norm{w_c}_{N_2} + \sum_{N_1+N_2=N}\norm{w_c}_{N_1}\norm{w_c}_{N_2}
\lec \la_{q+1}^N\cdot\frac{\de_{q+1}}{\la_{q+1}\mu_q} ,\\
\norm{\as D_t\asR_{O2}}_{N-1}
&\leq \norm{D_{t,\ell} \asR_{O2}}_{N-1} + \norm{(w+v-v_\ell)\cdot\na \asR_{O2}}_{N-1}\\
&\lec  \sum_{N_1+N_2=N-1} 
\norm{D_{t,\ell} w_o}_{N_1}\norm{w_c}_{N_2} 
+ \norm{ w_o}_{N_1}\norm{D_{t,\ell} w_c}_{N_2}
+\norm{D_{t,\ell}w_c}_{N_1}\norm{w_c}_{N_2}\\
&\quad+ \sum_{N_1+N_2=N-1} 
(\norm{w}_{N_1} + \norm{v-v_\ell}_{N_1})\norm{\as R_{O2}}_{N_2+1}
\lec \la_{q+1}^{N}\de_{q+1}^\frac12 \cdot\frac{ \de_{q+1}}{\la_{q+1}\mu_q}.
\end{align*}
Therefore, we have 
\begin{equation}\label{est.RO}
\norm{\asR_O}_N \lec \la_{q+1}^N \frac{\de_{q+1}}{\la_{q+1}\mu_q}, \quad
\norm{\as D_t \asR_O}_{N-1}   
\lec \la_{q+1}\de_{q+1}^{\frac12}
\cdot \la_{q+1}^N \frac{\de_{q+1}}{\la_{q+1}\mu_q}.
\end{equation} 
\subsection{Mediation stress error}
Recall that 
\[
\asR_M = R-R_\ell+(v-v_\ell)\otimes w + w\otimes (v-v_\ell).
\]
Using \eqref{est.R.dif}, \eqref{est.v.dif}, and \eqref{est.w}, we have
\begin{align*}
\norm{\asR_M}_{N} 
&\lec \norm{R-R_\ell}_N+\sum_{N_1+N_2=N}\norm{v-v_\ell}_{N_1}\norm{w}_{N_2}\\
&\lec \la_{q+1}^N \cdot ( {\la_q^\frac12}\la_{q+1}^{-\frac12} \de_q^\frac14\de_{q+1}^\frac34  + (\ell\la_q)^2\de_q^\frac12\de_{q+1}^\frac12)
\lec \la_{q+1}^N \cdot  {\la_q^\frac12}\la_{q+1}^{-\frac12} \de_q^\frac14\de_{q+1}^\frac34.
\end{align*}
To estimate $\as D_t  \asR_M$, we additionally use the
decomposition $\as D_t  \asR_M  = D_{t,\ell} \asR_M + ((v-v_\ell)+w)\cdot\na \asR_M $ to obtain
\begin{align*}
\norm{\as D_t \asR_M}_{N-1}
&\lec \norm{\as D_t (R-R_\ell)}_{N-1} 
+ \sum_{N_1+N_2=N-1} \norm{\as D_t (v-v_\ell)}_{N_1} \norm{w}_{N_2} + \norm{v-v_\ell}_{N_1} \norm{\as D_t w}_{N_2} \nonumber\\
&\lec \la_{q+1}^N \de_{q+1}^\frac12\cdot  {\la_q^\frac12}\la_{q+1}^{-\frac12} \de_q^\frac14\de_{q+1}^\frac34.
\end{align*} 
To summarize, we obtain 
\begin{align}\label{est.RM}
\norm{\asR_M}_{N}
\lec  \la_{q+1}^N  {\la_q^\frac12}\la_{q+1}^{-\frac12} \de_q^\frac14\de_{q+1}^\frac34, \quad
\norm{\as D_t \asR_M}_{N-1}
\lec \la_{q+1}^N \de_{q+1}^\frac12  {\la_q^\frac12}\la_{q+1}^{-\frac12} \de_q^\frac14\de_{q+1}^\frac34.
\end{align} 
 \
 
{Finally, Proposition \ref{p:Reynolds} follows from
\eqref{est.RT}-\eqref{est.RM} and \eqref{rel.par}.}

\section{Estimates for the new current}\label{sec.cur} 

In this section, we obtain the last needed estimates, on the new unsolved current $\ph_{q+1}$ and on the remaining part of the Reynolds stress $\frac{2}{3} \varrho\I$, which we
summarize in the following proposition.

\begin{prop}\label{p:current}
{There exists  $\bar{b}(\al)>1$ with the following property. For any $1<b<\bar{b}(\al)$ there is 
$\Lambda_0 = \Lambda_0 (\al,b,M)$ such that} the following estimates hold for $\lambda_0 \geq \Lambda_0 $:
\begin{align}\begin{split}
\label{est.asph}
\norm{\ph_{q+1}}_N
&\leq \la_{q+1}^{N-3\ga} \de_{q+2}^\frac32 , \qquad\quad \forall N=0,1,2,\\
\norm{\as D_t \ph_{q+1}}_{N-1} 
&\leq \la_{q+1}^{N-3\ga}\de_{q+1}^\frac12 \de_{q+2}^\frac32, \quad \forall N=1,2\, ,
\end{split}
\\
\norm{\varrho}_0 + \norm{\varrho'}_0
&\leq \la_{q+1}^{-3\ga} \de_{q+2}^\frac32\, .
\label{est.varho}
\end{align}
\end{prop}

Without mentioning, we assume that $N$ is in the range above and allow the dependence on $M$ of the implicit constants in $\lec$ in this section.
For convenience, we single out the following fact, which will be repeatedly used: note that there exists $\bar{b}(\al)>1$ such that for any $1<b<\bar{b}(\al)$ and a constant $C_M$ depending only on $M$, we can find $\La_0=\La_0(\al, b,M)$ which gives
\[
C_M\left[ \frac {\de_{q+1}}{\la_{q+1}\tau_q} +\frac{\de_{q+1}^\frac32}{\la_{q+1}\mu_q} + \frac {\la_q^\frac12}{\la_{q+1}^\frac12} \de_q^\frac14\de_{q+1}^\frac54\right]
\leq  \la_{q+1}^{-3\ga} \de_{q+2}^\frac32,
\] 
for any $\lambda_0 \geq \La_0$.

Another important remark which will be used in this section is that
\begin{equation}\label{e:average-free}
\mathcal{R} (g (t, \cdot)+ h(t)) = \mathcal{R} (g (t, \cdot))
\end{equation}
for every smooth periodic time-dependent vector field $g$ and for every $h$ which depends only on time. 

\subsection{High frequency current error}
We start by observing that $\as\ph_{H1}$ is
 \begin{equation}\label{e:formula_H1}
 \as\ph_{H1}= \idv{(\div P_{\le \ell^{-1}}R + Q(v_q,v_q) ) \cdot w}
 \end{equation}
 by \eqref{e:average-free}. We thus can apply Corollary \ref{cor.mic2} to
\begin{align*}
(\div P_{\le \ell^{-1}}&R_q + Q(v_q,v_q)) \cdot w\\
&= \sum_{m,k} (\div P_{\le \ell^{-1}}R_q + Q(v_q,v_q)) \de_{q+1}^\frac12(b_{m,k} + (\la_{q+1}\mu_q)^{-1}e_{m,k})e^{i\la_{q+1}k\cdot \xi_I}.
\end{align*}
Indeed, using \eqref{est.R}, \eqref{est.Qvv}, \eqref{est.b}, \eqref{est.e}, we obtain
\begin{align*}
&\norm{\as\ph_{H1} }_N
\lec \la_{q+1}^{N-1} \de_{q+1}^\frac12 (\la_q^{1-3\ga}\de_{q+1} + (\ell \la_q) \la_q\de_q) 
\lec \la_{q+1}^N \frac{\la_q^{1-3\ga}}{\la_{q+1}} \de_{q+1}^\frac32.
\end{align*}
Furthemore, \eqref{est.vp}, \eqref{est.R}, \eqref{est.v.dif}, and \eqref{est.com1} imply
\begin{align*}
\norm{D_{t,\ell} \div P_{\le \ell^{-1}} R_q}_{N-1}
&\leq \norm{\div P_{\le \ell^{-1}}  D_{t,\ell} R_q}_{N-1}
+ \norm{\div [v_\ell\cdot\na, P_{\le \ell^{-1}}] R_q}_{N-1}\\
&\quad+ \norm{(\na v_\ell)_{ki} \pa_k P_{\le l^{-1}} (R_q)_{ij}}_{N-1}\\
&\lec \la_{q+1}^{N-1} (\norm{D_{t,\ell} R_q}_{1}
+ \norm{ [v_\ell\cdot\na, P_{\le \ell^{-1}}] R_q}_{1}
+\norm{\na v_q}_0 \norm{ \na R_q}_0)\\
&\lec \la_{q+1}^{N}\de_{q+1}^\frac12 \la_q^{1-3\ga} \de_{q+1}.
\end{align*}
Then, it follows that
\begin{align*}
\norm{\as D_t \ph_{H1}}_{N-1}
\lec\la_{q+1}^N\de_{q+1}^\frac12\cdot \frac{\la_q^{1-3\ga}}{\la_{q+1}} \de_{q+1}^\frac32.
\end{align*}

In order to deal with $\as\ph_{H2}$, we use the definition of $R_{q+1}$ to get
\begin{equation}\begin{split}
\label{rep:low.freq.app}
w\otimes w&-\de_{q+1}\I+R_q -R_{q+1}\\
&=
(w_o\otimes w_o-\de_{q+1}\I+R_\ell )- \asR_{O1}  -\asR_T -\asR_N-\asR_{M2} - {\textstyle{\frac{2}{3}}}\varrho \I\, .
\end{split}\end{equation}
Using $\I : \nabla v_\ell^\top = \nabla \cdot v_\ell =0$ and \eqref{e:average-free}, we can then write
\begin{align}
\as\ph_{H2}
&= \idv{(w\otimes w-\de_{q+1}I+R_q-R_{q+1} + (v-v_\ell)\otimes w+ w\otimes (v-v_\ell)):\na v_\ell^\top}\label{e:formula_H2}\\
&= \idv{((w_o\otimes w_o-\de_{q+1}\I+R_\ell )- \asR_{O1}  -\asR_T -\asR_N):\na v_\ell^\top}.\nonumber
\end{align}
Apply Corollary \ref{cor.mic2} with \eqref{alg.eq}, \eqref{est.c}, and \eqref{est.vp}, we have
\begin{align*}
\norm{\idv{(w_o\otimes w_o-\de_{q+1}\I+R_\ell ):\na v_\ell^\top}}_N
&\lec \la_{q+1}^N \la_q \de_q^\frac12 \frac {\de_{q+1}}{\la_{q+1}},\\
\norm{\as D_t \idv{(w_o\otimes w_o-\de_{q+1}\I+R_\ell ):\na v_\ell^\top}}_{N-1}
& \lec  \la_{q+1}^N\de_{q+1}^\frac12 \la_q \de_q^\frac12 \frac {\de_{q+1}}{\la_{q+1}}.
\end{align*}
By \eqref{dec.mic} and \eqref{rem.low}, recall that the Reynolds stress errors $\asR_\tri$, which represents either $\asR_{O1}$, $\asR_T$, or $\asR_N$, can be written as $\asR_\tri  = \cR G_\tri$ satisfying 
\begin{align}
\norm{G_\tri}_N \lec \la_{q+1}^N \left(\frac{\de_{q+1}}{\mu_q} + \frac{\de_{q+1}^\frac12}{\tau_q} \right), 
\quad
&\norm{D_{t,\ell} G_\tri}_{N-1} \lec \la_{q+1}^N \de_{q+1}^\frac12
\left(\frac{\de_{q+1}}{\mu_q} + \frac{\de_{q+1}^\frac12}{\tau_q} \right), \label{est.Gtr} 
\end{align}
Furthermore, such $G_\tri$ has the form $\sum_{m,k}g_\tri^{m,k} e^{i\la_{q+1}k\cdot\xi_I}$ and has a decomposition 
\[
G_\tri = \cP_{\gtrsim \la_{q+1}} G_\tri+\cP_{\lec \la_{q+1}} G_\tri\, ,
\] 
as in \eqref{dec.mic} and \eqref{rem.low}, where $\cP_{\lec \la_{q+1}} G_\tri$ satisfies
\begin{align}
\norm{\cP_{\lec \la_{q+1}}G_\tri}_0 \lec \la_{q+1}^{-2}\de_{q+1}^\frac12 \left(\frac{\de_{q+1}}{\mu_q} + \frac{\de_{q+1}^\frac12}{\tau_q} \right),\ 
&\norm{D_{t,\ell}\cP_{\lec \la_{q+1}}G_\tri}_0 \lec \la_{q+1}^{-1}\de_{q+1} \left(\frac{\de_{q+1}}{\mu_q} + \frac{\de_{q+1}^\frac12}{\tau_q} \right). \label{est.DtGtr}
\end{align}
Indeed, they follow from \eqref{est.eN} and \eqref{est.DteN}.
Since $\na v_\ell$ has the frequency localized to $\lec \ell^{-1}$ and $\ell^{-1} \leq \frac 1{64}\la_{q+1}$ for sufficiently large $\la_0$, {$\cR \cP_{\gtrsim \la_{q+1}} G_\tri:(\na v_\ell)^\top$ has the frequency localized to $\gtrsim \la_{q+1}^{-1}$ and}
\begin{align*}
\norm{\idv{\cR \cP_{\gtrsim \la_{q+1}} G_\tri:(\na v_\ell)^\top}}_N
&\lec \frac 1{\la_{q+1}} \norm{\cR  \cP_{\gtrsim \la_{q+1}} G_\tri:(\na v_\ell)^\top}_N\\
&\lec \frac 1{\la_{q+1}^2} \sum_{N_1+N_2=N} \norm{ G_\tri}_{N_1} \norm{\na v_\ell}_{N_2}
\end{align*}
On the other hand, $\cR \cP_{\lec \la_{q+1}} G_\tri:(\na v_\ell)^\top$ has the frequency localized to $\lec \la_{q+1}^{-1}$, so that
\begin{align*}
\norm{\idv{\cR \cP_{\lec \la_{q+1}} G_\tri:(\na v_\ell)^\top}}_N
&\lec \la_{q+1}^N \norm{\idv{\cR \cP_{\lec \la_{q+1}} G_\tri:(\na v_\ell)^\top}}_0
%\lec \la_{q+1}^N \norm{\cR P_{\leq \la_{q+1}} G:\na v_\ell}_0\\
\lec \la_{q+1}^N \norm{\cP_{\lec \la_{q+1}} G_\tri}_0\norm{\na v}_0.
\end{align*}
Therefore, using \eqref{est.Gtr} and \eqref{est.DtGtr}, we obtain
\begin{align*}
\norm{\cal{R}(\asR_{O1}:(\na v_\ell)^\top)}_N+\norm{\cal{R}(\asR_T:(\na v_\ell)^\top)}_N
&+\norm{\cal{R}(\asR_N:(\na v_\ell)^\top)}_N\\
&\lec\la_{q+1}^N
\left(\frac {\de_{q+1}^\frac12}{\la_{q+1}\tau_{q}}+ \frac {\de_{q+1}}{\la_{q+1}\mu_{q}}\right) \frac{\la_q\de_q^\frac12}{\la_{q+1}}.
\end{align*}

To estimate their advective derivatives, we use the decomposition
\begin{align*}
\as D_t \idv{\asR_{\triangle}:(\na v_\ell)^\top} 
= D_{t,\ell}\idv{\asR_{\triangle}:(\na v_\ell)^\top} + (w+ (v_q-v_\ell))\cdot \na(\asR_{\triangle}:(\na v_\ell)^\top). 
\end{align*}
We can easily see that 
\begin{align*}
\norm{(w+ (v-v_\ell))\cdot \na\cR(\asR_{\triangle}:(\na v_\ell)^\top)}_{N-1}
%&\lec \sum_{N_1+N_2=N-1} 
%(\norm{w}_{N_1} + \norm{v-v_\ell}_{N_1})\norm{\cR(\asR_{\triangle}:\na v_\ell)}_{N_2+1}\\
&\lec  \la_{q+1}^N\de_{q+1}^\frac12
\left(\frac {\de_{q+1}^\frac12}{\la_{q+1}\tau_{q}}+ \frac {\de_{q+1}}{\la_{q+1}\mu_{q}}\right) \frac{\la_q\de_q^\frac12}{\la_{q+1}}.
\end{align*}
For the estimate of the first term, we use again the decomposition $\asR_\tri = \cR \cP_{\gtrsim \la_{q+1}} G_\tri + \cR \cP_{\lec \la_{q+1}} G_\tri$ and 
\begin{align*}
\norm{D_{t,\ell}\cal{R}(\asR_{\triangle}:(\na v_\ell)^\top)}_{N-1}
&\leq 
\norm{D_{t,\ell}\cal{R}(\cR \cP_{\gtrsim \la_{q+1}} G_\tri :(\na v_\ell)^\top)}_{N-1}\\
&\quad+\norm{D_{t,\ell}\cal{R}( \cR \cP_{\lec \la_{q+1}} G_\tri:(\na v_\ell)^\top)}_{N-1}.
\end{align*}
In order to estimate the right hand side, consider the decomposition 
\begin{align*}
D_{t,\ell}\cR P_{\gtrsim \la_{q+1}} H
%&=\cR D_{t,\ell} P_{\geq \la_{q+1}} G  + [v_\ell\cdot\na, \cR] P_{\geq \la_{q+1}} G\\
&=\cR \cP_{\gtrsim \la_{q+1}} D_{t,\ell}  H  
+\cR [v_\ell\cdot\na, P_{\gtrsim \la_{q+1}}]  H 
+ [v_\ell\cdot\na, \cR] P_{\gtrsim \la_{q+1}} H,
\end{align*}
for any smooth function $H$ and Littlewood-Paley operator $P_{\gtrsim\la_{q+1}}$ projecting to the frequency $\gtrsim \la_{q+1}$. Similar to the proof of Lemma \ref{lem:com2}, we have
\begin{align*}
\norm{[v_\ell\cdot\na,P_{\gtrsim \la_{q+1}}] H}_{N-1}
\lec \la_{q+1}^{N-2} \norm{\na v_\ell}_0\norm{\na H}_0 
\end{align*}
Also, similar to \eqref{est.DtcR}, we obtain
\begin{align}\label{est1}
\norm{[v_\ell\cdot\na, \cR]P_{\gtrsim \la_{q+1}} H}_{N-1} 
\lec \sum_{N_1+N_2=N-1} \ell \norm{\na v_\ell}_{N_1} \norm{H}_{N_2}. 
\end{align}
Since $P_{\gtrsim \la_{q+1}} D_{t,\ell}  H$ and $ [v_\ell\cdot\na, P_{\gtrsim \la_{q+1}}]  H$ have frequencies localized to $\gtrsim \la_{q+1}$, it follows that
\begin{equation}\begin{split}\label{est.H}
\norm{&D_{t,\ell}\cR P_{\gtrsim \la_{q+1}} H}_{N-1}\\
&\lec \norm{\cR P_{\gtrsim \la_{q+1}} D_{t,\ell}  H }_{N-1} 
+\norm{\cR [v_\ell\cdot\na, P_{\gtrsim \la_{q+1}}]  H }_{N-1}
+\norm{ [v_\ell\cdot\na, \cR] P_{\gtrsim\la_{q+1}} H}_{N-1}\\
&\lec 
\frac 1{\la_{q+1}} \norm{ P_{\gtrsim \la_{q+1}} D_{t,\ell}  H }_{N-1} 
+  \frac 1{\la_{q+1}}\norm{ [v_\ell\cdot\na, P_{\gtrsim \la_{q+1}}]  H }_{N-1}
+\norm{ [v_\ell\cdot\na, \cR] P_{\gtrsim \la_{q+1}} H}_{N-1}\\
&\lec \frac 1{\la_{q+1}} \norm{ D_{t,\ell}  H }_{N-1} 
+ \la_{q+1}^{N-3}\norm{\na v}_0 \norm{\na H}_0
+\sum_{N_1+N_2=N-1} \ell \norm{\na v_\ell}_{N_1}\norm{H}_{N_2}.
\end{split}\end{equation}
Now, we apply it to $H= \cR \cP_{\gtrsim \la_{q+1}} G_\tri:(\na v_\ell)^\top$. For such $H$, we have $H= P_{\geq \frac18\la_{q+1}}H$ for sufficiently large $\la_0$, so that
\begin{align*}
&\norm{D_{t,\ell}\cR(\cR  \cP_{\gtrsim \la_{q+1}}  G_\tri:(\na v_\ell)^\top)}_{N-1}\\
&\lec \frac 1{\la_{q+1}}\norm{D_{t,\ell}(\cR  \cP_{\gtrsim \la_{q+1}}  G_\tri: (\na v_\ell)^\top)}_{N-1} 
+ \la_{q+1}^{N-3}\norm{\na v_\ell}_0 
\norm{ \cR \cP_{\gtrsim \la_{q+1}} G_\tri:(\na v_\ell)^\top}_1\\
&\quad+\sum_{N_1+N_2=N-1} \ell\norm{\na v_\ell}_{N_1}\norm{\cR \cP_{\gtrsim \la_{q+1}} G_\tri:(\na v_\ell)^\top}_{N_2}\\
&\lec \la_{q+1}^{N-2} \de_{q+1}^\frac12
\left(\frac{\de_{q+1}}{\mu_q} + \frac{\de_{q+1}^\frac12}{\tau_q} \right) \la_q\de_q^\frac12.
\end{align*}
Indeed, the second inequality can be obtained by applying \eqref{est.H} again to $H= G_\tri$,
\begin{align*}
&\norm{D_{t,\ell}(\cR \cP_{\gtrsim \la_{q+1}} G_\tri: \na v_\ell)}_{N-1} \\
&\lec \sum_{N_1+N_2=N-1}
\norm{D_{t,\ell}\cR  \cP_{\gtrsim \la_{q+1}}  G_\tri}_{N_1} \norm{\na v_\ell}_{N_2}
+\norm{\cR  \cP_{\gtrsim \la_{q+1}}  G_\tri}_{N_1}\norm{ D_{t,\ell}\na v_\ell}_{N_2}\\
&\lec \sum_{N_1+N_2=N-1}
\left(\frac1{\la_{q+1}}\norm{D_{t,\ell}G_\tri}_{N_1}
+ \la_{q+1}^{N_1-2}\norm{\na v}_0\norm{\na G_\tri}_0 \right)\norm{\na v_\ell}_{N_2}\\
&\quad +\sum_{\substack{N_{11}+N_{12} = N_1\\ N_1+N_2=N-1}} \ell \norm{\na v_\ell}_{N_{11}}\norm{G_\tri}_{N_{12}}\norm{\na v_\ell}_{N_2}
+\sum_{N_1+N_2=N-1}\frac{\norm{ G_\tri}_{N_1}}{\la_{q+1}}\norm{ D_{t,\ell}\na v_\ell}_{N_2}\\
&\lec \la_{q+1}^{N-1} \de_{q+1}^\frac12
\left(\frac{\de_{q+1}}{\mu_q} + \frac{\de_{q+1}^\frac12}{\tau_q} \right) \la_q\de_q^\frac12,
\end{align*}
and
\begin{align*}
\norm{\cR \cP_{\gtrsim \la_{q+1}} G_\tri:(\na v_\ell)^\top}_{N-1}
&\lec \sum_{N_{1}+N_{2} = N-1}
\norm{\cR \cP_{\gtrsim \la_{q+1}} G_\tri}_{N_{1}} \norm{\na v_\ell}_{N_{2}}\\
&\lec \la_{q+1}^{N-2} \left(\frac{\de_{q+1}}{\mu_q} + \frac{\de_{q+1}^\frac12}{\tau_q} \right) \la_q\de_q^\frac12.
\end{align*}
To estimate the remaining term $\norm{D_{t,\ell}\cal{R}( \cR \cP_{\lec \la_{q+1}} G_\tri:(\na v_\ell)^\top)}_{N-1}$, we observe that the frequency of $D_{t,\ell} \cR(\cR \cP_{\lec \la_{q+1}} G_\tri:(\na v_\ell)^\top)$ is  localized to $\lec \la_{q+1}$, so that 
\begin{align*}
\norm{&D_{t,\ell} \cR(\cR \cP_{\lec \la_{q+1}} G_\tri:(\na v_\ell)^\top)}_{N-1}
\lec \la_{q+1}^{N-1} \norm{D_{t,\ell} \cR(\cR \cP_{\lec \la_{q+1}} G_\tri:(\na v_\ell)^\top)}_0.
\end{align*}
Then, we control the right hand side as 
\begin{align*}
\norm{D_{t,\ell} &\cR(\cR \cP_{\lec \la_{q+1}}G_\tri:(\na v_\ell)^\top)}_0
\\
&\lec\norm{ \cR D_{t,\ell}(\cR \cP_{\lec \la_{q+1}} G_\tri:(\na v_\ell)^\top)}_0
+
 \norm{[v_\ell\cdot\na ,\cR] (\cR \cP_{\lec \la_{q+1}} G_\tri:(\na v_\ell)^\top)}_0\\
&\lec  \norm{ D_{t,\ell}\cR \cP_{\lec \la_{q+1}} G_\tri}_0\norm{\na v_\ell}_0 + \norm{  \cP_{\lec \la_{q+1}}G_\tri}_0\norm{ D_{t,\ell}\na v_\ell}_0\\
&\qquad\qquad +\norm{v_\ell}_0\norm{\na(\cR \cP_{\lec \la_{q+1}} G_\tri:(\na v_\ell)^\top)}_0 \\
&\lec ( \norm{ D_{t,\ell}\cP_{\lec \la_{q+1}} G_\tri}_0
+ \norm{[v_\ell\cdot \na, \cR] \cP_{\lec \la_{q+1}} G_\tri}_0)\norm{\na v_\ell}_0\\
&\quad+
\norm{  \cP_{\lec \la_{q+1}} G_\tri}_0\norm{ D_{t,\ell}\na v_\ell}_0
+\la_{q+1}\norm{\cP_{\lec \la_{q+1}} G_\tri}_0\norm{\na v_\ell}_0\\
&\lec
\la_q\de_q^\frac12(\norm{ D_{t,\ell}\cP_{\lec \la_{q+1}} G_\tri}_0
+\la_{q+1}\norm{\cP_{\lec \la_{q+1}} G_\tri}_0)
\lec \la_{q+1}^{-1}\de_{q+1}^\frac12 \left(\frac {\de_{q+1}^\frac12}{\tau_{q}}+ \frac {\de_{q+1}}{\mu_{q}}\right) \la_q\de_q^\frac12.
\end{align*}
Here, we used $\norm{\cR g}_0 \lec \norm{g}_0$. As a result, we obtain
\begin{align*}
\norm{D_{t,\ell}\cR(\asR_\triangle:(\na v_\ell)^\top)}_{N-1}
\lec \la_{q+1}^N\de_{q+1}^\frac12
\left(\frac {\de_{q+1}^\frac12}{\la_{q+1}\tau_{q}}+ \frac {\de_{q+1}}{\la_{q+1}\mu_{q}}\right) \frac{\la_q\de_q^\frac12}{\la_{q+1}},
\end{align*}
and
\begin{align*}
\norm{\as D_t \cal{R}(\asR_\tri:(\na v_\ell)^\top)}_{N-1}
\lec
\la_{q+1}^N\de_{q+1}^\frac12
\left(\frac {\de_{q+1}^\frac12}{\la_{q+1}\tau_{q}}+ \frac {\de_{q+1}}{\la_{q+1}\mu_{q}}\right) \frac{\la_q\de_q^\frac12}{\la_{q+1}}.
\end{align*}
Therefore, the estimates for $\as\ph_{H2}$ follow, 
\begin{align*}
\norm{\as\ph_{H2}}_N 
\lec \la_{q+1}^N \la_q \de_q^\frac12 \frac {\de_{q+1}}{\la_{q+1}},  \quad
\norm{\as D_t\as\ph_{H2}}_{N-1} 
\lec\la_{q+1}^N\de_{q+1}^\frac12 \la_q \de_q^\frac12 \frac {\de_{q+1}}{\la_{q+1}}.
\end{align*}
To summarize, we get
\begin{align*}
\norm{\as\ph_{H}}_N 
\leq \frac 15\la_{q+1}^{N-3\ga} \de_{q+2}^\frac32 ,  \quad
\norm{\as D_t\as\ph_{H}}_{N-1} 
\leq \frac 15\la_{q+1}^{N-3\ga}\de_{q+1}^\frac12 \de_{q+2}^\frac32\, .
\end{align*}

\subsection{Estimates on $\varrho_1$ and $\varrho_2$} We will in fact show the stronger estimate
\begin{align}
\norm{\varrho_1'}_0 &\leq \frac{1}{5{(T+ \tau_0)}} \la_{q+1}^{-3\ga} \de_{q+2}^\frac32\label{e:est_rho_1}\\
\norm{\varrho_2'}_0 &\leq \frac{1}{5{(T+ \tau_0)}} \la_{q+1}^{-3\ga} \de_{q+2}^\frac32\label{e:est_rho_2}
\end{align}
from which the estimates follow by integration
\begin{align}
\norm{\varrho_1}_0 &\leq \frac{1}{5} \la_{q+1}^{-3\ga} \de_{q+2}^\frac32\\
\norm{\varrho_2}_0 &\leq \frac{1}{5} \la_{q+1}^{-3\ga} \de_{q+2}^\frac32\, .
\end{align}
Observe that, if we denote $G_1$ and $G_2$, respectively, the arguments of $\mathcal{R}$ in the formulas \eqref{e:formula_H1} and \eqref{e:formula_H2} we just have
\[
\varrho_i' (t) = \int_{\mathbb T^3} G_i (t,x)\, dx\, .
\]
We can then argue as we did in the previous section to estimate $\|\as R_{Hi}\|_0 = \|\mathcal{R} (G_i)\|_0$, taking advantage of Lemma \ref{phase} and the representations \eqref{rep.W}-\eqref{e:rep4}.

\subsection{Transport current error}
We use the definition of $\as \ph_T$ and recall its splitting into $\as\ph_{T1}+\as\ph_{T2}$. 
Since  we have $\norm{e^{ i\la_{q+1} k\cdot \xi_I}}_{N} \lec \la_{q+1}^N|k|^2$ for any $k\in \Z^3\setminus\{0\}$, $D_{t,\ell} e^{i\la_{q+1}k\cdot \xi_I} =0$, and almost disjoint support of $c_{m,k}$, \eqref{alg.eq} and \eqref{est.c} imply
\begin{align}\label{est.sym.h1}
\norm{w_o\otimes w_o -\de_{q+1}I+ R_\ell}_N
&\lec \sum_{k\in \Z^3\setminus \{0\}}\norm{\sum_{m\in \Z}   \de_{q+1} c_{m,k} e^{ i\la_{q+1} k\cdot \xi_I}}_{N} 
\lec \la_{q+1}^N \de_{q+1}, 
\end{align}
 \begin{equation}\begin{split}
\label{est.sym.h2}
\norm{\as D_t (w_o\otimes w_o -\de_{q+1}I+ R_\ell)}_{N-1}
&\lec \sum_{k\in \Z^3\setminus \{0\}}\norm{\sum_{m\in \Z}   \de_{q+1} (D_{t,\ell}c_{m,k}) e^{ i\la_{q+1} k\cdot \xi_I}}_{N-1}\\
&\quad +\norm{(w+(v-v_\ell)\cdot \na )(w_o\otimes w_o -\de_{q+1}I+ R_\ell)}_{N-1}\\
&\lec \la_{q+1}^N\de_{q+1}^\frac12 \cdot \de_{q+1}. 
\end{split}\end{equation}
We then can use \eqref{est.asR}, \eqref{est.w}, \eqref{est.sym.h1}, \eqref{est.sym.h2}, \eqref{est.v.dif}, and $\ka_{q+1} = \frac12 \tr(R_{q+1})$ to estimate
\begin{align*}
\norm{\as\ph_{T1}}_N 
&\lec  \sum_{N_1+N_2=N} \norm{R_{q+1} -{\textstyle{\frac 23}} \varrho\I}_{N_1} \norm{w}_{N_2}
+  \sum_{N_1+N_2=N}\norm{ w_o\otimes w_o -\de_{q+1}I+ R_\ell}_{N_1}\norm{(v_q-v_\ell)}_{N_2}\\
&\lec  \la_{q+1}^N\left( \la_q^\frac12\la_{q+1}^{-\frac12}\de_q^\frac14\de_{q+1}^\frac14 + \ell^2\la_q^2 \de_q^\frac12\right) \de_{q+1}
\lec \la_{q+1}^N 
 \la_q^\frac12\la_{q+1}^{-\frac12}\de_q^\frac14\de_{q+1}^\frac54,
\end{align*}
\begin{align*}
\norm{\as D_t \ph_{T1}}_{N-1}
&\lec
 \sum_{N_1+N_2=N-1} \norm{\as D_t (R_{q+1} -{\textstyle{\frac 23}}\varrho \I)}_{N_1}\norm{ w}_{N_2} + \norm{R_{q+1}-{\textstyle{\frac 23}}\varrho \I}_{N_1}\norm{ \as D_t w}_{N_2} \\
&\quad+ \sum_{N_1+N_2=N-1}  \norm{\as D_t (w_o\otimes w_o -\de_{q+1}I+ R_\ell)}_{N_1}\norm{v_q-v_\ell}_{N_2}\\
&\quad+ \sum_{N_1+N_2=N-1}   \norm{w_o\otimes w_o -\de_{q+1}I+ R_\ell}_{N_1}\norm{\as D_t (v_q-v_\ell)}_{N_2}\\
&\lec \la_{q+1}^N \de_{q+1}^\frac12
 \la_q^\frac12\la_{q+1}^{-\frac12}\de_q^\frac14\de_{q+1}^\frac54.
 \end{align*}
As for $\as\ph_{T2}$, by \eqref{alg.eq},  
$\as\ph_{T2} 
=\frac 12\idv{\sum_{m} \sum_{k\in \Z^3\setminus \{0\}}  \de_{q+1} \tr(D_{t,\ell} c_{m,k}) e^{ i\la_{q+1} k\cdot \xi_I}} $
and estimate it using Corollary \ref{cor.mic2} with \eqref{est.c}
as follows
\begin{align*}
\norm{\as\ph_{T2}}_N
\lec  \la_{q+1}^N\cdot
\frac {\de_{q+1}}{\la_{q+1} \tau_q},\quad
\norm{\as D_t\as\ph_{T2}}_{N-1} \lec \la_{q+1}^N\de_{q+1}^\frac12 \cdot
\frac {\de_{q+1}}{\la_{q+1} \tau_q}. 
\end{align*}
To summarize, we have
\begin{align*}
\norm{\as\ph_{T}}_N\leq \frac 1{5}  \la_{q+1}^{N-3\ga}\de_{q+2}^\frac32,\quad
\norm{\as D_t\as\ph_{T}}_{N-1}\leq \frac 1{5}  \la_{q+1}^{N-3\ga}\de_{q+1}^\frac12\de_{q+2}^\frac32.
\end{align*}
by a suitable choice of $b$ and $\la_0$.

\subsection{Estimates on $\varrho_0$} We next observe that we have
\[
(2\pi)^3 \varrho_0' = \int \sum_{m} \sum_{k\in \Z^3\setminus \{0\}} \frac{ \de_{q+1} }2\tr(D_{t,\ell} c_{m,k}) e^{ i\la_{q+1} k\cdot \xi_I}\, dx
\]
and it thus suffices to use Lemma \ref{phase} to estimate
\[
\norm{\varrho_0'}_0 \leq \frac{1}{5{(T+ \tau_0)}} \la_{q+1}^{-3\ga}\de_{q+2}^\frac32\, .
\]
The estimate for $\norm{\varrho_0}_0$ follows thus from integrating the latter in time.

\subsection{Oscillation current error} 
Recall that $\as\ph_{O1}
=\idv{\div\left(\frac 12|w_o|^2 w_o + \ph_\ell \right)}$. We remark that \eqref{e:rep4} gives
\begin{align*}
\div\left(\frac 12|w_o|^2 w_o + \ph_\ell \right)
&= \div\left(\sum_{m\in \Z} \sum_{k\in \Z^3\setminus \{0\}} \de_{q+1}^\frac 32d_{m,k} e^{ i\la_{q+1} k\cdot \xi_I}\right)
= \sum_{m,k}  \de_{q+1}^\frac 32 \div(d_{m,k})e^{ i\la_{q+1} k\cdot \xi_I},
\end{align*}
because of $\dot{d}_{I,k} (f_I\cdot k) =0$. Also, we have
\begin{align*}
\norm{D_{t,\ell} \div d_{m,k}}_{\bar N}
\lec \norm{D_{t,\ell} d_{m,k}}_{\bar N+1} + \norm{(\na v_l)^{\top}:\na d_{m,k}}_{\bar N }
{\lec_{M,\bar N}\la_{q+1}^{\bar N} \frac{\la_{q+1}\de_{q+1}^\frac12}{\mu_q}}, \quad \forall \bar N\geq 0.
\end{align*}
Therefore, using $\supp(d_{m,k}) \subset (t_m -\frac 12\tau_q, t_m + \frac 32\tau_q) \times \R^3$, it follows from Corollary \ref{cor.mic2} with \eqref{est.d} that
\begin{align*}
\norm{\as\ph_{O1}}_N
\lec \la_{q+1}^N \cdot \frac {\de_{q+1}^\frac 32}{\la_{q+1}\mu_q}, \quad
\norm{\as D_t\as\ph_{O1}}_{N-1} 
\lec \la_{q+1}^N \de_{q+1}^\frac12\cdot \frac {\de_{q+1}^\frac 32}{\la_{q+1}\mu_q} .
\end{align*}
Next recall that $\as\ph_{O2} = \frac 12(|w|^2w - |w_o|^2w_o) $. Then, \eqref{est.W}-\eqref{est.w} imply
\begin{align*}
\norm{\as\ph_{O2}}_N
&\lec 
\norm{(w\cdot w_c) w}_N 
+ \norm{|w_c|^2 w}_N + \norm{|w_o|^2 w_c}_N
\lec \la_{q+1}^N \cdot \frac {\de_{q+1}^\frac32}{\la_{q+1}\mu_q}\\
\norm{\as D_t\as\ph_{O2}}_N
&\lec \norm{\as D_t (|w_o|^2 w_c)}_N
+ \norm{\as D_t [(2w_o\cdot w_c + |w_c|^2)w]}_N
 \lec \la_{q+1}^N\de_{q+1}^\frac12 \cdot \frac {\de_{q+1}^\frac32}{\la_{q+1}\mu_q}.
\end{align*}
Therefore, combining the estimates, we get
\begin{align*}
\norm{\as\ph_{O}}_N\leq \frac 1{5} \la_{q+1}^{N-3\ga} \de_{q+2}^\frac32, \quad 
\norm{\as D_t\as\ph_{O}}_{N-1}\leq \frac 1{5} \la_{q+1}^{N-3\ga} \de_{q+1}^\frac12\de_{q+2}^\frac32,
\end{align*}
for sufficiently small $b-1>0$ and large $\la_0$. 

\subsection{Reynolds current error} 
Recall that $\as\ph_R = (R_{q+1}-{\textstyle{\frac23}}\varrho \I) w$. Similar to the estimate for $\as\ka w$ in $\as\ph_{T1}$, we have
\begin{align*}
\norm{\as\ph_R}_N 
&\lec \la_{q+1}^N 
\la_q^\frac12\la_{q+1}^{-\frac12} \de_q^\frac14\de_{q+1}^\frac54
\leq \frac 15 \la_{q+1}^{N-3\ga} \de_{q+2}^\frac32,\\
\norm{\as D_t \as\ph_R}_{N-1}
&\lec \la_{q+1}^N\de_{q+1}^\frac12 \la_q^\frac12\la_{q+1}^{-\frac12} \de_q^\frac14\de_{q+1}^\frac54
\leq \frac 15 \la_{q+1}^{N-3\ga}\de_{q+1}^\frac12 \de_{q+2}^\frac32
\end{align*}
for sufficiently small $b-1>0$ and large $\la_0$. 

\subsubsection{Mediation current error} 
Recall that
\begin{align*}
\as\ph_M &= \left( \frac12|v_q-v_\ell|^2+ (p_q-p_\ell)\right)w+ (\ph_q-\ph_\ell) \\
&\qquad\quad
+ (w\otimes w-\de_{q+1}\I+R_q -R_{q+1}+{\textstyle{\frac23}}\varrho \I) (v_q-v_\ell).
\end{align*}
Then we recall that
\begin{align}\label{rep:low.freq.app1}
w\otimes w-\de_{q+1}\I+R_q -R_{q+1}+ {\textstyle{\frac23}}\varrho \I=
(w_o\otimes w_o-\de_{q+1}\I+R_\ell )- \asR_{O1}  -\asR_T -\asR_N-\asR_{M2},
\end{align}
so that it can be controlled in a similar way with $\as\ph_{T1}$,
\begin{align*}
\norm{(w\otimes w-\de_{q+1}\I+R_q -R_{q+1} +{\textstyle{\frac23}} \varrho \I) (v_q-v_\ell)}_{N}
&\leq \la_{q+1}^N\la_q^\frac12\la_{q+1}^{-\frac12} \de_q^\frac14\de_{q+1}^\frac54,\\
\norm{\as D_t[(w\otimes w-\de_{q+1}\I+R_q -R_{q+1} + {\textstyle{\frac23}}\varrho \I) (v_q - v_\ell)]}_{N-1}
&\leq_M \la_{q+1}^N\de_{q+1}^\frac12\la_q^\frac12\la_{q+1}^{-\frac12} \de_q^\frac14\de_{q+1}^\frac54
\end{align*}
For the remaining terms, we use \eqref{est.v.dif}, \eqref{est.p.dif}, \eqref{est.ph.dif}, and \eqref{est.w} to get 
\begin{align*}
\norm{\as\ph_M}_N \leq \frac 15 \la_{q+1}^{N-3\ga} \de_{q+2}^\frac32, \quad 
\norm{\as D_t  \as\ph_M}_{N-1} \leq \frac 15 \la_{q+1}^{N-3\ga}\de_{q+1}^\frac12 \de_{q+2}^\frac32.
\end{align*}

\section{Proofs of the key inductive propositions}

\subsection{Proof of Proposition \ref{ind.hyp}}
For a given dissipative Euler-Reynolds flow $(v_q,p_q,R_q,\ka_q,\ph_q)$ on the time interval $[0,T]+\tau_{q-1}$, we recall the construction of the corrected one: $v_{q+1} = v_q + w_{q+1}$ and $p_{q+1} = p_q + q_{q+1}$, where $w_{q+1}$ is defined by \eqref{def.w} and $q_{q+1} = 0$ on $[0,T]+\tau_{q}$. Furthermore, we find a new Reynolds stress $R_{q+1}$ and an unsolved flux current $\ph_{q+1}$ which solve \eqref{app.eq} together with $v_{q+1}$, $p_{q+1}$, and $\ka_{q+1} = \frac 12\tr(R_{q+1})$ and satisfy \eqref{est.asR}, \eqref{est.asph} and \eqref{est.varho} for sufficiently small $b-1>0$ and large $\la_0$. In other words, $(v_{q+1},p_{q+1},R_{q+1},\ka_{q+1},\ph_{q+1})$ is a dissipative Euler-Reynolds flow {for the energy loss $E$} and the error $(R_{q+1}, \ph_{q+1}, \ka_{q+1})$ satisfies \eqref{ka}-\eqref{est.ph} for $q+1$ as desired. Now, denote the absolute implicit constant in the estimate {\eqref{est.w.indM}} for $w$ by $M_0$ and define $M = 2M_0$. Then, one can easily see that
\begin{align*}
\norm{v_{q+1}- v_q}_0
+ \frac 1{\la_{q+1}} \norm{v_{q+1} - v_q}_1 
= \norm{w_{q+1}}_0 
+\frac 1{\la_{q+1}} \norm{w_{q+1}}_1 
\leq 2M_0 \de_{q+1}^\frac12  = M \de_{q+1}^\frac12. 
\end{align*}
Also, using \eqref{est.vp} and \eqref{est.w}, we have
\begin{align*}
\norm{v_{q+1}}_0
&\leq \norm{v_q}_0 + \norm{w_{q+1}}_0
\leq 1-\de_q^\frac12 + M_0\de_{q+1}^\frac12 
\leq 1- \de_{q+1}^\frac12,\\
\norm{v_{q+1}}_N
&\leq \norm{v_q}_N + \norm{w_{q+1}}_N
\leq M \la_q^N\de_q^\frac12 + \frac12 M \la_{q+1}^N \de_{q+1}^\frac12
\leq M\la_{q+1}^N \de_{q+1}^\frac12,\\
\norm{p_{q+1}}_N
&\leq \norm{p_q}_{N} + \norm{q_{q+1}}_N 
= \norm{p_q}_N \leq \la_q^N \de_q \leq \la_{q+1}^N \de_{q+1},\\
\norm{(\partial_t +v_{q+1}\cdot \nabla) p_{q+1}}_{N-1}&\leq \norm{(\partial_t +v_q\cdot \nabla) p_q}_{N_1} + \|w\cdot \nabla p_q\|_{N-1}\nonumber\\
& \leq \delta_q^{\frac{3}{2}}
\lambda_q^N + M \delta_{q+1}^{\frac{1}{2}} \lambda_{q+1}^{N-1} \delta_q \lambda_q \leq \delta_{q+1}^{\frac{3}{2}} \lambda_{q+1}^N,
\end{align*}
for $N=1,2$, provided that $\la_0$ is sufficiently large. Therefore, we construct a desired corrected flow $(v_{q+1},p_{q+1},R_{q+1},\ka_{q+1},\ph_{q+1})$. 

\subsection{Proof of Proposition \ref{p:ind_technical}}
Consider a given time interval $\cal I \subset (0,T)$ with $|\cal I| \geq 3\tau_q$. Then, we can always find $m_0$ such that $\supp(\th_{m_0}(\tau_q^{-1}\cdot))\subset  \cal I$. Now, if $I = (m_0, n, f)$ belongs to $\mathscr{S}_R$, we replace $\ga_{I}$ in $w_{q+1}$ by $\td \ga_{I} =-\ga_{I}$. In other words, we replace $\Ga_{I}$ by $\td \Ga_{I} = -\Ga_{I}$. Otherwise, we keep the same $\ga_{I}$. Note that $\td\Ga_{I}$ still solves \eqref{eq.Ga1} and hence $\td \ga_{I}$ satisfies \eqref{eq.Ga}. Also, we set $\td p_{q+1} = p_{q+1}$. Since the replacement does not change the estimates for $\Ga_{I}$ used in the proof of Lemma \ref{lem:est.coe}, the corresponding coefficients $\td b_{m,k}$, $\td c_{m,k}$, $\td d_{m,k}$, and $\td e_{m,k}$ satisfy \eqref{est.b}-\eqref{est.e}, and $\td w = \td w_o$, $\td w_c$, and $\td w_{q+1}$, generated by them, also fullfill \eqref{est.W}-\eqref{est.w}. As a result, the corrected dissipative Euler-Reynolds flow $(\td v_{q+1}, \td p_{q+1}, \td R_{q+1}, \td \ka_{q+1}, \td\ph_{q+1})$ satisfies \eqref{est.vp}-\eqref{est.ph} for $q+1$ and \eqref{cauchy} as desired. On the other hand, by the construction, the correction $\td w_{q+1}$ differs from $w_{q+1}$ on the support of $\th_{m_0}(\tau_q^{-1}\cdot)$. Therefore, we can easily see 
\[
\supp_t(v_{q+1} - \td v_{q+1})
=\supp_t(w_{q+1} - \td w_{q+1})
\subset \cal I. 
\] 
Furthermore, by \eqref{eq.Ga} and \eqref{eq.Ga1}, we have 
\begin{align*}
\sum_{I\in \mathscr{I}_R}\ga_I^2 |(\na \xi_I)^{-1}f_I|^2
&=\tr\left(
(\na \xi_I)^{-1}  \sum_{f\in \cF_{I,R}} \ga_I^2 f_I\otimes f_I [(\na \xi_I)^{-1}]^\top
\right)\\
&= \tr (\de_{q+1}\I - (R_q)_l - \widetilde{M})
\end{align*} 
where $\widetilde{M} = \sum_{m',n'}\th_{I'}^2\chi_{I'}^2(\xi_{I'})\sum_{f'\in \mathscr{I}_\ph}\ga_{I'}^2\left(\dint_{\T^3}\psi_{I'}^2dx \right) (\na \xi_{I'})^{-1} f' \otimes  (\na \xi_{I'})^{-1} f' $.
In particular 
\[
\norm{\widetilde{M}}_0 \lec \la_q^{-2\ga} \de_{q+1}
\] 
(see the proof in section \ref{subsec:R}). In this proof, $\norm{\cdot}_N$ denotes $\norm{\cdot}_{C([0,T];C^N(\T^3))}$.
Then, it follows that
\begin{align*}
|w_o - \td w_o |^2
&=  \sum_{I\in \mathscr{I}_R: m_I = m_0}
4\th_I^2(t) \chi_I^2(\xi_I) \ga_I^2 |(\na \xi_I)^{-1} f_I|^2 (1+(\psi_I^2(\la_{q+1} \xi_I)-1))\\
&= \sum_{I\in \mathscr{I}_R: m_I = m_0}
4\th_I^2\chi_I^2(\xi_I)  (3\de_{q+1} - \tr(R_\ell) - \tr(\widetilde{M}))\\&\quad+
\sum_{k\in \Z^3\setminus\{0\}}\sum_{I\in \mathscr{I}_R: m_I = m_0}
4\th_I^2\chi_I^2(\xi_I) \ga_I^2 |(\na \xi_I)^{-1} f_I|^2
\dot{c}_{I,k} e^{i\la_{q+1}k\cdot \xi_I}\\
&= 4\th_{m_0}^6(\tau_q^{-1}t) (3\de_{q+1}- \tr R_\ell - \tr {\widetilde{M} }) + \sum_{k\in \Z^3\setminus\{0\}} 4\de_{q+1}\tr(\td c_{m_0,k}^R)e^{i\la_{q+1}k\cdot \xi_I},
\end{align*}
where 
\[
\tr(\td c_{m_0,k}^R) = \sum_{I\in \mathscr{I}_R: m_I = m_0}
\th_I^2(t) \chi_I^2(\xi_I) \de_{q+1}^{-1}\ga_I^2\dot{c}_{I,k} |(\na \xi_I)^{-1} f_I|^2.
\]
Since we can obtain $\norm{\tr(\td c_{m_0,k}^R)}_N \lec \mu_q^{-N} |\dot{c}_{I,k}|$ for $N=0,1,2$ in the same way used to get the estimate \eqref{est.c} for $c_{m_0,k}$, we conclude
\begin{align*}
\norm{w_o - \td  w_o}_{C^0([0,T]; L^2(\T^3))}^2
&\geq 4(2\pi)^3(3\de_{q+1} - \norm{R_\ell}_0 - \norm{\tr(\widetilde{M})}_0) \\
&\quad- \sum_{k\in \Z^3}4 
\left| 
\de_{q+1} \int\tr(\td c_{m_0,k}^R)e^{i\la_{q+1}k\cdot \xi_I} dx
\right|\\
&\geq 12\de_{q+1} - C\de_{q+1}(\la_q^{-3\ga} + \la_q^{-2\ga} + (\la_{q+1} \mu_q)^{-2})\\
&\geq 4\de_{q+1}
\end{align*}
for sufficiently large $\la_0$. Indeed, in the second inequality, we used Lemma \ref{phase} to get
\begin{align*}
\sum_k\left| \int \tr(\td c_{m_0,k}^R) e^{\la_{q+1} k\cdot \xi_I} dx\right|
&\lec \sum_k\frac{\norm{\tr(\td c_{m_0,k}^R)}_2 + \norm{\tr(\td c_{m_0,k}^R)}_0\norm{{\na} \xi_I}_{C^0([t_{m_0}-\frac12\tau_q, t_{m_0}+\frac32\tau_q];C^2( \T^3))}}{\la_{q+1}^2 |k|^2}
\\
&\lec (\la_{q+1}\mu_q)^{-2} \sum_k\frac{|\dot{c}_{I,k}|}{|k|^2}
\lec (\la_{q+1}\mu_q)^{-2} 
\left(\sum_k |\dot{c}_{I,k}|^2\right)^\frac12
\left(\sum_k \frac{1}{|k|^4}\right)^\frac12.
\end{align*}
Therefore, we obtain
\begin{align*}
\norm{v_{q+1} - \td v_{q+1} }_{C^0([0,T]; L^2(\T^3))}
&=\norm{w_{q+1}- \td w_{q+1}}_{C^0([0,T]; L^2(\T^3))}\\
&\geq \norm{w_o - \td  w_o}_{C^0([0,T]; L^2(\T^3))}
-(2\pi)^\frac32(\norm{w_c}_0 + \norm{\td w_c}_0)\\
&\geq  2\de_{q+1}^\frac12 -\frac{(2\pi)^\frac32 2M_0}{\la_{q+1}\mu_q} \de_{q+1}^\frac12 \geq \de_{q+1}^\frac12
\end{align*}
for sufficiently large $\la_0$. 

Lastly, we suppose that a dissipative Euler-Reynolds flow $(\td v_q, \td p_q, \td R_q, \td \ka_q, \td \ph_q)$ satisfies \eqref{est.vp}-\eqref{est.ph} and
\[
\supp_t(v_q-\td v_q, p_q-\td p_q, R_q-\td R_q, \ka_q-\td \ka_q, \ph_q-\td \ph_q)\subset \cal{J} 
\]
for some time interval $\cal{J}$. Proceed to construct the regularized flow, $\td R_\ell$ and $\td \ph_\ell$ as we did for  $R_\ell$ and $\ph_\ell$ and note that they differ only in $\cal J + l_t\subset\cal J + (\la_q\de_q^\frac12)^{-1}$. Consequently, $w_{q+1}$ differ from $\td w_{q+1}$ only in $\cal J + (\la_q\de_q^\frac12)^{-1}$ and hence the corrected dissipative Euler-Reynolds flows $({v}_{q+1},  p_{q+1},  R_{q+1}, \ka_{q+1}, \ph_{q+1})$ and $(\td v_{q+1},  \td p_{q+1}, \td R_{q+1}, \td\ka_{q+1}, \td\ph_{q+1})$ satisfy 
\[
\supp_t(v_{q+1}-\td v_{q+1}, p_{q+1}-\td p_{q+1}, R_{q+1}-\td R_{q+1}, \ka_{q+1}-\td \ka_{q+1}, \ph_{q+1}-\td \ph_{q+1})\subset \cal{J} + (\la_q\de_q^\frac12)^{-1}.
\] 

 \appendix
 \section{Some technical lemmas}
The proof of the following two lemmas can be found in \cite[Appendix]{BDLSV2020}.
 \begin{lem}[H\"{o}lder norm of compositions] \label{lem:est.com} Suppose $F:\Omega \to \R$ and $\Psi: \R^n \to \Omega$ are smooth functions for some $\Omega\subset \R^m$. Then, for each $N\in \N$, we have
\begin{align}
&\norm{\na^N (F\circ \Psi)}_0
\lec \norm{\na F}_0 \norm{\na\Psi}_{N-1} + \norm{\na F}_{N-1} \norm{\Psi}_0^{N-1}\norm{\Psi}_N \nonumber\\
&\norm{\na^N (F\circ \Psi)}_0
\lec \norm{\na F}_0 \norm{\na\Psi}_{N-1} + \norm{\na F}_{N-1} \norm{\na\Psi}_0^{N} \label{chain},
\end{align}
where the implicit constant in the inequalities depends only on $n$, $m$, and $N$.
 \end{lem}
 
 \begin{lem}\label{phase}
 Let $N\geq 1$. Suppose that $a\in C^\infty(\T^3)$ and $\xi\in C^\infty(\T^3;\R^3)$ satisfies
 \[
 \frac 1{C} \leq |\na \xi| \leq C
 \]
 for some constant $C>1$. Then, we have
 \[
 \left| \int_{\T^3} a(x) e^{ik\cdot \xi} dx \right|
 \lec \frac{\norm{a}_N + \norm{a}_0\norm{{\na} \xi}_{N}}{|k|^N} ,
 \]
 where the implicit constant in the inequality is depending on $C$ and $N$, but independent of $k$. 
 \end{lem}
 
 \begin{lem}[Commutator estimate] Let $f$ and $g$ be in $C^\infty([0,T]\times \T^3)$ and set $f_\ell = P_{\leq \ell^{-1} } f$, $g_\ell = P_{\leq \ell^{-1}} g$ and $(fg)_\ell =
 P_{\leq \ell^{-1}} (fg)$. Then, for each $N \geq 0$, the following holds,
\begin{align}
\norm{f_\ell g_\ell  - (fg)_\ell }_N \lec_N  \ell^{2-N} \norm{f}_1\norm{g}_1. \label{est.com} 
\end{align}
 \end{lem}
 \begin{proof}
 Since the expression that we need to estimate is localized in frequency, by Bernstein's inequality it suffices to prove the case $N=0$. Recall now the function $m$ used to define the Littlewood-Paley operators and the number $J$, which is the maximal natural number such that $2^J \leq \ell^{-1}$. Denoting by $\widecheck{m}$ the inverse Fourier transform and by $\widecheck{m}_\ell$ the function $\widecheck{m}_\ell (x) =
2^{3J} \widecheck{m} (2^{J} x)$,
 a simple computation (see for instance \cite{CoETi1994}) gives
\begin{align*}
(f_\ell g_\ell  - (fg)_\ell) (x) &= \frac{1}{2} \int\int (f(x)- f(x-y)) (g(x)-g(x-z)) \widecheck{m}_\ell (y) \widecheck{m}_\ell (z)\, dy\, dz\, \\
&= \frac{1}{2} P_{> \ell^{-1}} f (x) P_{>\ell^{-1}} f (y)
\end{align*}
and the claim follows at once from Bernstein's inequality.
 \end{proof}

\begin{lem}\label{lem:com2} For any $N\geq 0$, we have
\begin{align}
&\norm{ [v_\ell\cdot \na, {P}_{\le \ell^{-1}}]F}_N \lec \ell^{1-N}\norm{\na v}_0\norm{\na F}_0\label{est.com1}\\
&\norm{[v_\ell\cdot\na, {P}_{> \ell^{-1}}] F}_{N}\lec \ell^{1-N}\norm{\na v}_0 \norm{\na F}_0.  \label{est.com2}
\end{align} 
\end{lem}
\begin{proof} First, we observe that  
\begin{align*}
[v_\ell\cdot\na, {P}_{> \ell^{-1}}] F (x)
&= v_\ell\cdot \na  ({P}_{> \ell^{-1}}F -F)+(v_\ell\cdot \na F) -  {P}_{> \ell^{-1}}(v_\ell\cdot \na F)    \\
&= -v_\ell \cdot \na {P}_{\le \ell^{-1}} F + {P}_{\le \ell^{-1}} (v_\ell \cdot \na F) = - [v_\ell\cdot\na, {P}_{\le \ell^{-1}}] F. 
\end{align*}
First of all it suffices to consider the case $N=0$, as the expression that we want to estimate is localized in frequency. 
Next, using the functions $\widecheck{m}_\ell$ introduced in the proof of the previous lemma, we can compute at once
\begin{align*}
|[v_\ell\cdot\na, {P}_{\le \ell^{-1}}] F (x)|
&\le \int_{\R^3} |v_\ell (x) - v_\ell (y)||\widecheck{m}_{\ell}(x-y) ||\na F(y)| dy \\
&\le \norm{|x||\widecheck{m}_{\ell}|}_{L^1} \norm{\na v_\ell}_0\norm{\na F}_0 \lec \ell\norm{\na v}_0\norm{\na F}_0\, .
\end{align*}
Observe now that
\[
\norm{|x||\widecheck{m}_{\ell}|}_{L^1} = 2^{-J} \int |x| \widecheck{m} (x)\, dx \le C \ell 
\]
to conclude the proof.
\end{proof}

\subsection*{Acknowledgments}
The first author has been supported by the National Science Foundation under Grant No. DMS-1946175. The second author has been supported by the NSF under Grant No. DMS-1638352.

\bibliographystyle{abbrv}

\end{document}